\documentclass{amsart}

\usepackage[margin=1in]{geometry}
\usepackage{amsthm,amsmath,amsfonts,amssymb,mathabx}
\usepackage{graphicx}

\usepackage[usenames,dvipsnames,svgnames,table]{xcolor}

\usepackage[colorlinks=true, pdfstartview=FitV, linkcolor=ForestGreen,citecolor=ForestGreen, urlcolor=blue]{hyperref}

\usepackage{hyperref}
\usepackage{esint}

\newtheorem{thm}{Theorem}[section]
\newtheorem{prop}[thm]{Proposition}
\newtheorem{cor}[thm]{Corollary}
\newtheorem{lemma}[thm]{Lemma}
\newtheorem{defn}[thm]{Definition}
\theoremstyle{remark}
\newtheorem{remark}[thm]{Remark}
\numberwithin{equation}{section}

\newcommand{\R}{\mathbb R}
\newcommand{\Sp}{\mathbb S^{d-1}}

\newcommand{\abs}[1]{\left\vert#1\right\vert}
\newcommand{\eps}{\varepsilon}
\newcommand{\grad} {\nabla}
\newcommand{\lap} {\Delta}

\newcommand{\dd} {\; \mathrm{d}}
\newcommand{\dt}{\; \mathrm{d} t}

\newcommand{\dx}{\; \mathrm{d} x}
\newcommand{\dv}{\; \mathrm{d} v}
\newcommand{\un}{\mathbf{1}}

\newcommand{\linenotell}{\mathcal{L}}

\DeclareMathOperator*{\osc}{osc}

\DeclareMathOperator*{\esssup}{esssup}
\DeclareMathOperator*{\essinf}{essinf}
\DeclareMathOperator{\supp}{supp}

\DeclareMathOperator{\Lip}{Lip}

\newcommand{\E}{\mathcal E}
\newcommand{\Ex}{\E}
\newcommand{\Esym}{{\E}^{\text{sym}}}
\newcommand{\Eskew}{{\E}^{\text{skew}}}

\newcommand{\Lv}{L_v}

\newcommand{\angulars}{{2s}}
\newcommand{\cone} {\Xi}
\newcommand{\Radius} {{\bar R}}

\newcommand{\kernels} {\mathcal{K}_0}
\newcommand{\Mp}{\mathcal M^+}
\newcommand{\Mm}{\mathcal M^-}

\title{The Weak Harnack inequality for the Boltzmann equation without cut-off}

\author{Cyril Imbert}
\address[C.~Imbert]{CNRS \& Department of Mathematics and Applications, \'Ecole Normale Sup\'erieure (Paris) \\
45 rue d'Ulm, 75005 Paris, France}
\email{Cyril.Imbert@ens.fr}
\author{Luis Silvestre} 
\thanks{LS is supported in part by NSF grants DMS-1254332 and DMS-1362525.}
\address[L.~Silvestre]{Mathematics
    Department, University of Chicago, Chicago, Illinois 60637, USA}
\email{luis@math.uchicago.edu}

\date{\today}

\begin{document}
\begin{abstract}
We obtain the weak Harnack inequality and H\"older estimates for a large class of kinetic integro-differential equations. We prove that the Boltzmann equation without cut-off can be written in this form and satisfies our assumptions provided that the mass density is bounded away from vacuum and mass, energy and entropy densities are bounded above. As a consequence, we derive a local H\"older estimate and a quantitative
  lower bound for solutions of the (inhomogeneous) Boltzmann equation
  without cut-off. 
\end{abstract}
\maketitle

\setcounter{tocdepth}{1}
\tableofcontents

\section{Introduction}

The main result in this article is a version of the weak Harnack inequality, in
the style of De Giorgi, Nash and Moser, for kinetic
integro-differential equations. As a consequence, we derive local
H\"older estimates and a quantitative lower bound for the
inhomogeneous Boltzmann equation without cut-off.

Our estimates are local in the sense that they only require the
equation to hold in a bounded domain. 

The Boltzmann equation has the form
\[
f_t + v \cdot \grad_x f = Q (f,f)  \quad \text{ for } \; 
t \in (-1,0], \; x \in B_1, \; v \in B_1.
\]
Here, the function $f=f(t,x,v)$ must be defined for $t \in (-1,0]$,
$x \in B_1$ and $v \in \R^d$ in order to make sense of the nonlocal right hand side $Q(f,f)$.

We recall that   Boltzmann's collision operator $Q(f,f)$ is defined as follows
\[ Q (f,f) = \int_{\R^d} \int_{\Sp} (f(v'_*)f(v') - f(v_*)f(v)) B(|v-v_*|,\cos \theta) \dd v_* \dd \sigma\]
where $v_*'$ and $v'$ are given by 
\[ v' = \frac{v+v_*}2 + \frac{|v-v_*|}2 \sigma \quad \text{ and } \quad 
v'_* = \frac{v+v_*}2 - \frac{|v-v_*|}2 \sigma \]
and $\cos \theta$ (and $\sin (\theta/2)$) is \emph{defined} as 
\[ \cos \theta := \frac{v-v_*}{|v-v_*|} \cdot \sigma 
\qquad \left( \text{ and } \quad \sin (\theta/2) := \frac{v'-v}{|v'-v|} \cdot \sigma \right) .\]
We assume that the cross-section $B$ satisfies
\begin{equation}\label{assum:B}
B(r,\cos \theta) = r^\gamma b(\cos \theta) \quad \text{ with} \quad b (\cos \theta) 
\approx|\sin (\theta/2)|^{-(d-1)-2s}
\end{equation}
with $\gamma \in  (-d,1]$ and $s \in (0,1)$.

The equation describes the density of particles at a specific time
$t$, point in space $x$ and with velocity $v$. This model stands at a
mesoscopic level, in between the microscopic description of
interactions between individual particles, and the macroscopic models
of fluid dynamics. 

We define the hydrodynamic quantities
\begin{align*}
\text{(mass density)} \qquad M (t,x) &:= \int f(t,x,v) \dd v, \\
\text{(energy density)} \qquad E (t,x) & := \int f(t,x,v) |v|^2 \dd v, \\
\text{(entropy density)} \qquad H (t,x) & := \int  f \ln f (t,x,v) \dd v. 
\end{align*} 
These are the only quantities associated with a solution $f$ which are meaningful at a macroscopic scale. Under some asymptotic regime, the hydrodynamic quantities in the Boltzmann equation formally converge to solutions of the compressible Euler equation, which is known to develop singularities in finite time (see for example \cite{bardos1991}). Because of this fact, one could speculate that the Boltzmann equation may develop singularities as well. From this point of view, the best regularity result that one would expect is that if the hydrodynamic quantities are under appropriate control, then the solution $f$ will be smooth. In other words, that every singularity of $f$ would be observable at the macroscopic scale.

It is proved in \cite{luis} that when $M(t,x)$, $E(t,x)$ and $H(t,x)$ are uniformly bounded above, and in addition $M(t,x)$ is bounded below by a positive constant, 
then the solution $f$ satisfies the $L^\infty$ a priori estimate
depending on those bounds only. The result in this paper goes a
step further by proving a H\"older modulus of continuity, in all
variables, under the same assumptions.

\begin{thm}[H\"older continuity]\label{thm:holder-boltzmann}
  Assume $s \in (0,1)$, $\gamma \in (-d,1]$, $\gamma + 2s \le 2$ and let $f$ be a non-negative solution of
  the Boltzmann equation for all $t \in (-1,0]$, $x \in B_1$ and
  $v \in B_1$. Assume that $f$ is essentially bounded in
  $(-1,0] \times B_1 \times \R^d$ and there are positive constants
  $M_0$, $M_1$, $E_0$ such that for all $(t,x)$ we have
  $M_1 \leq M(t,x) \leq M_0$ and $E(t,x) \leq E_0$ for all
  $(t,x) \in (-1,0] \times B_1$, then $f$ is H\"older continuous in
  $(-1/2,0] \times B_{1/2} \times B_{1/2}$ with
  \[ \|f\|_{C^\alpha((-1/2,0] \times B_{1/2} \times B_{1/2})} \leq
  C \]
  where $C>0$ and $\alpha \in (0,1)$ are constants
    depending on dimension, the $L^\infty$ bound of $f$, $M_0$, $M_1$
    and $E_0$.
\end{thm}
\begin{remark}
  Theorem~\ref{thm:holder-boltzmann} also holds true in any cylinder
  $Q \subset \R \times \R^d \times \R^d$.  In this case, constants $C$
  and $\gamma$ also depends on the center of the cylinder and its
  radius.
\end{remark}

Note that the value of the entropy $H(t,x)$ is bounded above by some
constant $H_0$ depending only on $M_0$, $E_0$ and $\|f\|_{L^\infty}$
so we do not need to include the hypothesis $H(t,x) \leq H_0$ in
Theorem \ref{thm:holder-boltzmann} . Recall also that
$\|f\|_{L^\infty}$ is bounded above for $t>0$ in terms of $M_0$,
$M_1$, $E_0$ and $H_0$, according to the result in \cite{luis},
provided that $\gamma + 2s > 0$. So, at least in this range of values
of $\gamma$, the H\"older modulus of continuity depends on the values
of $M_0$, $M_1$, $E_0$ and $H_0$ only.

The best regularity results previously available for the inhomogeneous
Boltzmann equation without cut-off give us $C^\infty$ regularity
depending on the assumption that the solution has infinite moments and
belongs to the space $H^5$ with respect to all variables ($v$, $x$ and
$t$) \cite{amuxyCRAS2009}, \cite{amuxy2012}, \cite{chen-he-2012}. Of
course this is a much more stringent assumption than what we need for
our Theorem \ref{thm:holder-boltzmann} to hold. We make further
comments about these and other related results in Section
\ref{subsec:relatedresults}.

We also obtain a quantitative lower bound for the solution $f$.

\begin{thm}[Lower bound] \label{thm:lowerbound} Let $f$ be a
  non-negative supersolution of the Boltzmann equation in
  $[0,T] \times B_R \times B_R$. Under the same assumptions on
  $\gamma$, $s$ and $f$ as in Theorem \ref{thm:holder-boltzmann}, we
  have the lower bound
  \[ \inf_{[T/2,T] \times B_{R/2} \times B_{R/2}} f \geq c(R). \]
  The constant $c(R)$ depends on $T$, $R$, $\gamma$, $s$, $d$, $M_0$,
  $M_1$, $E_0$, and $\|f\|_{L^\infty}$.
\end{thm}

It has been a longstanding issue to find appropriate lower bounds for
the solutions of the Boltzmann equation. The best result available is
perhaps from the work of Mouhot \cite{mouhot2005}. He obtains an
explicit exponentially decaying lower bound for the Boltzmann equation
without cut-off. He makes strong a priori regularity assumptions on the
solution $f$, in addition to the assumptions that we make in this
paper. We do not provide an explicit formula for $c(R)$. Its precise decay as $R \to \infty$ will be the subject of future work.


\begin{remark}
  If $\gamma + 2s >2$, similar results can be obtained by further
  assuming that the $(\gamma+2s)$-momentum of the function $f$ is
  finite at every point $(t,x)$.
\end{remark}

\subsection{A linear kinetic integro-differential equation}

The main result of this paper concerns a general kinetic
integro-differential equation. The results for the Boltzmann equation
described above follow as corollaries. We study an equation of the
form
\begin{equation} \label{e:main}
 f_t + v \cdot \grad_x f = \Lv f + h
\end{equation}
for $t \in (-1,0]$,  $x \in B_1$ and $v \in B_1$,
where $\Lv f$ is a linear integro-differential operator in the velocity variable
of the following form
\[ 
\Lv f (t,x,v) = PV \int_{\R^d} (f(t,x,v') - f(t,x,v)) K(t,x,v,v') \dd v' 
\]
for a  locally bounded function $h$
and a measurable kernel
$K : [-1,0] \times B_1 \times B_{\Radius} \times \R^d
\to [0,+\infty)$
satisfying appropriate assumptions that we describe below.

For every value of $t$ and $x$, the kernel $K(t,x,v,w)$ is a
non-negative function of $v$ and $w$. We assume that the following
conditions hold for every value of $t$ and $x$ (we omit $t$ and $x$
dependence to clean up the notation).

Let us fix a $\Radius \geq 1$. We will make assumptions on the kernel
$K(v,v')$ for $v \in B_{\Radius}$. We need to pick $\Radius$ slightly
larger than one for technical reasons that will be apparent in Section
\ref{sec:reduction}.

Our first assumption is a coercivity condition on $\Lv$. We assume
that there exists $\lambda>0$ and $\Lambda>0$ such that, for any
function $f:\R^d \to \R$ supported in $B_{\Radius}$,
\begin{equation} \label{e:Klower}
 \lambda 
\iint_{\R^d \times \R^d} \frac{|f(v)-f(v')|^2}{|v-v'|^{d+2s}} \dd v \dd v' 
\leq - \int_{\R^d} \Lv f(v) \,  f(v)  \dd v  + \Lambda \|f\|_{L^2(\R^d)}^2.
\end{equation}
This coercivity condition is well known to hold for the Boltzmann
equation when the function $f$ has bounded mass, energy and entropy
density, and its mass is also away from vacuum (see
\cite{lions1998,villani1999,advw} and the discussion below). The
proofs in the literature are based on Fourier analysis. We provide a
proof in the appendix which follows by a direct geometric computation
in physical variables.

In the case $s < 1/2$, we also make the following nondegeneracy assumption.
\begin{equation} \label{e:Knondegeneracy} 
\inf_{|e|=1} \int_{B_r(v)}  ((v'-v)\cdot e)_+^2 K(v,v') \dd v' \geq \lambda r^{2-2s} \qquad
  \text{ for every value of $v \in B_\Radius$}.
\end{equation}
Here, when we write $(w\cdot e)_+^2$, we mean $((w \cdot e)_+)^2 = \max(w \cdot e,0)^2$.

The coercivity condition would be obviously true if $K$ is symmetric
(i.e. $K(v,v') = K(v',v)$) and $K(v,v') \geq \lambda |v-v'|^{-d-2s}$.
These assumptions are not satisfied by the Boltzmann kernel a priori. 

For some kernels (not necessarily coming from the Boltzmann equation) it might be difficult to
check whether the coercivity condition \eqref{e:Klower} holds. The nondegeneracy assumption \eqref{e:Knondegeneracy} is usually very easy to check in explicit examples of kernels $K$. We do not know of any example of a kernel which satisfies \eqref{e:Knondegeneracy} but not \eqref{e:Klower}. It is natural to conjecture this implication (modulo adjusting $\lambda$ by a fixed factor).

The second assumption is a weak upper bound on the kernel $K$. 
\begin{equation} 
\label{e:Kabove}
 \begin{cases}
(i) & \int_{\R^d \setminus B_r(v)} K (v,v') \dd v' \leq \Lambda r^{-2s}  \text{ for any } 
r>0 \text{ and } v \in B_{\Radius} \\[1ex]
(ii) & \int_{B_{\Radius} \setminus B_r(v')} K (v,v') \dd v \leq \Lambda r^{-2s}  \text{ for any } 
r>0 \text{ and } v' \in B_{\Radius}.
\end{cases}
\end{equation}
Note that if $K(v,v') \lesssim |v-v'|^{-d-2s}$, then the assumption
\eqref{e:Kabove} holds. Our assumption only concerns average values of
$K$ on the complementary set of balls. Therefore, a kernel containing
a singular part is allowed. We will see that the Boltzmann kernel
satisfies \eqref{e:Kabove} even though
$K(v,v') \lesssim |v-v'|^{-d-2s}$ may not hold a priori.

Note that both inequalities in \eqref{e:Kabove} would be the same if
$K$ were symmetric.  But we do \emph{not} assume symmetry of the kernel. That is
$K(v,v') \neq K(v',v)$ in general. The symmetry assumption is very
common for integro-differential equations because
it represents the fact that the equation is in \emph{divergence
  form}. It is equivalent to the operator $\Lv$ being self adjoint. We
explain this concept in Subsection \ref{sub:review-ide}.

The following assumptions provide a mild control on the
anti-symmetric part of the kernel.

We assume that 
\begin{equation} \label{e:Kcancellation0} 
\forall v\in B_{7\Radius/8}, \qquad 
\left| PV \int_{B_{\Radius/8}(v)} \left( K(v,v') - K (v',v)\right) \dd v' \right| \leq \Lambda .
\end{equation}

Moreover, if $s \geq 1/2$, we need to assume the following extra cancellation.
\begin{equation} \label{e:Kcancellation1} 
\forall r \in (0,\Radius/8], \forall v\in B_{7\Radius/8}, \qquad 
\left| PV \int_{B_r(v)}
    (v-v') \left( K(v,v') - K (v',v)\right) \dd v' \right| \leq  \Lambda (1+r^{1-2s}).
\end{equation}
When $K$ is symmetric, the left hand
sides in \eqref{e:Kcancellation0} and \eqref{e:Kcancellation1} are
identically zero and therefore the assumptions trivially hold.

When $s > 1/2$, if the assumption \eqref{e:Kcancellation1} holds for
$r=\bar R$ and in addition \eqref{e:Kabove} holds, then we observe
that \eqref{e:Kcancellation1} automatically holds for all
$r \in (0,\bar R]$. The requirement that the inequality
\eqref{e:Kcancellation1} holds for all $r \in (0,\bar R]$, as opposed
to only $r=\bar R$, only makes a difference for the case $s=1/2$. We
discuss the scaling properties of our assumptions in subsections
\ref{sub:com} and \ref{sub:invariant}.

  When we apply our results to the Boltzmann equation, the kernel $K$
depends on the solution $f$ and is determined by the
  formula
\[ \int_{\R^d} \int_{\partial B_1} f(v_\ast') ( g(v') - g(v) )
B(|v-v_\ast|, \theta) \dd \sigma \dd v_\ast = \int_{\R^d} (g(v') -
g(v)) K_f (v,v') \dd v'.\]
In this way,
\[ Q(f,g) = \int (g(v') - g(v)) K_f(v,v') \dd v' + (\text{lower order
  terms}).\]
The constant $\lambda$ in the assumption \eqref{e:Klower} depend only
on the mass, energy and entropy densities of $f$. The constant
$\Lambda$ in \eqref{e:Kabove}, \eqref{e:Kcancellation0}
and \eqref{e:Kcancellation1}, depends only on the mass and energy
density of $f$ when $\gamma \in [0,1]$. It depends on further
integrability properties of $f$ when $\gamma < 0$ (they are bounded in
terms of $\|f\|_{L^\infty}$ for example). All these assumptions will
be verified in Section \ref{sec:kernel}.

\subsection{Main results}
\label{subsec:main}

The notion of weak solution will be made precise by the end of Section \ref{sec:reduction}.
\begin{thm}[H\"older continuity]\label{thm:holder}
  Assume the kernel is non-negative and there exist $\lambda >0$ and
  $\Lambda >0$ such that \eqref{e:Klower}, \eqref{e:Kabove} and
  \eqref{e:Kcancellation0} hold true with $\Radius=2$. If $s \ge 1/2$,
  we also assume \eqref{e:Kcancellation1}; if $s < 1/2$, we also
  assume \eqref{e:Knondegeneracy}.  Let $f$ be a solution of
  \eqref{e:main} for all $t \in (-1,0]$, $x \in B_1$ and $v \in B_1$.
  Assume that $f$ is essentially bounded in
  $(-1,0] \times B_1 \times \R^d$.  Then $f$ is H\"older continuous in
  $(-1/2,0] \times B_{1/2} \times B_{1/2}$ with
  \[ \|f\|_{C^\gamma((-1/2,0] \times B_{1/2} \times B_{1/2})} \leq C
  \left( \|f\|_{L^\infty((-1,0] \times B_{1} \times \R^d )} +
    \|h\|_{L^\infty((-1,0] \times B_{1} \times B_{1})} \right) \]
  where $C>0$ and $\gamma \in (0,1)$ are constants only depending on
  dimension, $\lambda$ and $\Lambda$.
\end{thm}
This theorem is in fact derived from the following estimate. 
\begin{thm}[Weak Harnack inequality]\label{thm:whi} 
  There are constants $r_0$, $R_1>1$, $\eps$ and $C$ so that the
  following proposition holds.  Assume the kernel is non-negative and
  there exist $\lambda >0$ and $\Lambda >0$ such that
  \eqref{e:Klower}, \eqref{e:Kabove} and \eqref{e:Kcancellation0}
  holds true with $\Radius=2R_1$. If $s \ge 1/2$, we also assume
  \eqref{e:Kcancellation1}; if $s < 1/2$, we also assume
  \eqref{e:Knondegeneracy}.  Assume that $f$ is a non-negative
  {supersolution} of \eqref{e:main} in
  $(-1,0] \times B_{R_1^{1+2s}} \times B_{R_1}$. Then
  \[ \left( \int_{Q^-} f^\eps (t,x,v) \dv \dx \dt \right)^{1/\eps} 
\leq C \left( \inf_{Q^+} f + \|h\|_{L^\infty ((-1,0]\times B_1 \times B_1)} \right) \] 
where 
\[ Q^+ = (-r_0^{2s},0] \times B_{r_0^{1+2s}} \times B_{r_0} \quad
\text{ and } \quad Q^- = (-1,-1+r_0^{2s}] \times B_{r_0^{1+2s}} \times
B_{r_0} \]
(see Figure~\ref{fig:whi}) and the constants $C>0$, $\eps >0$, only
depend on dimension, $s$, $\lambda$ and $\Lambda$. The constants $r_0$
and $R_1$ depend on dimension and $s$ only (not on $\lambda$ and
$\Lambda$).
\end{thm}
\begin{figure}
\setlength{\unitlength}{1in} 
\begin{picture}(2.903 ,1.000)
\put(0,0){\includegraphics[height=1.000in]{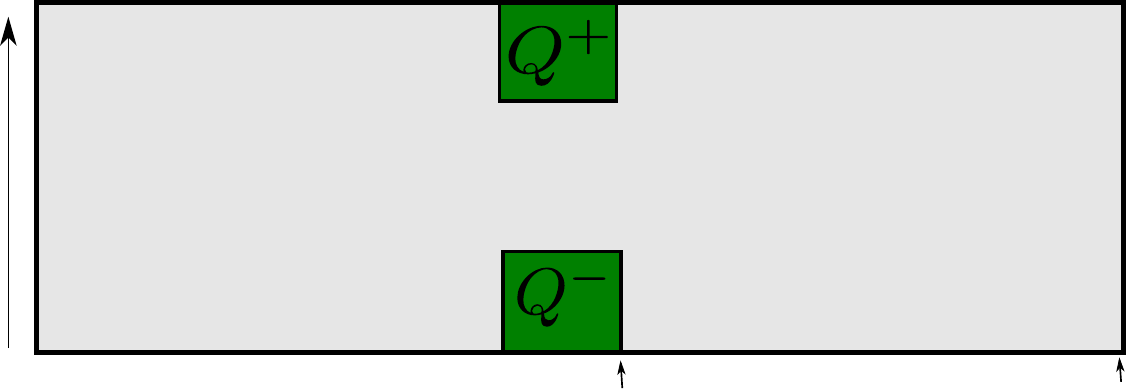}}
\put(1.57,-0.083){$r_0$}
\put(2.8,-0.102){$R_1$}
\put(-0.6,0.479){$t \in (-1,0]$}
\end{picture}
\caption{The geometric setting of the weak Harnack inequality}
\label{fig:whi}
\end{figure}
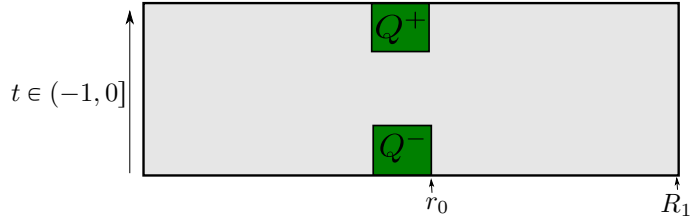

\subsection{Comments on the results and related works}

\label{subsec:relatedresults}

\subsubsection{Difficulties related to this problem}

This subsection is our attempt to explain and compare the main
challenges that we faced proving the main results in this paper, and
the new ideas that were introduced.

We start by reviewing some recent developments about parabolic kinetic
equations in divergence form, with rough coefficients. In some sense,
our main theorems are an integro-differential counterpart of these previous
results. The equations have the form
\[ f_t + v \cdot \nabla_x f = \frac{\partial}{\partial v_i} \left(
  a_{ij} \frac{\partial}{\partial v_j} f \right).\]
The diffusion coefficient $a_{ij} = a_{ij}(t,x,v)$ is assumed to be
uniformly elliptic. No regularity assumption should be made in
$a_{ij}$, otherwise the equation may fit into the more classical
hypoelliptic theory, and would not imply such interesting results for
the Landau equation. Pascucci and Polidoro \cite{pp} obtained the
local $L^\infty$ estimate for this equation using Moser's
method. Continuing in that direction, Wang and Zhang obtained H\"older
estimates in \cite{zhang2008}, \cite{wangzhang2009} and
\cite{wangzhang2011}. Their proof is quite involved. A highly
nontrivial step is to obtain an appropriate formulation of a Poincar\'e
inequality adapted to the Lie group action related to the equation.

A simplified proof, following the method of De Giorgi, was recently
obtained by Golse, Imbert, Mouhot and Vasseur \cite{gimv}. In this
paper a version De Giorgi's isoperimetric inequality is obtained by a
compactness argument. We use that idea for the case $s \in
[1/2,1)$.
This general method is not applicable to the case $s < 1/2$, since it
uses crucially that the characteristic function of a nontrivial set
can never be in $H^s$. We do not use velocity averaging lemmas like in
\cite{gimv} anywhere in this paper. Instead, we take advantage of more
elementary properties of the fractional Kolmogorov equation.

The first step in the proof of De Giorgi, Nash and Moser, which
consists of a local $L^\infty$ estimate, needs to be formulated
appropriately to hold for integro-differential equations with
degenerate kernels. 
Our proof in Section \ref{sec:fdl}
follows a properly adapted version of De Giorgi's iteration. We do not
use either averaging lemmas, or hypoelliptic estimates for the $x$
variable (like in \cite{gimv} or \cite{pp}). Instead, we iterate an
improvement of integrability obtained directly from the fundamental
solution to the fractional Kolmogorov equation.

In the second part of the proof of the theorem of De Giorgi, Nash and
Moser we take different strategies depending on whether
$s \in (0,1/2)$ or $s \in [1/2,1)$. In the first case, we construct a barrier function to propagate lower bounds as in the method by Krylov and Safonov for nondivergence equations. 
When $s \in [1/2,1)$, the
proof is based on a measure estimate of intermediate sets (as in De
Giorgi's original work) obtained by compactness (as in \cite{gimv}),
but using a more direct approach based on the fractional
Kolmogorov equation instead of hypoelliptic estimates and averaging
lemmas. We could not find a single method that works for the full range $s \in (0,1)$ for general integro-differential equations. However, in the case of the Boltzmann equation, the method used for the range $s \in (0,1/2)$ actually works for the full range, as we explain below.

The kernel $K_f$, from the Boltzmann equation, satisfies the extra symmetry condition $K(v,v+w) = K(v,v-w)$ which we do not use in this paper. We chose not to take advantage of this condition in order to have the most natural result for general integro-differential equations.  Using this assumption would allow us to simplify some of the proofs. Most importantly, the barriers of Section \ref{sec:barriers} would hold for the full range $s \in (0,1)$ and therefore the results from Section \ref{sec:intermediate} would be unnecessary. Moreover, the proof of Lemma \ref{lem:gain-int} could be done more easily using a similar function $g$ as in the proof of Lemma \ref{lem:caccio}. The commutator estimates of Lemmas \ref{l:commutator-s<1/2} and \ref{l:commutator-s>1/2} would not be necessary anywhere.

One of the main ideas in the work of Caffarelli, Chan and Vasseur
\cite{CCV} about parabolic integro-differential equations (not
kinetic) is how they formulate De Giorgi's isoperimetric lemma in the
integro-differential setting. Their original method is purely
nonlocal. It does not work for second order equations. It uses
crucially that $\E(g_+,g_-) \gtrsim \|g_+\|_{L^1} \|g_-\|_{L^1}$,
where $\E$ is a bilinear form like the one we define in Section
\ref{sec:bilinear}. In our context, this is not true for two
reasons. First, because we have the additional variable $x$ that
plays no explicit role in the integral diffusion and is not
\emph{seen} by the bilinear form $\E$. Secondly, because the
assumptions that we make in the kernels are too mild for this
condition to hold even in the space homogeneous case. In \cite{CCV},
they assume that the kernel $K$ is symmetric and
$K(v,v') \gtrsim \lambda |v-v'|^{-d-2s}$ for every value of $v$ and
$v'$.

Our equation \eqref{e:main} involves three different variables: $t$,
$x$ and $v$. It reduces to a more standard parabolic
integro-differential equation when $f$ is constant in $x$. The
diffusion takes place with respect to the variable $v$ only. The
equation includes the kinetic transport term $v \cdot \nabla_x f$,
which somehow transfers the regularization effect from the $v$
variable to the $x$ variable. The variable $x$ has to be dealt with
differently to the $t$ and $v$ variables. For example it has a
different scaling and it is affected by translations of the function
with respect to the $v$ variable. One major difficulty that it brings
is in the proof of the ink-spots theorem. The original ink-spots
covering by Krylov and Safonov was for non-kinetic parabolic equations
without the extra variable $x$. Including this extra variable changes
the geometry. The natural parabolic cylinders, which are invariant by
the Lie group acting on the equation, are oblique in the variable
$x$. With this geometry, there is no chance to apply a
Calder\'on-Zygmund decomposition like in \cite{is} because we cannot
tile the space with slanted cylinders with varying slopes. We need a
custom made version of the ink-spots covering theorem, which is
developed in Section \ref{sec:ink-spot}. See that section for further
explanation on the difficulties and ideas involved in this covering
result.

When we apply our main results to the Boltzmann equation in Theorems
\ref{thm:lowerbound} and \ref{thm:holder-boltzmann}, we only want to
assume a priori some minimal physically relevant information on
$f$. We assume a control, for all $t$ and $x$, of the mass, energy and
entropy densities. Under these assumptions, there is very little one
can say about the Boltzmann collision kernel $K_f$. We are forced to
work with very general, non-symmetric, and possibly singular
kernels. This paper would be much simpler if we made a convenience
assumption like $K(v,v') = K(v',v) \approx |v-v'|^{-d-2s}$, but it
would not suffice to apply the result to the Boltzmann equation. It is
not a priori obvious what assumptions the Boltzmann kernel will
satisfy. In Section \ref{sec:boltzmann}, we prove that $K_f$ satisfies
\eqref{e:Klower}, \eqref{e:Kabove}, \eqref{e:Kcancellation0} and
\eqref{e:Kcancellation1}. Our assumptions \eqref{e:Kcancellation0} and
\eqref{e:Kcancellation1} allow us to consider non-symmetric kernels
whose anti-symmetric part is as singular as the symmetric one in
absolute value, but contains some cancellation. Up to the authors'
knowledge, this is the first time such a condition appears in the
literature of integro-differential equations.

The estimate for the bilinear form given in Theorem \ref{t:upperbound}
is interesting in itself and new.  It tells us that the bilinear form
$\langle \Lv f, g\rangle$ is bounded in $H^s \times H^s$ assuming the
very mild, and easy to check, conditions on the kernel $K$ given in
\eqref{e:Kabove}, \eqref{e:Kcancellation0} and
\eqref{e:Kcancellation1}.  Such an estimate is reminiscent of some
others proved specifically for the Boltzmann equation, see for
instance \cite{amuxy2010}, \cite{alexandre2009review},
\cite{morimoto2009regularity}, \cite{chen-he-2012}.  Here, the
estimate is proved for a very general bilinear form associated with a
non-symmetric integro-differential operator.  Note that in previous
works in integro-differential equations, the upper bound of Theorem
\ref{t:upperbound} was included as an assumption together with
\eqref{e:Kabove} and symmetry (see \cite{kassmann2013regularity}).

\subsubsection{Boltzmann without cut-off}
The main results of this paper apply to the Boltzmann equation without
cut-off in the inhomogeneous setting.

In the case of moderately soft potentials, which corresponds to
$\gamma+2s > -2$, an a priori estimate in $L^\infty$ is given in
\cite{luis}. In that case, we obtain a H\"older modulus of continuity
depending on the bounds on $M(t,x)$, $E(t,x)$ and $H(t,x)$ only. For
very soft potentials, Theorem \ref{thm:holder-boltzmann} gives us a
H\"older modulus of continuity provided that we know a priori that $f$
is bounded. Note that our estimates do not depend on any further regularity assumption on the initial data.

 Since Carlo Cercignani in 1969, it is believed that the Boltzmann collision operator without cut-off has a regularizing effect. Some similarities with the fractional Laplacian operator in the velocity variable have been observed in the form of coercivity estimates. This is the first time that ideas originating in the work of De Giorgi and Nash for parabolic equations are applied in the context of the Boltzmann equation.

The first results for the Boltzmann equations without cut-off that
indicate a regularization effect appear in the study of the entropy
dissipation.  A lower bound for the entropy dissipation with respect
to a fractional Sobolev norm is first obtained \cite{lions1998} and
improved in \cite{villani1999}.  The optimal space $H^s$ is finally
obtained in \cite{advw}. We can also deduce a coercivity estimate from
the proof in this paper. The coercivity estimate, which we mention in
Proposition \ref{p:lowerbound}, essentially says that the Boltzmann
collision operator satisfies the assumption \eqref{e:Klower}. It plays
an essential role in most of the works concerning the regularization
effect of the Boltzmann equation without cut-off. The proof of the
coercivity estimate in \cite{advw} is done using Fourier analysis
after reducing the problem to the case of Maxwellian molecules
($\gamma=0$).  There is a simplified proof, also using Fourier
analysis and in particular the Littlewood-Paley decomposition, in
\cite{aes2005} and \cite{aes2009}. These proofs are considerably
easier in the Maxwellian case ($\gamma=0$), because they use Bobylev's
formula. We give a new alternative proof in the Appendix
\ref{sec:appendix} based on the geometric understanding of the
Boltzmann kernel. All computations are done in physical variables.
Our proof works in the same way for any value of $\gamma$. It
transparently gives us an estimate with respect to the same
anisotropic weighted Sobolev spaces as in \cite{gressman2011global}.

The coercivity estimate implies some gain of regularity for the
Boltzmann equation without cut-off. In the space homogeneous case,
iterating this gain of regularity, it is known that solutions belong
to the Schwartz class for all positive times. This result holds under
rather general cross section assumptions, including essentially hard
and moderately soft potentials in the non-cut-off case. See
\cite{dw2004}, \cite{aes2005}, \cite{aes2009}, \cite{huo2008},
\cite{morimoto2009regularity} and \cite{chen2011smoothing}.

For the spacially inhomogeneous case without cut-off, one can also
obtain some regularization effect combining the coercivity with
hypoelliptic estimates. Iterating such estimates leads to the
$C^\infty$ regularity of solutions. However, it is necessary to impose
significant conditional regularity in order to start the iteration. The
best regularity results available require the assumptions that
$\langle v \rangle^k f(t,x,v)$ belongs to $H^5([0,T],\R^3,\R^3)$ for
all values of $k \in \mathbb N$, and in addition the mass density is
assumed to be bounded below. Under these assumptions, they prove that
$f$ belongs to the Schwartz class for positive time in
\cite{amuxyCRAS2009}, \cite{amuxy2010}, \cite{chen-he-2012}.

It may be interesting to compare the current state of the regularity
results for the Boltzmann equation with the classical development of
nonlinear elliptic equations. Hilbert's 19\textsuperscript{th} problem
consisted in the regularity of minimizers of smooth convex functionals
in $H^1$ (see \cite{wiki:hilbert19}). These minimizers solve a
nonlinear elliptic equation in divergence form. From the beginning of
the century (starting by the work of Bernstein \cite{bernstein1904}),
people proved that solutions were analytic provided that some conditional
regularity assumption was satisfied. The assumptions were
progressively improved through the years. By iterating the Schauder
estimates, it was possible to prove that solutions were analytic
starting from a $C^{1,\alpha}$ estimate. However, variational
techniques only provided a weak solution in $H^1$. It was a long
standing problem to bridge that gap, and it was finally achieved
independently by De Giorgi \cite{e1957sulla} and Nash
\cite{nash1958continuity}. Our result in this paper plays the role, in
the context of the inhomogeneous non-cut-off Botlzmann equation, of
the results of De Giorgi and Nash for elliptic and parabolic
equations.  Unfortunately, there is still a gap between what we prove
($C^\alpha$ regularity) and what is necessary to iteratively obtain
$C^\infty$ regularity of the solution by current methods ($H^5$
regularity plus infinite moments). So, more work is necessary.

In \cite{luis}, results from general integro-differential equations
are applied to the Boltzmann equation. There is an $L^\infty$
estimate, a H\"older estimate and a lower bound. However, the last two
apply only to the space homogeneous case. The results in this paper
are proved with different techniques compared to \cite{luis}. In this
work, we develop a result in the flavor of De Giorgi, Nash and Moser
theorem for equations in divergence form. The results in \cite{luis}
use the methods from \cite{ss} which are in the flavor of
Krylov-Safonov theory for equations in nondivergence form. The
coercivity estimate plays no role in \cite{luis}, and it certainly
does here. Our result in Theorem \ref{thm:holder} complements the
$L^\infty$ estimate from \cite{luis}.

In \cite{amuxy2012}, the authors prove that if the inital data is sufficiently nice, the Boltzmann equation admits a unique smooth solution locally in time. For small perturbations around a Maxwellian, the equation is known to have global
smooth solutions \cite{gs2010}, \cite{gressman2011global},
\cite{amuxy-CRAS-2010}, \cite{amuxy2011}. 
As far as existence of weak solutions is concerned, Alexandre and
Villani prove in \cite{av2002} the existence of a certain type of renormalized solution. Neither the uniqueness nor the regularity of these
solutions is well understood. They prove that the family of solutions is compact using the entropy dissipation estimate.

The study of the regularity of solutions is relevant for most aspects
of the qualitative analysis of the Boltzmann equation without cutoff. For example, Desvillettes and Villani prove in \cite{dv2005} that the
solutions converge to equilibrium, at a specific rate, provided that the solution remains
smooth.

\medskip
We consider this paper to be an important step towards a longer term
goal to prove the following conjecture. We believe that if $f$ is a
solution to the Boltzmann equation with $\gamma+2s \in (0,2]$ and such
that $0 < M_1 \leq M(t,x) \leq M_0$, $E(t,x) \leq E_0$ and
$H(t,x) \leq H_0$, then $f$ should be $C^\infty$ for positive time.

\medskip

It is not at all clear whether the assumption $\gamma+2s > 0$ is
necessary to obtain regularity. However, the $L^\infty$ estimate for
very soft potentials is out of reach by current methods without
further assumptions. This is also the case for the space homogeneous
Boltzmann equation.

It would be possible to study the precise
behavior of the constants $\lambda$ and $\Lambda$ for which
\eqref{e:Klower}, \eqref{e:Kabove}, \eqref{e:Kcancellation0} and
\eqref{e:Kcancellation1} hold and obtain a global weighted $C^\alpha$
estimate using a scaling argument as in Remark \ref{r:rate}. However,
this estimate also depends on the $L^\infty$ norm of $f$. It is to be
expected that the solution $f$ should decay exponentially for large
velocities, in addition to the $L^\infty$ bound given in
\cite{luis}. See \cite{irene} for a result in that direction in the
space homogeneous case. A better decay in $f$ for large velocities
would imply a better $C^\alpha$ estimate for large velocities. Because
of that, we postpone the analysis of large velocities to future work
when the decay of $f$ is better understood. The local result provided
in Theorem \ref{thm:holder} provides the right tool to study the
$C^\alpha$ estimate for large velocities in terms of the decay of $f$.

\subsubsection{Regularity theory for integro-differential equations}

\label{sub:review-ide}

The study of H\"older estimates and the Harnack inequality for
integro-differential equations of the form
\[ f_t(t,v) = \int_{\R^d} (f(t,v') - f(t,v)) K(t,v,v') \dd v' \]
is a very active area of current research. It developed originally
motivated by problems in probability, with applications to mathematical finance \cite{tankov2003financial} and physics \cite{metzler2004restaurant}. The main technical novelty
of this work is our study of a kinetic equation with this kind of
diffusion. Our equation has the extra variable $x$, and the transport
term $v \cdot \nabla_x f$, without any explicit diffusion in
$x$. Previous H\"older estimates for integro-differential equations
may be applied to the Boltzmann equation, at most, in the space
homogeneous case only. Yet, even in the space homogeneous case, the
results in this paper present novelties. The assumptions we make on
the kernel \eqref{e:Klower}, \eqref{e:Kabove},
\eqref{e:Kcancellation0} and \eqref{e:Kcancellation1} are more general
than in previous works about integro-differential equations. Because
of that, our main results in Theorems \ref{thm:holder} and
\ref{thm:whi} are new even in the space homogeneous case. In this
subsection, we review and compare the literature about
integro-diferential diffusions. We stress that all previous results apply to the space homogeneous case only.

The interest in H\"older estimates and Harnack inequalities started
from the study of regularization properties of classical parabolic
equations of second order. For equations in divergence form (like
$f_t = \partial_i a_{ij}(t,v) \partial_j f$), the estimates were
originally obtained independently by De Giorgi \cite{e1957sulla} and
Nash \cite{nash1958continuity}, and later reproved by Moser
\cite{moser1964harnack}. For equations in nondivergence form (like
$f_t = a_{ij}(t,v) \partial_{ij} f$) the result was obtain much later
by Krylov and Safonov \cite{ks}. The techniques used for equations in
divergence or nondivergence form are very different. In the former
case, the equation's structure is amenable to variational methods, and
energy estimates in Sobolev spaces. In the latter case, tools like the
Alexandroff estimate and explicit barrier functions are used for the
proofs. Both types of results, with their corresponding approaches,
have their counterparts for integro-differential equations. In this
paper, we use the variational structure of the equation and work with
localized energy estimates. These are ideas for equations in
divergence form. However, we use some ideas that originated in the
study of equations in nondivergence form, like the ink-spots theorem
and barrier functions. Below, we review other results for
integro-differential equations following each approach.

A second order operator in divergence form
$f \mapsto \partial_i (a_{ij}(t,v) \partial_j f)$ is characterized by the
fact that it is self-adjoint in $L^2$. For integro-differential
operators, this is reflected in a symmetry condition for the kernel:
$K(v,v') = K(v',v)$. A second order operator in nondivergence form
$f \mapsto a_{ij}(t,v) \partial_{ij} f$ has the convenient property
that it returns a bounded function when evaluated in a smooth function
$f$. For integro-differential operators, this is reflected in a
different symmetry condition $K(v,v+w) = K(v,v-w)$. The Boltzmann
collision kernel has the symmetry condition that corresponds to
equations in nondivergence form. This structure is exploited in
\cite{luis} to obtain H\"older estimates in the space homogeneous
case, and $L^\infty$ estimates for the full equation. In this paper we
apply techniques for equations in divergence form. We include
assumptions \eqref{e:Kcancellation0} and \eqref{e:Kcancellation1}
which measure how much the kernel $K$ is allowed to depart from being
symmetric (as in $K(v,v') = K(v',v)$).

The Harnack inequality and H\"older estimates for integro-differential
equations in \emph{divergence} form has a long history with several
major contributions. Some results in this direction are
\cite{komatsu1988}, \cite{barlow2009}, \cite{kassmann2009},
\cite{chen2011}, \cite{CCV}, \cite{fk}, \cite{kassmann2013regularity} and
\cite{dyda2015regularity}. There is a small survey on the subject in
\cite{kassmann2013regularity}. In these papers the kernel $K$ satisfies the symmetry
condition $K(v,v') = K(v',v)$ plus some ellipticity assumptions. It is
perhaps clear that there is some room in the methods for a lower order
asymetric part in $K$. Our assumptions \eqref{e:Kcancellation0} and
\eqref{e:Kcancellation1} allow us to consider a non-symmetric kernel
$K$ whose asymetric part is as singular as the symmetric part. We
require a control of the asymetric part in terms of cancellation
conditions, which is new.

A natural ellipticity condition on the kernel is to assume that it is
comparable with the fractional Laplacian. The classical assumption
would be $K(v,v') \approx |v-v'|^{-d-2s}$ for every value of $v$ and
$v'$. This assumption is made in \cite{komatsu1988},
\cite{barlow2009}, \cite{kassmann2009}, \cite{chen2011} and
\cite{CCV}. The results were extended to a much more general class of
kernels in \cite{fk}, \cite{kassmann2013regularity} and \cite{dyda2015regularity}. The
assumptions there are essentially equivalent to our assumptions
\eqref{e:Klower} (the lower bound on the bilinear form in the
symmetric case) and \eqref{e:Kabove} (the upper bound on the kernel),
plus the result of our Lemma~\ref{lem:upper-sym} (the upper bound for
the bilinear form). It is a new contribution of this paper that
Lemma~\ref{lem:upper-sym} follows from \eqref{e:Kabove}. We also prove
in Theorem \ref{t:upperbound} that the integro-differential operator
$\Lv$ is bounded in $H^s$ to $H^{-s}$ for a non-symmetric kernel
satisfying \eqref{e:Kabove}, \eqref{e:Kcancellation0} and
\eqref{e:Kcancellation1}. The proof is significantly more complicated
in the non-symmetric case.

The study of integro-differential equations in \emph{nondivergence}
form followed a parallel path using different tools. These are the
H\"older estimates and the Harnack inequality for kernels satisfying
the other symmetry condition: $K(v,v+w) = K(v,v-w)$. There are also
many important results in this direction including
\cite{basslevin2002}, \cite{song2004}, \cite{basskassmann2005},
\cite{basskassmann2005tams}, \cite{silvestre2006},
\cite{caffarellisilvestre2009}, \cite{silvestre2011}, \cite{bci}, \cite{bcci}, 
\cite{bjorlandcaffarellifigalli2012}, \cite{caffarellisilvestre2014},
\cite{changdavila2012}, \cite{changdavila2014},
\cite{changdavila2016}, \cite{kassmannrangschwab2014} and
\cite{ss}. The majority of these results make the pointwise assumption
on the kernel $K(v,v') \approx |v-v'|^{-d-2s}$, and therefore are not
directly applicable to the Boltzmann equation. It is only in
\cite{bjorlandcaffarellifigalli2012}, \cite{kassmannrangschwab2014}
and \cite{ss} that more singular kernels are considered. The
assumptions in \cite{ss} are sufficiently general to be applicable to
the space homogeneous Boltzmann equation. Our result is for equations
in \emph{divergence} form, and thus none of these papers either
implies or follows from ours. Interestingly, we use some of the ideas
for nondivergence equations. Most importantly, the ink-spots theorem
that we develop in Section \ref{sec:ink-spot} is a generalization of a
similar covering argument in \cite{ss}.

We stress that our main regularity result in Theorem \ref{thm:holder}
requires the equation to hold in a bounded domain only. The parameters
$\lambda$ and $\Lambda$ in the assumptions \eqref{e:Klower},
\eqref{e:Kabove}, \eqref{e:Kcancellation0} and
\eqref{e:Kcancellation1} will deteriorate as $|v| \to \infty$ in the
case of the Boltzmann equation.

\subsection{Organization of the article}

We set our notation and further analyze our assumptions in
Section~\ref{sec:notation}.  The relationship between our main results
and the Boltzmann equation is discussed in
Section~\ref{sec:boltzmann}, where we prove in particular that the
Boltzmann kernel satisfies the assumptions listed above.  The analysis
of the operator $\Lv$ and its associated bilinear form $\Ex$ is done
in Section~\ref{sec:bilinear}. This section should be interesting in
itself. This is where the generality of our assumptions on the kernels
is reflected. All the results in Section \ref{sec:bilinear} would be
straight forward if we assumed that the kernels satisfy
$K(v,v') = K(v',v)$ and $K(v,v') \approx |v-v'|^{-d-2s}$.  The core of
the proof of the Weak Harnack inequality and H\"older estimates for
integro-differential equations is done in sections \ref{sec:fdl},
\ref{sec:barriers}, \ref{sec:intermediate}, \ref{sec:growth},
\ref{sec:ink-spot} and \ref{sec:main}.  Section~\ref{sec:reduction}
contains fairly unsurprising statements that are technically necessary
for the completeness of the rest of our proofs. Experts will probably
skim through this section quickly.  The appendix~\ref{sec:appendix}
contains a new proof of the coercivity bound for the Boltzmann
equation (Subsection~\ref{subsec:lowerbound}) and some technical
lemmas (Subsection~\ref{subsec:technical}).

\section{Preliminaries}
\label{sec:notation}

\subsection{Notation}

For a real number $a$, $a^+ = \max(a,0)$. 

A constant is called \emph{universal} if it only depends on dimension and the constants $s$, $\lambda$
and $\Lambda$ in the assumptions \eqref{e:Klower}, \eqref{e:Kabove}, 
\eqref{e:Kcancellation0} and \eqref{e:Kcancellation1}.

When we write $a \lesssim b$, we mean that there exists a universal constant
$C$, so that $a \leq C b$. We write $a \approx b$ when both $a \lesssim b$ and $b \lesssim a$ hold.

When we write $\dot H^s(\Omega)$ for some $\Omega \subset \R^d$, we
mean the space whose norm is given by
\[ \|f\|_{\dot H^s(\Omega)}^2 := \iint_{\Omega \times \Omega} \frac{|f(v') -f(v)|^2}{|v-v'|^{d+2s}} \dd v' \dd v.\]

The space $H^s(\Omega)$ is the one corresponding to the norm
\[ \|f\|_{H^s(\Omega)}^2 := \|f\|_{\dot H^s(\Omega)}^2 +
\|f\|_{L^2(\Omega)}^2.\]
The space $H^s_0(\Omega)$ is obtained by completing the space of
$C^\infty$ functions in $\R^d$ supported in $\Omega$ with respect to
the norm $\|\cdot\|_{H^s(\Omega)}$. When $\Omega=\R^d$,
$H^s_0(\Omega) = H^s(\Omega)$.  We also define $H^{-s}(\Omega)$ as the
dual of $H^s_0(\Omega)$.

It is well known that
$\|f\|_{\dot H(\R^d)}^2 = \int_{\R^d} |\xi|^{2s} |\hat f(\xi)|^2 \dd
\xi$.
Moreover, $f \in H^{-s}(\R^d)$ if and only if
$f = g_1 + (-\Delta)^{s/2} g_2$ with $g_1, g_2 \in
L^2(\R^d)$.
Similarly, $f$ is in the dual of $\dot H^s(\R^d)$ if
$f = (-\Delta)^{s/2} g$ for some function $g \in L^2(\R^d)$.

Note also that if $f: \R^d \to \R$ is supported in $B_1$, then
$\|f\|_{H^s(\R^d)}$, $\|f\|_{H^s(B_2)}$ and $\|f\|_{\dot H^s(B_2)}$
are all equivalent.

\subsection{First consequences of assumptions}
\label{sub:com}

After an obvious readjustment of constants (depending
  on $d$ and $s$), the assumption \eqref{e:Kabove} is equivalent to
  the following
\[
 \begin{cases}
(i) & \int_{B_{2r}(v) \setminus B_r(v)} K (v,v') \dd v' \leq \Lambda r^{-2s}  \text{ for any } 
r>0 \text{ and } v \in B_\Radius \\[1ex]
(ii) & \int_{B_\Radius \cap B_{2r}(v') \setminus B_r(v')} K (v,v') \dd v \leq \Lambda r^{-2s}  \text{ for any } 
r>0 \text{ and } v' \in B_\Radius .
\end{cases}
\]
It is also equivalent to
\[
 \begin{cases}
(i) & \int_{B_r(v)} |v-v'|^2 K (v,v') \dd v' \leq \Lambda r^{2-2s}  \text{ for any } 
r>0 \text{ and } v \in B_\Radius \\[1ex]
(ii) & \int_{B_\Radius \cap B_r(v')} |v-v'|^2 K (v,v') \dd v \leq \Lambda r^{2-2s}  \text{ for any } 
r>0 \text{ and } v' \in B_\Radius.
\end{cases}
\]
We use the three forms of the assumption \eqref{e:Kabove}
indistinctively in different parts of the paper.

As we mentioned before, when $s>1/2$, if the assumption
\eqref{e:Kcancellation1} holds for some value of $r=r_0$ and also
\eqref{e:Kabove} holds, then \eqref{e:Kcancellation1} also holds for
any other value $r \in (0,\Radius/8]$. The reason is the following
computation. We write it for the case $r<r_0$. The case $r>r_0$
follows similarly.
\begin{align*} 
 \left| PV \int_{B_r}
    (v-v') \left( K(v,v') - K (v',v)\right) \dd v' \right| &\leq  \left| PV \int_{B_{r_0}}
    (v-v') \left( K(v,v') - K (v',v)\right) \dd v' \right| \\ &\phantom{=} + \left| PV \int_{B_{r_0} \setminus B_r}
    (v-v') \left( K(v,v') - K (v',v)\right) \dd v' \right|, \\
&\leq \Lambda (1+r_0^{1-2s} ) + \int_{\R^d \setminus B_r} |v-v'| (K(v,v')+K(v',v)) \dd v', \\
&\leq \Lambda (1+r_0^{1-2s} + r^{1-2s}).
\end{align*}
The last inequality is a consequence of \eqref{e:Kabove}. Note that,
for the case $s=1/2$, this last integral may be divergent and thus the
assumption \eqref{e:Kcancellation1} is made so that the inequality
holds for all values of $r$ in the range $(0,\Radius/8]$.

\subsection{Invariant transformations}
\label{sub:invariant}

If $f$ satisfies the equation \eqref{e:main} for some kernel $K$
satisfying \eqref{e:Klower}, \eqref{e:Kabove}, \eqref{e:Kcancellation0} and
\eqref{e:Kcancellation1}, then the scaled function
$f_r(t,x,v) = f(r^{2s} t, r^{2s+1} x, r v)$ satisfies a modified
equation
\[ \partial_t f_r + v \nabla_x f_r + \tilde \Lv f_r = h_r,\]
where
\begin{align*}
h_r(t,x,v) &= r^{2s} h(r^{2s} t, r^{2s+1} x, r v), \\
K_r(t,x,v,v') &= r^{d+2s} K(r^{2s} t, r^{2s+1} x, r v).
\end{align*}
For any $r \in [0,1]$, the kernel $K_r$ satisfies the assumptions \eqref{e:Klower}, 
\eqref{e:Kabove}, \eqref{e:Kcancellation0} and
\eqref{e:Kcancellation1} with a larger radius $\Radius/r$ instead of $\Radius$. Moreover,
$\|h_r\|_{L^\infty(Q_1)} \leq r^{2s} \|h\|_{L^\infty(Q_1)} \leq
\|h\|_{L^\infty(Q_1)} $.  

The equation is also invariant under the family of transformations
$\mathcal{T}_{z_0}$. Here
$z_0 = (t_0,x_0,v_0) \in \R \times \R^d \times \R^d$.
\begin{align*}
 \mathcal{T}_{z_0} (t,x,v) = (t_0+t, x_0 +x + t v_0 ,v_0+v) = z_0 \circ z , \\
\mathcal{T}_{z_0}^{-1} (t,x,v) = (t-t_0,x-x_0 - (t-t_0)v_0,v-v_0) = z_0^{-1} \circ z
\end{align*}
(see Figure~\ref{fig:transformation}).
\begin{figure}
\includegraphics[height=4cm]{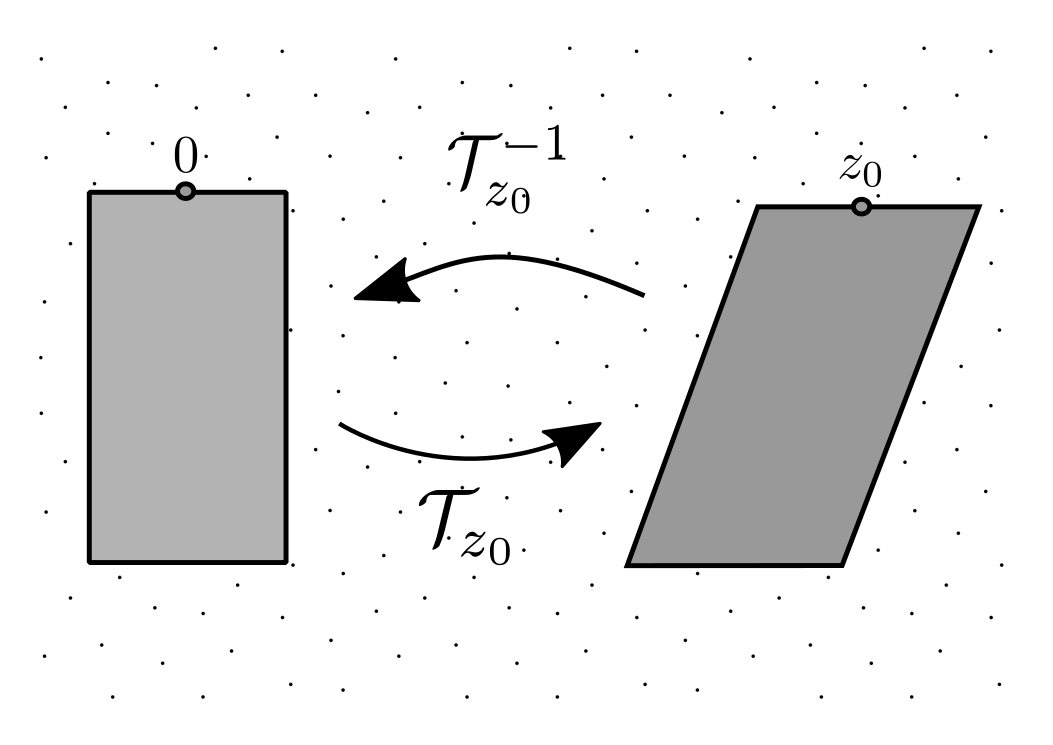}
\caption{The transformation leaving the equation invariant. On the
  left, a straight cylinder centered at the origin.  On the right a
  slanted cylinder centered at $z_0$.  }
\label{fig:transformation}
\end{figure}
Indeed, the product $\circ$ induces a Lie group structure on
$\R \times \R^d \times \R^d$. We remark that
\[ (T,0,0) \circ (t,x,v) = (t+T,x,v), \]
that is to say, translation in time coincides with a left Lie product. 

Because of the scaling and the group action that keep our
class of equations invariant, we are forced to work with
\emph{slanted} cylinders: for a given center $z_0=(t_0,x_0,v_0)$ and
some radius $r>0$ by the following formula
\begin{equation}\label{defi:cylinder}
  Q_r (z_0) = \{ (t,x,v): -r^{2s} \le t-t_0 \le 0, |v-v_0| < r, |x-x_0 - (t-t_0) v_0| < r^{1+2s} \}.
\end{equation}
Remark that for $z_0=0$, 
\[ Q_r = Q_r(0) = (-r^{2s},0] \times B_{r^{2s+1}} \times B_r .\]

\subsection{The fractional Kolmogorov equation}

In this subsection we review the fractional Kolmogorov's equation:
  \begin{equation} \label{e:kolmogorov}
 f_t + v \cdot \grad_x f + (-\lap_v)^{s} f = h.
\end{equation}
The previous Lie group structure also preserves this
  equation. There is a fundamental solution $J(t,x,v)$ which has the
  following form
\[ 
J(t,x,v) = c_d \, \frac1{t^{d+d/s}} \, \mathcal{J} \left( \frac x
  {t^{1+1/2s}} , \frac v {t^{1/{2s}}} \right).
\]
The function $\mathcal{J}$ can be computed
explicitly in Fourier variables by the formula
\[ \hat {\mathcal J} (\varphi,\xi) = \exp \left( -\int_0^1 |\xi-\tau \varphi|^{2s} \dd \tau \right).\]
In the physical variables $x$ and $v$, the formula for $\mathcal J$ is not explicit. However, some simple properties can be deduced from classical considerations. We collect them in the following proposition.
\begin{prop}[Fundamental solution of the fractional Kolmogorov equation] \label{p:kolmogorov-kernel}
The functions $J$ and $\mathcal J$ have the following properties.
\begin{enumerate}
\item The function $\mathcal J$ is $C^\infty$ and decays polynomially at infinity. Moreover, $\mathcal J$ and all its derivatives are integrable in $\R^{2d}$.
\item For every $t>0$, $\int_{R^{2d}} J(t,x,v) \dd v \dd x = 1$.
\item Both functions are nonnegative: $J \geq 0$ and $\mathcal J \geq 0$.
\item For any $p \geq 1$, we have
\begin{align*} 
 \|J(t,\cdot,\cdot) \|_{L^p(\R^{2d})} &=  t^{-d(1+1/s)(1-1/p)}\, \|\mathcal J\|_{L^p(\R^{2d})}, \\
 \|(-\Delta)^{s/2} J(t,\cdot,\cdot) \|_{L^p(\R^{2d})} &=  t^{-d(1+1/s)(1-1/p)-1/2} \, \|(-\Delta)^{s/2} \mathcal J\|_{L^p(\R^{2d})}.
\end{align*}
In particular, for $p_\star = (2d(1+s)+2s)/(2d(1+s)+s) \in (1,2)$, we have $\|J(t,\cdot,\cdot)\|_{L^{p_\star}(R^{2d})} \leq C t^{1/2 - 1/p_\star}$ and \[ \|(-\Delta)_v^{s/2} J(t,\cdot,\cdot)\|_{L^{p_\star}(\R^{2d})} \leq C t^{-1/p_\star}.\]
\end{enumerate}
\end{prop}
The initial value problem \eqref{e:kolmogorov} is solved by the formula
\begin{equation} \label{e:duhamel}
\begin{aligned}
  f(t,x,v) = &\int_{\R^d} \int_{\R^d} f_0(y,w) J(t,x-y- tw,v-w) \dd w
  \dd y \\ &+\int_{0} ^t \int_{\R^d} \int_{\R^d} h(\tau,y,w)
  J(t-\tau,x-y-(t-\tau)w,v-w) \dd w \dd y \dd \tau
\end{aligned}
\end{equation}

We define the \emph{modified} convolution $\ast_t$ by the formula
\[ h \ast_t j (x,v) = \iint h(y,w) j(x-y-tw,v-w) \dd w \dd y.\]
If we make the change of variables $\tilde j(x,v) = j(x+tv,v)$, then
$h \ast_t j(x,v) = h \ast \tilde j (x-tv,v)$. Thus, the
\emph{modified} convolution is the same as the usual convolution
conjugated by that change of variables (of Jacobian one). We observe
that this convolution satisfies the usual Young's inequality:
\begin{equation}\label{eq:young-modified}
  \left\Vert \iint h(y,w) j(x-y-tw,v-w) \dd w \dd y \right\Vert_{L^r_{x,v}} \leq \|h\|_{L^p_{x,v}} \|j\|_{L^q_{x,v}} 
\qquad \text{ independenly of } t.
\end{equation}
Here $1+1/r = 1/p + 1/q$.

The following proposition is simply a consequence of Young's inequality.
\begin{prop}[Gain of integrability] \label{l:kolmogorov-youngs}
  Let $f$ be the solution of \eqref{e:kolmogorov} in
  $[0,T] \times \R^{2d}$, with $f(0,x,v) = f_0(x,v) \in
  L^2(\R^{2d})$.
  Assume $h \in L^2([0,T] \times \R^d, H^{-s}(\R^d))$. Then
  \[ \|f\|_{L^q([0,T] \times \R^{2d})} \leq C(T) \left(
    \|f_0\|_{L^2(\R^{2d})} + \|h\|_{L^2([0,T]\times \R^{d},
      H^{-s}(\R^d) ) } \right) \]
  for any $q$ such that $1/q > 1/p_\star - 1/2$ and $p_\star$ is the
  one from Proposition \ref{p:kolmogorov-kernel}.
\end{prop}
\begin{proof}
Since $h \in L^2([0,T] \times \R^d, H^{-s}(\R^d))$, then there exists $h_1$ and $h_2$ in $L^2([0,T] \times \R^{2d})$. so that $h=h_1 + (-\Delta)^{s/2}_v h_2$ and
\[ \|h_1\|_{L^2} + \|h_2\|_{L^2} \approx \|h\|_{L^2([0,T] \times \R^d, H^{-s}(\R^d))}.\]

We use the formula \eqref{e:duhamel} to solve \eqref{e:kolmogorov}. Let us write $f(t,x,v) = f_1(t,x,v) + f_2(t,x,v) + f_3(t,x,v)$, where
\begin{align*} 
 f_1(t,\cdot,\cdot) &:= f_0 \ast_t J(t,\cdot,\cdot),\\
f_2(t,\cdot,\cdot) &:= \int_0^t h_1(\tau) \ast_{(t-\tau)} J(t-\tau,\cdot, \cdot) \dd \tau, \\
f_3(t,\cdot,\cdot) &:= \int_0^t  h_2(\tau) \ast_{(t-\tau)} (-\Delta)^{s/2}_v J(t-\tau,\cdot, \cdot) \dd \tau.
\end{align*}

Let $p \in [1, p_\star)$ be the number such that $1/q = 1/p -
1/2$. Applying Young's inequality for each value of $t$, we have
\begin{align*} 
 \|f_1(t,\cdot,\cdot) \|_{L^q} &\leq \|f_0\|_{L^2} \|\mathcal J\|_{L^{p}} t^{1/2-\alpha}, \\
 \|f_2(t,\cdot,\cdot) \|_{L^q} &\leq \int_0^t \|h_1(\tau)\|_{L^2} \|\mathcal  J\|_{L^{p}} (t-\tau)^{1/2-\alpha} \dd \tau, \\
 \|f_3(t,\cdot,\cdot) \|_{L^q} &\leq \int_0^t \|h_2(\tau)\|_{L^2} \|(-\Delta)^{s/2}\mathcal  J\|_{L^{p}} (t-\tau)^{-\alpha} \dd \tau.
\end{align*}
Here $\alpha = d(1+1/s)(1-1/p) + 1/2 < 1/p_\star < 1/p$ since $p < p_\star$. Moreover, $q (1/2 - \alpha) > -1$ so $f_1 \in L^q([0,T] \times R^{2d})$ with
\[ \|f_1\|_{L^q([0,T] \times R^{2d})} \leq C T^{(1/2 - \alpha)+1/q} \|f_0\|_{L^2}.\]

We estimate the other two terms applying Young's inequality once again
\begin{align*} 
 \|f_2\|_{L^q([0,T] \times \R^{2d})} &\leq C \|h_1\|_{L^2([0,T],R^{2d})} T^{1/2 + 1/p - \alpha} , \\
 \|f_3\|_{L^q([0,T] \times \R^{2d})} &\leq C \|h_2\|_{L^2([0,T],R^{2d})} T^{1/p - \alpha}.
\end{align*}
This finishes the proof.
\end{proof}

\begin{remark}
The power $p$ in Lemma \ref{l:kolmogorov-youngs} can also be taken equal to $p_\star$ by using the weak-type Young's inequality in place of the usual Young's inequality for convolutions and a finer analysis of the $L^{p_\star,\infty}$ norm of $J$. Since we do not need a sharp result  in this paper, we prefer to keep this lemma as elementary as possible.
\end{remark}

\section{The Boltzmann kernel}
\label{sec:kernel}
\label{sec:boltzmann}

In this subsection, we explain why the Boltzmann collision operator
associated with inverse power-law potentials (see \eqref{assum:B})
satisfy the assumptions we made on the kernel as soon the quantities $M(t,x)$, $E(t,x)$ and $H(t,x)$ defined in the introduction are under control.

\subsection{The collision operator as an integro-differential operator plus a lower order term.}

It is classical to observe that $B$ can be replaced with any
$\tilde B$ satisfying for all $k,\sigma \in \Sp$,
\[ B(r,k \cdot \sigma) + B(r,-k \cdot \sigma) = \tilde B(r,k \cdot
\sigma) + \tilde B(r,-k \cdot \sigma).\]
For this reason, we can (and do) follow \cite{luis} and assume
\begin{equation}\label{assum:B-bis}
  \text{If } k \cdot \sigma < 0, \text{ then} \quad B(r, k \cdot \sigma ) 
\approx r^\gamma |\cos (\theta /2)|^{\gamma + 2s + 1}
\end{equation}
where $\cos (\theta/2) := \frac{v-v_*}{|v-v_*|} \cdot \frac{v-v_*'}{|v-v_*-|}$.

We split $Q$ in $Q_1$ and $Q_2$ as follows: $Q(f,g) = Q_1 (f,g)+Q_2(f,g)$ with 
\[\begin{cases}
Q_1 (f,g) &= \iint f'_* (g'-g) B \dd v_* \dd \sigma, \\[1ex]
Q_2(f,g) &= \left( \iint (f'_*-f_*) B \dd v_* \dd \sigma \right) g .
\end{cases}\]
Such a decomposition appears for instance in \cite{v:review,luis}. 

The term $Q_1$ can be rewritten using Carleman coordinates \cite{carleman}. 
\begin{lemma}[The integro-differential operator
  \cite{luis}] \label{l:Kformula} The term $g \mapsto Q_1(f,g)$
  corresponds to some linear operator $\Lv g$ with $K=K_f$ given by
\begin{equation} \label{defi:kf}
 K_f (v,v') = \frac{2^{d-1}}{|v'-v|} \int_{w \perp v'-v} f(v+w) B(r, \cos \theta) r^{-d+2} \dd w 
\end{equation}
where 
\[ r^2 = |v'-v|^2 + |w|^2 \quad \text{ and } \quad 
\cos \theta = \frac{v-v'-w}{|v-v'-w|}\cdot \frac{v'-v-w}{|v'-v-w|}.\]
\end{lemma}
The proof of the previous lemma is simply a change of variables to
Carleman coordinates, see \cite{luis}. It is recalled in Appendix for
the reader's convenience, see Lemma~\ref{lem:cdv}.  The term
$Q_2(f,g)$ is of lower order because of the cancellation lemma
\cite{villani1999},\cite{advw}.
\begin{lemma}[Cancellation \cite{villani1999},
  \cite{advw}] \label{l:classic-cancellation} The following formula
  holds true for any $v \in \R^d$,
\[ \iint (f'_*-f_*) B \dd v_* \dd \sigma = C_b |\cdot|^\gamma \star f (v)\]
with 
\[  C_b = \int_{\Sp} \bigg\{ \frac{2^{(d+\gamma)/2}}{(1+\sigma \cdot e
  )^{(d+\gamma)/2}} -1 \bigg\} b (\sigma \cdot e) \dd \sigma\]
for any $e \in \Sp$. 
\end{lemma}

\subsection{Coercivity bound}

We prove this lower bound in the Appendix \ref{subsec:lowerbound}. This
is a well known result in the Boltzmann literature.
\begin{prop}[Lower bound
  \cite{advw,gressman2011global}] \label{p:lowerbound}
  Let $g: \R^d \to \R$ be a function supported in $B_\Radius$. Then
\[ c \|g\|_{\dot H^s(\R^d)}^2 \leq - \int_{\R^d} Q(f,g)(v) g(v) \dd v + C
\|g\|_{L^2(\R^d)}^2.\]
The constants $c$ and $C$ depend on the mass, energy and entropy of
$f$, the dimension $d$ and the radius $\Radius$. In other words, $K_f$
satisfies \eqref{e:Klower} as soon as mass, energy and entropy of $f$
are bounded. In the case of the mass, we also need it to be bounded below.
\end{prop}

The assumption \eqref{e:Knondegeneracy}, which we need in the case $s \in (0,1/2)$ is clearly satisfied by the Boltzmann kernel. This follows as consequence of Lemma 4.8 in \cite{luis}.

\subsection{Upper bounds}

In this paragraph, we justify that the Boltzmann kernel satisfies
\eqref{e:Kabove}.  We recall that \eqref{e:Kabove}-(i) was already
proved in \cite{luis}. Recall the equivalent formulations of
\eqref{e:Kabove} explained in Section \ref{sec:notation}.
\begin{lemma}[Upper bound \eqref{e:Kabove}-(i)
  {\cite[Corollary~4.4]{luis}}] \label{p:upperbound-1} Assume
  $\gamma + 2s \le 2$. Then for all $r>0$ and $v \in B_\Radius$,
  \[ \int_{B_{2r}(v) \setminus B_{r}(v)} K_f (v,v') \dd v' \lesssim
  r^{-2s} \left( \int_{\R^d} f(z) |z-v|^{\gamma+2s} \dd z \right).\]
  In particular, $K_f$ satisfies \eqref{e:Kabove}-(i) with
    $\Lambda$ that depends only on
    $\|f \star |\cdot|^{\gamma + 2s}\|_{L^\infty(B_2)}$. More precisely,  if
    $\gamma +2 s \in [0,2]$, then $\Lambda$ in \eqref{e:Kabove}-(i)
    depends only on mass and energy; if $\gamma + 2 s \le 0$, then it
    depends on mass, dimension, $\gamma, s$ and $\|f\|_{L^\infty}$. 
\end{lemma}

We can now derive  \eqref{e:Kabove}-(ii).
\begin{lemma}[Upper bound \eqref{e:Kabove}-(ii)] \label{l:Kabove2}
  Assume $\gamma + 2s \le 2$.  Then for all $v' \in B_\Radius$ and $r>0$,
  \[ \int_{\R^d \setminus B_r(v')} K_f(v,v') \dd v \lesssim 
r^{-2s} \left( \int_{\R^d} f(z) |z-v'|^{\gamma+2s} \dd z \right).\]
  In particular, $K_f$ satisfies \eqref{e:Kabove}-(ii) with $\Lambda$
  that only depends on
  $\|f \star |\cdot |^{\gamma +2s} \|_{L^\infty(B_\Radius)}$.  More precisely, if
  $\gamma +2 s \in [0,2]$, it depends only on mass and energy; if
  $\gamma + 2 s \le 0$, $\Lambda$ then it depends on mass, dimension,
  $\gamma, s$ and $\|f\|_{L^\infty}$.
\end{lemma}
\begin{proof}
According to the formula for $K(v,v')$ in terms of $f$ (Corollary 4.2 in \cite{luis}),
\[ K_f(v,v') \approx |v-v'|^{-d-2s} \left( \int_{w \perp (v-v')} f(v+w)
  |w|^{\gamma+2s+1} \dd w \right).\]
Without loss of generality, let us take $v'=0$ in order to simplify
the notation. Therefore
\[ \int_{\R^d \setminus B_r} K_f(v,0) \dd v \lesssim \int_r^\infty \rho^{-d-2s} \int_{\partial B_\rho} 
\int_{w \perp v} f(v+w) |w|^{\gamma+2s+1} \dd w \dd S(v) \dd \rho.\]
Applying \eqref{e:cdv2} from Lemma \ref{l:change-of-variables},
\begin{align*} 
 \int_{\R^d \setminus B_r} K_f(v,0) \dd v &
\lesssim \int_r^\infty \rho^{-2s-1} \int_{\R^d \setminus B_\rho} f(z) \frac{(|z|^2-\rho^2)^{\frac{d-2+\gamma+2s}2}}{|z|^{d-2}} \dd z \dd \rho, \\
&= \int_{\R^d \setminus B_r} \frac{f(z)}{|z|^{d-2}} \left( \int_r^{|z|} \rho^{-2s-1} (|z|^2-\rho^2)^\frac{d-2+\gamma+2s}2 \dd \rho \right)  \dd z,\\
&\leq \int_{\R^d \setminus B_{r}} \frac{f(z)}{|z|^{d-2}} \left( r^{-2s} |z|^{d-2+\gamma+2s} \right) \dd z, \\
&= r^{-2s} \int_{\R^d \setminus B_r} f(z) |z|^{\gamma+2s} \dd z. \qedhere
\end{align*}
\end{proof}

\subsection{The  cancellation assumptions}

In this paragraph, we justify that the kernel associated with the
Boltzmann equation satisfies the cancellation assumptions
\eqref{e:Kcancellation0} and \eqref{e:Kcancellation1}.

The first cancellation condition, assumption~\eqref{e:Kcancellation0},
is essentially the cancellation lemma, which is well known in the
kinetic community.
\begin{lemma}[Classical cancellation lemma] \label{l:Kcancellation0}
The kernel $K_f$ satisfies for all $v \in \R^d$,
\[ 
\left| PV \int_{\R^d}  (K_f(v,v')-K_f(v',v)) \dd v' \right| 
\leq C \left( \int_{\R^d} f(z) |z-v|^{\gamma}   \dd z \right).
\]
In particular, $K_f$ satisfies \eqref{e:Kcancellation0} with
  $\Lambda$ that only depends on
  $\|f \star |\cdot|^{\gamma} \|_{L^\infty(B_{\bar R})}$. More precisely, if
  $\gamma \in [0,2]$, $\Lambda$ in \eqref{e:Kcancellation0} depends
  only on upper bounds on mass and energy; if $\gamma \le 0$, it
  depends on mass, dimension, $\gamma$ and $\|f\|_{L^\infty}$.
\end{lemma}
\begin{proof}
Let $P(v)$ denote $PV \int (K_f(v',v) - K_f(v,v')) \dd v'$. 
In view of the definition of $K_f$, we have
\begin{align*}
P(v) & = 2^{d-1} \int_{\R^d}  \dd v' \left( \int_{w \perp v'-v}  f (v+w) 
\frac{B(r,\cos \theta)}{|v'-v| r^{d-2}} \dd w - \int_{w \perp v'-v} 
 f(v'+w) \frac{B(\tilde r,\cos \tilde \theta)}{|v'-v| {\tilde r}^{d-2}} \dd w \right)\\
& = 2^{d-1}  \int_{\R^d}  \int_{w \perp v'-v}  (f (v+w)-f(v'+w)) 
\frac{B(r,\cos \theta)}{|v'-v| r^{d-2}} \dd w \dd v' 
\end{align*}
since $r=\tilde r$ and $\cos \theta = \cos \tilde \theta$. Using now \eqref{cdv:v'} from 
Lemma~\ref{lem:cdv}, we
get
\[ \nonumber P(v) = \int_{\R^d} \int_{\Sp} (f(v'_*)-f(v_*)) B(r,\cos
\theta) \dd \sigma \dd v_*. \]
The  cancellation Lemma~\ref{l:classic-cancellation} tells us that
\[ P(v) = c \int_{\R^d} |v-w|^\gamma f(w) \dd w.\]

The proof is now complete.
\end{proof}

\begin{lemma}[More subtle cancellation lemma] \label{l:Kcancellation1}
The two following properties hold true for any $R>0$,
\begin{align}
\label{e:cancellation1-zeroterm} 
PV \int_{B_R(v)} (v'-v) K_f(v,v') \dd v' &= 0, \\
\label{e:cancellation1-diffterm} 
\left| PV \int_{B_R(v)} (v'-v) K_f(v,v') \dd v \right| 
&\leq C \left( \int_{\R^d} f(z) |z-v'|^{1+\gamma}   \dd z \right).
\end{align}
In particular, the kernel satisfies \eqref{e:Kcancellation1} and
$\Lambda$ only depends on $\|f \star |\cdot|^{1+\gamma}\|_{L^\infty(B_{\bar R})}$.
More precisely, if $\gamma \in [-1,1]$, $\Lambda$ in
\eqref{e:Kcancellation1} depends only on upper bounds on mass and
energy; if $\gamma \le -1$, it depends on  mass, dimension, $\gamma$
and $\|f\|_{L^\infty}$.
\end{lemma}
\begin{proof}
  The first identity \eqref{e:cancellation1-zeroterm} is obvious from
  the symmetry property: $K_f(v,v+w) = K_f(v,v-w)$. The difficulty is
  thus to justify the second identity
  \eqref{e:cancellation1-diffterm}.

  Without loss of generality, let us assume $v'=0$.  In view of
  Lemma~\ref{l:Kformula} (coming from \cite{luis}), the kernel $K_f$
  can be written for $v'=0$ as follows,
\[ K_f (v,0) = \frac{2^{d-1}}{|v|}\int_{\{w: w\perp v\}} f(v + w) B(r,\cos \theta) {\frac1{r^{d-2}}}\dd w \]
where $r^2 = |v|^2 + |w|^2 = |z|^2$ and $z=v+w$ and
\[ \cos \theta = \widehat{v+w} \cdot \widehat{w-v} = \frac{|w|^2-|v|^2}{|v+w|^2} 
= \frac{|z|^2-2|v|^2}{|z|^2} .\]
The way $b(\cos \theta)$ is modified for $\cos \theta <0$, implies 
that 
\[ {\frac1{|v|r^{d-2}}}B(r,\cos \theta) \approx |v|^{-d - 2 s}
|w|^{\gamma + 2s +1 }.\]
Since $r$ and $\cos \theta$ only depend on $|z|$ and $|v|$,
this implies that there exists $A (|z|,|v|)$ such that
\[ {\frac1{|v|r^{d-2}}} B(r,\cos \theta) = A (|z|,|v|) |v|^{-d - 2 s} |w|^{\gamma + 2s +1 }\]
and a constant $C_A>1$ such that for all $z,v$,
\[ C_A^{-1} \le A (|z|,|v|) \le C_A.\]
In the following computation, the definition of $r$ changes. We write $r = |v|$. We integrate in $v$ first on spheres $\partial B_r$ and then with respect to the radius $r$.
\begin{align*} 
\abs{ PV \int_{{B_R}} v K_f(v,0) \dd v} &= \abs{ \int_0^{{R}} r^{-d-2s} \int_{\partial B_r}  \int_{w \perp v} 
v A(|v+w|,|v|) f(v+w) |w|^{\gamma+2s+1} \dd w \dd v \dd r }\\
\intertext{We use the change of variables \eqref{e:cdv3} of Lemma \ref{l:change-of-variables}. Note that $|w|^2 + r^2 = |z|^2$.}
\abs{ PV \int_{{B_R}} v K_f(v,0) \dd v}&= \omega_{d-2} \abs{ \int_0^{ R} r^{1-2s} \int_{\R^d \setminus B_r} 
A(|z|,r) z f(z) \frac{(|z|^2-r^2)^{\frac{d-2+\gamma+2s}2}}{|z|^{d}} \dd z \dd r } \\
&= \omega_{d-2} \abs{ \int_{\R^d} z f(z) |z|^{-d} \left( \int_0^{{ \min(|z|,R) } } r^{1-2s} 
A(|z|,r) (|z|^2-r^2)^{\frac{d-2+\gamma+2s}2} \dd r \right) \dd z}\\
&\leq \omega_{d-2} C_A \int_{\R^d} f(z) |z|^{-1+\gamma+2s} \left( \int_0^{{ \min(|z|,R) }} r^{1-2s}  \dd r \right) \dd z\\
&\leq C \int_{\R^d} f(z) |z|^{1+\gamma}  \dd z.
\end{align*}
The proof is now complete. 
\end{proof}
\begin{remark}
  There is a subtle cancellation that allows this proof to work. The whole point of
  this lemma is that the principal value of the integral is bounded
  around the origin. The reader will notice that here we end up with
  an integral of the form $\int_0^{|z|} r^{1-2s} \dd r$.  In the proof
  of Lemma \ref{l:Kabove2}, we end up with an intergrand $r^{-1-2s}$
  which is not integrable around the origin. The difference originates
  in Lemma \ref{l:change-of-variables} given in Appendix. The third
  identity in that lemma incorporates an extra cancellation due to the
  fact that the average values of $v \in \partial B_r$ so that
  $v+w = z$, for some $w \perp v$, is $r^2 z / |z|^2$.
\end{remark}

\begin{remark}
Note that the cancellation condition given in Lemma \ref{l:Kcancellation1} is slightly stronger than \eqref{e:Kcancellation1} since the right hand side is bounded independently of $R$ even when $s > 1/2$. Moreover, a rate of convergence to zero as $R \to 0$ can be deduced from the proof.
\end{remark}

\subsection{Proofs of Theorems \ref{thm:holder-boltzmann} and \ref{thm:lowerbound}}

In this subsection we explain how Theorems \ref{thm:holder-boltzmann}
and \ref{thm:lowerbound} follow from Theorems \ref{thm:holder} and
\ref{thm:whi}. Theorem \ref{thm:holder-boltzmann} is indeed a straight forward
application of Theorem \ref{thm:holder}.

\begin{proof}[Proof of Theorem \ref{thm:holder-boltzmann}]
The Boltzmann equation can be written in the form
\[ f_t + v \cdot \nabla_x f = \left( \int_{\R^d} (f(v') -f(v)) K_f(v,v') \dd v' \right) + c \left( \int_{\R^d} f(v-w) |w|^\gamma \dd w \right) f.\]
Thus, if we define
\[ h := c \left( \int_{\R^d} f(v-w) |w|^\gamma \dd w \right) f,\]
then $h \in L^\infty$ with its norm bounded in terms of $\|f\|_{L^\infty}$ and $M_0$.

Moreover, from Proposition \ref{p:lowerbound} and Lemmas
\ref{p:upperbound-1}, \ref{l:Kabove2}, \ref{l:Kcancellation0} and
\ref{l:Kcancellation1}, the kernel $K_f$ satisfies the assumptions
\eqref{e:Klower}, \eqref{e:Kabove}, \eqref{e:Kcancellation0} and
\eqref{e:Kcancellation1}. Thus, the proof is finished as a corollary
of Theorem \ref{thm:holder}.
\end{proof}

Theorem \ref{thm:lowerbound} follows mostly from Theorem \ref{thm:whi}. We use some other results which are presented later in this article which allow us to extend the lower bound to an arbitrary radius $R>0$.
\begin{proof}[Proof of Theorem \ref{thm:lowerbound}]
Without loss of generality, we assume $T=4$. The general case follows by scaling.

  Like in the proof above of Theorem \ref{thm:holder-boltzmann}, we
  have that $f$ is a supersolution of \eqref{e:main} for some $h \geq
  0$. In particular,
\[ f_t + v \cdot \nabla_x f \geq \Lv f. \]
According to Lemma \ref{l:lifted-set}, there is an $R_0 > 0$, $m>0$ and $\ell>0$ so that for all $(t,x)$, 
\[ | \{v \in B_{R_0} : f(t,x,v) \geq \ell\} | \geq m.\]
Let $r_0$ be the one from Theorem \ref{thm:whi}.  We have that
\[ \int_{[0,r_0^{2s}] \times B_{r_0^{1+2s}} \times B_{R_0}} f^\eps \dd
v \dd x \dd t \geq \ell^\eps m r_0^{2s + (1+2s)d}.\]
It is possible to cover the set
$[0,r_0^{2s}] \times B_{r_0^{1+2s}} \times B_{R_0}$ with $N$ slanted
cylinders $Q_{r_0}(z)$ with $N \le (R_0/r_0)^{2d} / c$ for some
universal constant $c>0$. This implies that there must be some point
$z = (r_0^{2s},x,v) \in \{r_0^{2s}\} \times B_{r_0^{2s}(r_0+R_0)}
\times B_{R_0}$ so that
\[ \int_{Q_{r_0}(z)} f^\eps \dd v \dd x \dd t \geq c \ell^\eps m r_0^{2s + (3+2s)d} / R_0^{2d}.\]

Applying Theorem \ref{thm:whi} (properly translated), we get
\[ \inf_{Q_{r_0}(\tilde z)} f \geq c, \]
for some constant $c>0$ and $\tilde z \in \{1\} \times B_{r_0^{1+2s}+R_0 } \times B_{R_0}$.

This bound below in $Q_{r_0}(\tilde z)$ is propagated to
$[2,4] \times B_R \times B_R$, for any arbitrary $R>0$ using the
barrier function from Lemma \ref{l:barrier} if $s < 1/2$ or the
combination of Lemmas \ref{l:intermediate-set-local} and
\ref{lem:linfty} if $s \geq 1/2$.

Note that the geometric setting of Lemmas \ref{l:barrier},
\ref{l:intermediate-set-local} and \ref{lem:linfty} are independent of
the constants $\lambda$ and $\Lambda$. This is important since these
ellipticity constants depend on $R$.
\end{proof}

\section{Study of a bilinear form}
\label{sec:bilinear}

This section is devoted to the study of a general bilinear form $\Ex$
associated with a kernel $K$ through the following formula,
\[
\Ex (\varphi,g) =  - \int (\Lv \varphi)(v) g (v) \dd v =
\lim_{\eps \to 0} \left( \iint_{|v-v'| > \eps} (\varphi(v) -
  \varphi(v')) g(v) K(v,v') \dd v' \dd v \right) .
\]
In the remainder of this section, we abuse notation by ignoring the
limit as $\eps \to 0$. This means that some integrals corresponding to
the odd part of $K$ may need to be understood in the principal value
sense.
Indeed, we recall that the operator $\Lv \varphi$ is given by the
formula
\begin{equation} \label{e:Lv} \Lv \varphi(v) := \lim_{\eps \to 0}
  \int_{\R^d \setminus B_\eps(v)} (\varphi(v') - \varphi(v)) K(v,v')
  \dd v'.
\end{equation}
The limit does not necessarily converge for every value of $v$, even
if $\varphi$ is smooth. The correct understanding of $\Lv \varphi$ as
a distribution is obtained through the analysis of the bilinear form
$\Ex$ done in this section.

When we study the equation \eqref{e:main}, the bilinear form $\Ex$
will be computed for functions $\varphi$ and $g$ depending on values
of $t$ and $x$. The kernel, and consequently also the bilinear form,
depend on $t$ and $x$. In this section we study properties of bilinear
forms like this that will be applied for every fixed value of $t$ and
$x$.

In this section, we also assume that the kernel $K$ is defined for all
values of $v \in \R^d$ and our assumptions hold uniformly. This is
convenient for the exposition and some of the proofs. In Section
\ref{sec:reduction}, we will show that any kernel satisfying
\eqref{e:Klower}, \eqref{e:Kabove}, \eqref{e:Kcancellation0} and
\eqref{e:Kcancellation1} can be extended to all values of $v \in \R^d$ to satisfy a global version of these
assumptions. So, a posteriori, our approach is not limiting.

Since it is not necessary for $K$ to be non-negative for the results in
this section to hold, and they may be used elsewhere, we restate here
our main assumption allowing sign changing kernels.  We make the
following assumptions for some parameter $s \in (0,1)$ and a constant
$\Lambda$.
\begin{equation} 
\label{e:Kabove-signed}
\forall v \in \R^d, \forall r>0, \qquad 
 \begin{cases}
(i) & \int_{\R^d \setminus B_r(v)} |K (v,v')| \dd v' \leq \Lambda r^{-2s}   \\[1ex]
(ii) & \int_{\R^d \setminus B_r(v)} |K (v',v)| \dd v' \leq \Lambda r^{-2s}.  
\end{cases}
\end{equation}

We also state a global version of the cancellation assumptions
\eqref{e:Kcancellation0} and \eqref{e:Kcancellation1}.
\begin{equation}
\label{e:Kcancellation0G}
\forall v \in \R^d, \qquad \left| PV \int_{\R^d} (K(v,v')-K(v',v))  \dd v' \right| 
\le \Lambda .
\end{equation}

In the case $s \ge 1/2$, we also assume that for all $R>0$,
\begin{equation}
\label{e:Kcancellation1G}
\forall v \in \R^d, \qquad \left| PV \int_{B_R(v)} (K(v,v')-K(v',v)) (v-v') \dd v' \right| 
\le \Lambda (1+R^{1-2s}).
\end{equation}

The main result of this section will be that the bilinear form $\Ex$
is bounded in $H^s \times H^s$ provided that \eqref{e:Kabove-signed},
\eqref{e:Kcancellation0G} and \eqref{e:Kcancellation1G} hold. We also
show some other estimates that we will use.

\subsection{Estimates in $H^s$}

The main result of this section is the fact that the bilinear form
$\Ex$ is bounded in $H^s \times H^s$ as soon as our assumptions
\eqref{e:Kabove-signed}, \eqref{e:Kcancellation0G} and
\eqref{e:Kcancellation1G} hold. We state it in the following theorem.
\begin{thm}[Estimate in $H^s$] \label{t:upperbound} Let $K$ satisfy
  \eqref{e:Kabove-signed} and \eqref{e:Kcancellation0G}. If
  $s \geq 1/2$, we also assume that it satisfies
  \eqref{e:Kcancellation1G}. There then exists a constant $C$
  depending only on $s$, $d$ and $\Lambda$, so that
\[ \Ex(f,g) \leq C \|f\|_{H^s} \|g\|_{H^s}.\] 
\end{thm}

It is convenient for some of our proofs to spit $\Ex$ into the
symmetric and anti-symmetric part of $K$. Let
\[ \Ex(\varphi,g) = \Esym (\varphi,g) + \Eskew (\varphi,g) \]
with 
\begin{align*}
\Esym(\varphi,g) &= \frac12 \iint (\varphi(v) - \varphi(v')) (g(v)-g(v'))  K(v,v') \dd v' \dd v,  \\
\Eskew(\varphi,g) &= \frac12 PV \iint (\varphi(v) - \varphi(v')) (g(v)+g(v'))  K(v,v') \dd v' \dd v.
\end{align*}
Note that $\Ex = \Esym$ and $\Eskew=0$ when the symmetry condition
$K(v,v') = K(v',v)$ holds. Likewise, when $K$ is anti-symmetric
(i.e. $K(v,v') = -K(v',v)$) then $\Esym=0$ and $\Eskew =
\Ex$.
Consequently, writing $K$ as the sum of its symmetric plus
anti-symmetric part corresponds to writing $\Ex$ as the sum of $\Esym$
and $\Eskew$.

We will prove Theorem \ref{t:upperbound} estimating $\Esym$ and $\Eskew$ separately. Note that, because of the density of smooth functions in $H^s$, it suffices to prove  Theorem \ref{t:upperbound} when $g$ and $\varphi$ are smooth.

\subsubsection{Estimate of the symmetric part}

\begin{lemma}[Estimate of the symmetric part] \label{lem:upper-sym} Let $K$ be a kernel satisfying
  \eqref{e:Kabove-signed}. Then, there exists a
  constant depending only on $\Lambda$, $s$ and dimension, so that for
  any function $g \in H^s(\R^d)$,
\[ \Esym(g,g) \leq C \|g\|_{\dot{H}^s}^2.\]
\end{lemma}
\begin{proof}
  Without loss of generality, we can assume $K \geq 0$
  here. Otherwise, the value of $\Esym(g,g)$ would only increase if we
  replace $K(v,v')$ by $|K(v,v')|$. 
We write
\begin{equation} \label{e:Esym_as_sum}
 \Esym(g,g) =  \sum_{k=-\infty}^\infty P(2^k)
\end{equation}
where, for any $r >0$, 
\[
P(r) := \iint_{\{ (v,v') \in \R^d \times \R^d : r \leq |v-v'| < 2r\}}  |g(v') - g(v)|^2 K(v,v') \dd v' \dd v.
\]
The key of this proof is to estimate $P(r)$ with a similar
expression involving the kernel $|v-v'|^{-d-2s}$ instead. 

For any values of $v$ and $v'$, let $m = (v+v')/2$, we introduce
an auxiliary point $w \in B_{r/4}(m)$. From the triangle inequality
$|g(v') - g(v)|^2 \leq 2|g(v') - g(w)|^2 + 2|g(w) - g(v)|^2$. Then
\begin{align*}
  P(r) &\lesssim \frac 1 {r^d}  \iint_{\{ (v,v') \in \R^d \times \R^d : r \leq |v-v'| < 2r\}} 
         \int_{B_{r/4}(m)} \left(|g(v') - g(w)|^2 + |g(w) - g(v)|^2 \right) K(v,v') \dd w \dd v' \dd v,\\
  \intertext{we change the order of integration for each term in the integrand,}
       &\leq \frac 1 {r^d} \iint_{r/4 \leq |v'-w| < 5r/4} |g(v') - g(w)|^2 \left( \int_{\Omega_{v',w}} K(v,v') \dd v\right) \dd w \dd v' \\
       & \qquad \qquad + \frac 1 {r^d} \iint_{r/4 \leq |v-w| < 5r/4} |g(w) - g(v)|^2 \left( \int_{\Omega_{v,w}} K(v,v') \dd v' \right) \dd w \dd v.
         \intertext{Here the set $\Omega_{v,w}$ contains all values of $v'$ that correspond to any given pair $(v,w)$. 
         We will only use that $\Omega_{v,w} \subset B_{2r}(v) \setminus B_r(v)$. Both terms are bounded by the same expression 
         using each line in \eqref{e:Kabove-signed}. Thus,}
         P(r) &\lesssim \frac \Lambda {r^{d+2s}} \iint_{r/4 \leq |v-w| < 5r/4} |g(v') - g(w)|^2 \dd w \dd v 
                \lesssim \Lambda \iint_{r/4 \leq |v-w| < 5r/4} \frac{|g(v') - g(w)|^2}{|v-w|^{d+2s}} \dd w \dd v.
\end{align*}

Applying this estimate for each term in \eqref{e:Esym_as_sum}, we get the desired estimate. 
\end{proof}

\subsubsection{Estimates of the anti-symmetric part}

Finding an appropriate upper bound for $\Ex$ when $K$ is not symmetric
is more complicated than for $\Esym$. The cancellation assumptions
\eqref{e:Kcancellation0G} and \eqref{e:Kcancellation1G} are
necessary. We will prove the estimates differently for the case
$s \in (0,1/2)$ and $s \in [1/2,1)$. Note that the hypothesis
\eqref{e:Kcancellation1G} is only used in the later case.

\begin{lemma}[Estimate of $\Lv f$ for $s <1/2$] \label{l:using-Df}
  Assume $s \in (0,1/2)$. Let $K$ be a kernel satisfying
  \eqref{e:Kabove-signed}. The following estimate holds
\[
\|\Lv f\|_{L^2} \leq 
 C \|f\|_{L^2}^{\frac{1-2s}{1+2s}} \|f\|_{\dot H^{s+1/2}}^{\frac{4s}{1+2s}}.
\]
\end{lemma}
\begin{proof}
For some $R >0$, to be determined below, let us write $\Lv f = \ell_0 + \ell_1 + \ell_2$, where
\begin{align*} 
 \ell_0(v) &= \int_{B_R(v)} (f(v') - f(v)) K(v,v') \dd v', \\
 \ell_1(v) &= \int_{\R^d \setminus B_R(v)} f(v') K(v,v') \dd v', \\
 \ell_2(v) &= -\left( \int_{\R^d \setminus B_R(v)} K(v,v') \dd v' \right) f(v).
\end{align*}

We prove the estimate for each one of the three terms.

Let us start with $\ell_2$, which is the easiest. In this case, obviously,
\[ \|\ell_2\|_{L^2} \leq \left( \sup_v \int_{\R^d \setminus B_R(v)}
  K(v,v') \dd v' \right) \|f\|_{L^2} \leq \Lambda R^{-2s}
\|f\|_{L^2}.\]

The estimate for $\ell_1$ involves the Cauchy-Schwarz inequality and an application of Fubini's
theorem. In this case we use the second line of \eqref{e:Kabove-signed}.
\begin{align*} 
 \|\ell_1\|_{L^2}^2 &= \int_{\R^d} \left( \int_{\R^d\setminus B_R (v)} f(v') K(v,v') \dd v' \right)^2 \dd v, \\
&\leq \int_{\R^d} \left( \int_{\R^d\setminus B_R (v)} K(v,v') \dd v' \right) \left( \int_{\R^d\setminus B_R (v)} f(v')^2 K(v,v') \dd v' \right) \dd v, \\
&\leq \Lambda R^{-2s} \int_{\R^d} f(v')^2 \left( \int_{\{v : |v'-v|>R\}} K(v,v') \dd v\right)  \dd v' \leq \Lambda^2 R^{-4s} \|f\|_{L^2}^2.
\end{align*}

We estimate $\ell_0$ using the Cauchy-Schwarz inequality together with \eqref{e:Kabove-signed} and comments from  Subsection~\ref{sub:com}.
\begin{align*}
\|\ell_0\|_{L^2}^2 &= \int_{\R^d} \left( \int_{B_R(v)} (f(v')-f(v)) K(v,v') \dd v' \right)^2 \dd v, \\
&\leq \int_{\R^d} \left( \int_{B_R(v)} (f(v')-f(v))^2 |v-v'|^{-1} K(v,v') \dd v' \right)\left( \int_{B_R(v)} |v-v'| K(v,v') \dd v' \right) \dd v , \\
&\leq \Lambda R^{1-2s} \iint_{|v-v'|<r} (f(v')-f(v))^2 |v-v'|^{-1} K(v,v') \dd v' \dd v .
\end{align*}

The kernel $|v-v'|^{-1} K(v,v')$ satisfies \eqref{e:Kabove} with $s+1/2$ instead of $s$. Then, we apply Lemma \ref{lem:upper-sym} to get
\[ \|\ell_0\|_{L^2}^2 \lesssim R^{1-2s} \|f\|_{\dot H^{s+1/2}} ^2 .\]

The proof is finished choosing $R = (\|f\|_{L^2} / \| f\|_{H^{s+1/2}} )^{2/(1+2s)}$.
\end{proof}

The estimate for $\|\Lv f\|_{L^2}$ when $s \geq 1/2$ is harder to obtain. We will use the following auxiliary kernel.
\[ A(v,w) = 
\int_{\{v' \in B_R (v): (v'-v)\cdot(w-v) \geq |w-v|^2\}}  \frac{|v'-v|^{d-2} K(v,v')}{|w-v|^{d-2} |w-v'|^{d-2}} \dd v'.
 \]
\begin{lemma}[Estimates on the auxiliary kernel]
Let $K$ be a kernel satisfying \eqref{e:Kabove-signed} and $s \geq 1/2$. We have
\begin{equation} \label{e:Akernel-bis}
 \int_{B_R(v)} |A(v,w)| \dd w \lesssim R^{2-2s} \quad \text{ and } \quad \int_{B_R(w)} |A(v,w)| \dd v \lesssim R^{2-2s}.
\end{equation}
\end{lemma}
\begin{proof}
  The first of the two inequalities in \eqref{e:Akernel-bis} is a
  relatively straight forward computation using \eqref{e:Kabove-signed}. Let us choose $v=0$
  without loss of generality. We have
\begin{align*} 
 \int_{B_R} |A(0,w)| \dd w &\leq \int_{B_R}  \int_{\{v' \in B_R : 
v'\cdot w \geq |w|^2\}}  \frac{|v'|^{d-2} |K(0,v')|}{|w|^{d-2} |w-v'|^{d-2}} \dd v' \dd w, \\
&= \int_{B_R} |v'|^{d-2} |K(0,v')| \left( \int_{B_{|v'|/2}(v'/2)}  \frac{1}{|w|^{d-2} |w-v'|^{d-2}} \dd w \right) \dd v', \\
&= C \int_{B_R } |K(0,v')| |v'|^2 \dd v' \lesssim C R^{2-2s}.
\end{align*}

Let us move to the second inequality in \eqref{e:Akernel-bis}. Assume without loss of generality that $w=0$. We have
\begin{align*} 
 \int_{B_R} |A(v,0)| \dd v &\leq \int_{B_R} \int_{\{v' \in B_R(v) : v' \cdot v \leq 0\}}  \frac{|v'-v|^{d-2} |K(v,v')|}{|v|^{d-2} |v'|^{d-2}} \dd v' \dd v \leq I_1 + I_2.
\end{align*}
From the triangle inequality
$|v'-v|^{d-2} \lesssim |v|^{d-2} + |v'|^{d-2}$, we can estimate the
above integral by $I_1 + I_{2}$, where $I_{1}$ and $I_{2}$ are defined
below.
We analyze both terms using \eqref{e:Kabove-signed} and Fubini's
theorem.
\begin{align*}
I_{1} :&= \int_{B_R} \int_{\{v' : |v'-v| < R \text{ and } v' \cdot v \leq 0\} }  \frac{|v|^{d-2} |K(v,v')|}{|v|^{d-2} |v'|^{d-2}} \dd v' \dd v, \\
&=\int_{B_{2R}} \int_{\{v : |v-v'|<R \text{ and } v' \cdot v \leq 0\} \cap B_R}  \frac{ |K(v,v')|}{ |v'|^{d-2}} \dd v \dd v', \\
&\lesssim R^{-2s} \int_{B_{2R}} |v'|^{2-d} \dd v' \lesssim R^{2-2s}.
\end{align*}

We now consider
\[I_{2} := \int_{B_R} \int_{\{v' : |v'-v| < R \text{ and } v' \cdot v \leq 0\} }  \frac{|v'|^{d-2} |K(v,v')|}{|v|^{d-2} |v'|^{d-2}} \dd v' \dd v.\] 
The computation that proves that $I_{2} \lesssim R^{2-2s}$ is almost identical integrating in $v'$ first and in $v$ second.

This concludes the estimate for every term involved in \eqref{e:Akernel-bis}.
\end{proof}
We will use the following lemma from multivariate calculus when
proving estimates on the operator $\Lv f$ associated with the kernel
$K$.
\begin{lemma}[The multi-path lemma] \label{l:multi-path} Let $f: \R^d \to \R$ be any twice differentiable
  function. The following inequality holds for any pair of points
  $v, v' \in \R^d$.
\begin{align}
\label{e:sing2}
\left\vert f(v')-f(v) -  \frac{\nabla f (v) + \nabla f(v')} 2 \cdot (v'-v) \right\vert & \le \frac{1}{d\omega_{d}} |v-v'|^{d-2} \int_{B_R(m)}
\frac{|D^2 f (w)|}{|w-v|^{d-2}|w-v'|^{d-2}} \dd w .
\end{align}
Here $R = |v-v'|/2$ and $m = (v+v')/2$. Thus, $B_R(m)$ is the ball with diameter from $v$ to $v'$.
\end{lemma}
\begin{proof}
For any $w \in B_R(m)$, we write 
\[ \left\vert f(w)-f(v) - \nabla f(v)  \cdot (w-v) \right\vert  \le |w-v|
\int_{0}^{|w-v|} |D^2 f| (v + z \widehat{w-v}) \dd z.\]
where $\widehat{w-v} = (w-v)/|w-v|$.  In particular, computing with
spherical coordinates the integral in the first line below with origin
at $w=v$,
\begin{align*}
\left \vert \left( \fint_{B_R(m)} f(w) \dd w \right)  - f(v) - \nabla f(v) \cdot \frac{(v'-v)}2 \right \vert 
&= \left \vert \fint_{B_R(m)} f(w) - f(v) - \nabla f(v) \cdot (w-v) \dd w \right \vert, \\
& \le \fint_{B_R(m)} |w-v| \left\{ \int_0^{|w-v|} |D^2 f| (v + \zeta \widehat{w-v}) \dd \zeta \right\} \dd w \\
& \lesssim \int_{B_R(m)} \frac{|D^2 f| (w)}{|w-v|^{d-2}} \dd w.
\end{align*}
This implies that 
\[ \left \vert \left( \fint_{B_R(m)} f(w) \dd w \right)  - f(v) - \nabla f(v) \cdot \frac{(v'-v)}2 \right \vert  \lesssim R^{d-2} \int_{B_R(m)} \frac{|D^2 f|
  (w)}{|w-v|^{d-2}|w-v'|^{d-2}} \dd w .\]
Exchanging the role of $v$ and $v'$ and subtracting the resulting
inequalities yields \eqref{e:sing2}.
\end{proof}

\begin{lemma}[Estimate of $\Lv f$ for $s \ge 1/2$] \label{l:using-D2f}
  Assume $s \in [1/2,1)$. Let $K$ be an anti-symmetric kernel (i.e.
  $K(v,v') = -K(v',v)$) satisfying \eqref{e:Kabove-signed} and
  \eqref{e:Kcancellation1G}. The following
  estimate holds
\[
\|\Lv f\|_{L^2} \leq C \|f\|_{L^2}^{1-s} \|D^2f\|_{L^2}^{s}+ \Lambda \|\nabla f \|_{L^2}.
\]
\end{lemma}
\begin{proof}
  We write $\Lv f = \ell_0 + \ell_1 +
  \ell_2$ like in Lemma \ref{l:using-Df}.
  The estimates $\| \ell_1 \|_{L^2} \leq \Lambda R^{-2s} \|f\|_{L^2}$ and $\| \ell_2 \|_{L^2}\leq \Lambda R^{-2s} \|f\|_{L^2}$
  follow an identical proof. The estimate for $\| \ell_0 \|_{L^2}$ is
  different.

Recall that
\[  \ell_0(v) = \int_{B_R(v)} (f(v') - f(v)) K(v,v') \dd v' \]
We write $\ell_0 =\ell_0^0+\ell_0^1 + \ell_0^2$ with
\begin{align*}
\ell_0^0 &= \frac 12 \int_{B_R(v)}  \left(\nabla f (v') - \nabla f(v)\right) (v'-v) K (v,v') \dd v' ,\\
\ell_0^1 & =  \nabla f (v) \int_{B_R(v)} (v'-v) K (v,v') \dd v',\\
\ell_0^2 & = \int_{B_R(v)} \left(f(v') - f(v)-  \frac{\nabla f (v) + \nabla f(v')} 2 \cdot (v'-v) \right) K(v,v') \dd v'.
\end{align*}
The same argument that gives us the upper bound for $\ell_0$ in Lemma \ref{l:using-Df} gives us in this case
\[ \|\ell_0^0\|_{L^2} \leq C R^{1-s} \|\nabla f\|_{\dot H^{s}} \lesssim R^{1-s} \|f\|_{L^2}^{(1-s)/2} \|D^2 f\|_{L^2}^{(1+s)/2} . \]
Indeed, $\ell_0^0$ equals the same as $\ell_0$ in the proof of Lemma
\ref{l:using-Df} with $\nabla f$ instead of $f$ and $(v-v')K(v,v')$
instead of $K(v,v')$. The fact that these are vector valued functions
does not affect the proof. Note that $(v-v') K(v,v')$ satisfies
\eqref{e:Kabove-signed} with $s-1/2$ instead of $s$. The second inequality is an elementary interpolation.

The cancellation assumption~\eqref{e:Kcancellation1G} says that 
\begin{align*}
 \| \ell_0^1 \|_{L^2} &\le \Lambda (R^{1-2s}+1) \| \nabla f \|_{L^2} \\
& \le \Lambda R^{1-2s}  \|f\|_{L^2}^\frac12 \|D^2 f \|_{L^2}^\frac12 + \Lambda \|\nabla f \|_{L^2}. 
\end{align*}
In order to estimate $\ell_0^2$, we use Lemma~\ref{l:multi-path}. We write 
\begin{align*}
\|\ell_0^2\|_{L^2}^2 &= \int_{\R^d} \left( \int_{B_R(v)} \left( f(v')-f(v)-\frac{\nabla f(v) + \nabla f(v')}2 \cdot (v'-v)  \right) K(v,v') \dd v' \right)^2 \dd v, \\
&\lesssim \int_{\R^d} \left( \int_{B_R(v)} \int_{B_r(m)} |D^2f(w))| \, \frac{|v'-v|^{d-2} K(v,v')}{|w-v|^{d-2} |w-v'|^{d-2}} \dd w \dd v' \right)^2 \dd v, \\
\intertext{using Fubini's theorem}
&= \int_{\R^d} \left( \int_{B_R(v)} |D^2 f(w)| \left( \int_{\{v' : (v'-v)\cdot(w-v) \geq |w-v|^2\}}  \frac{|v'-v|^{d-2} K(v,v')}{|w-v|^{d-2} |w-v'|^{d-2}} \dd v' \right) \dd w \right)^2 \dd v.
\end{align*}
In view of the definition of $A(v,w)$, we can use \eqref{e:Akernel-bis} and get
\begin{align*}
\|\ell_0^2\|_{L^2}^2 &\leq \int_{\R^d} \left( \int_{B_R} |D^2 f(w)| A(v,w) \dd w \right)^2 \dd v, \\
&\leq \int_{\R^d} \left( \int_{B_R} A(v,w) \dd w \right) \left( \int_{B_R} |D^2f(w)|^2 A(v,w) \dd w \right) \dd v, \\
&\leq C R^{2-2s} \int_{\R^d} |D^2f(w)|^2 \left( \int_{|v-w|<\R} A(v,w) \dd v \right) \dd w \leq C R^{2(2-2s)} \|D^2 f\|_{L^2}^2.
\end{align*}
Choosing $R = \|f\|_{L^2}^\frac12 / \|D^2 f\|_{L^2}^\frac12$ completes the proof.
\end{proof}

We can now prove the main result of this section. 
\begin{proof}[Proof of Theorem~\ref{t:upperbound}]

  We prove the upper bound applying Lemmas~\ref{l:using-Df} and
  \ref{l:using-D2f} to both operators $\Lv$ and its adjoint $\Lv^t$, and doing some
  sort of interpolation. Note that $\Lv^t$ has the same form as $\Lv$
  plus a correction which is bounded from $L^2$ to $L^2$ (thanks to
  the cancellation assumption~\eqref{e:Kcancellation0G}), so Lemmas
  \ref{l:using-Df} and \ref{l:using-D2f} apply to $\Lv^t$ as well. Indeed,
\[ \Lv^t f(v) = \int_{\R^d} (f(v') - f(v)) K(v',v) \dd v' + \left( \int_{\R^d} K(v',v) - K(v,v') \dd v' \right) f(v).\]

  The following interpolation is probably classical. We prove it using
  Littlewood-Paley theory.  Since we have already obtained the
  estimate for $\Esym$ in Lemma \ref{lem:upper-sym}, we are only left
  to prove the estimate for $\Eskew$.  In the case $s \in (0,1/2)$,
  the proof below gives the estimate for $\Ex$ right away. For
  $s \in [1/2,1)$ the proof below applies to $\Eskew$ only.

  Let $\Delta_i$ be the Littlewood-Paley projectors. We use the
  convention that all low modes are enclosed in $\Delta_0$. That is
  $f = \sum_{i=0}^\infty \Delta_i f$, with the index $i$ being
  non-negative. We use the fact that for any $s \geq 0$,
\[ \|\Delta_i f\|_{H^s} \approx 2^{is} \|\Delta_i f\|_{L^2} \]
Moreover, from Lemma \ref{l:using-Df}, if $s \in (0,1/2)$,
\[
\|\Lv \Delta_i f\|_{L^2} \lesssim \|\Delta_i f\|_{L^2}^{\frac{1-2s}{1+2s}} \|\Delta_i f\|_{\dot H^{s+1/2}}^{\frac{4s}{1+2s}} \lesssim 2^{s i} \|\Delta_i f\|_{H^s}
\]
From Lemma \ref{l:using-D2f}, if $s \in [1/2,1)$,
\[
\|\Lv \Delta_i f\|_{L^2} \lesssim \|\Delta_i f\|_{L^2}^{1-s}  \|\Delta_i f\|_{H^2}^{s} + \|\Delta_i f\|_{H^1} \lesssim 2^{s i} \|\Delta_i f\|_{H^s}
\]
The same estimates hold for $\Lv^t$ in the place of $\Lv$.

Therefore,
\begin{align*} 
\Ex(f,g) &= \sum_{ij} \Ex(\Delta_i f, \Delta_j g), \\
&= \sum_{i \leq j} \langle \Lv \Delta_i f, \Delta_j g \rangle + \sum_{i > j} \langle \Delta_i f,L^t \Delta_j g \rangle, \\
&\lesssim \sum_{i , j} 2^{-s |i-j|} \|\Delta_i f\|_{H^s} \|\Delta_j g\|_{H^s}, \\
&= \sum_{k=0}^\infty 2^{-sk} \sum_{i=0}^\infty \|\Delta_i f\|_{H^s} \|\Delta_{i+k} g\|_{H^s} 
+ \|\Delta_{i+k} f\|_{H^s} \|\Delta_i g\|_{H^s}, \\
&\leq \sum_{k=0}^\infty 2^{-sk+1} \left( \sum_{i=0}^\infty \|\Delta_i f\|_{H^s}^2 \right)^{1/2} 
\left( \sum_{i=0}^\infty \|\Delta_j g\|_{H^s}^2 \right)^{1/2}, \\
&\lesssim \|f\|_{H^s} \|g\|_{H^s}.
\end{align*}
The proof is now complete.
\end{proof}

\subsection{A generalized cancellation lemma}

As a preparation for the next subsection, we prove the following generalized cancellation lemma. 
\begin{lemma}[Generalized cancellation] \label{lem:g-cancellation} Let
  $K$ be a kernel satisfying \eqref{e:Kabove-signed}; if $s \ge 1/2$, we also assume that $K$
  satisfies \eqref{e:Kcancellation1G}. Let $\varphi$ be a bounded
  $C^2$ function. Then
\[PV 
\int_{\R^d} (\varphi(v') - \varphi(v)) [K(v,v') - K (v',v)] \dd v' 
\leq C \|\varphi\|_{C_2},
\]
for some constant $C$ depending on $\Lambda$ and dimension.
\end{lemma}
\begin{proof}
  The proof is a direct computation. We estimate the tail of the
  integral using \eqref{e:Kabove-signed} together with
  the boundedness of $\varphi$. Then, we estimate the integral in
  $B_1$ using \eqref{e:Kcancellation1G} and the smoothness of
  $\varphi$. We write the proof for the case $2s \geq 1$ first, and
  later indicate its simplification when $2s <1$.
\begin{align*}
PV \int (\varphi(v') &- \varphi(v)) [K(v,v') - K (v',v)] \dd v' \\
\leq & PV \int_{B_1} (\varphi(v') - \varphi(v)) [K(v,v') - K (v',v)] \dd v' + C \Lambda \|\varphi\|_{L^\infty}, \\
\leq &PV \int_{B_1} (v'-v) \nabla \varphi(v) [K(v,v') - K (v',v)] \\
& + \|D^2 \varphi\|_\infty  |v-v'|^2 |K(v,v') - K (v',v)| \dd v' + C \Lambda \|\varphi\|_{L^\infty}, \\
\leq &C \Lambda \|\varphi\|_{C^2}.
\end{align*}
For the last inequality we used that thanks to \eqref{e:Kabove-signed}, 
\[ \int_{B_1} |v'-v|^2 [K(v,v') - K (v',v)] \dd v' \lesssim \Lambda,\]
and thanks to \eqref{e:Kcancellation1G},
\[PV \int_{B_1}  (v'-v) [K(v,v') - K (v',v)] \dd v' \leq \Lambda.\]

When $s < 1/2$, we do not need to use \eqref{e:Kcancellation1G}. We simply use \eqref{e:Kabove-signed}
to get
\begin{align*} 
  \int_{B_1} (\varphi(v') - \varphi(v)) [K(v,v') - K (v',v)] \dd v' 
&\leq \int_{B_1} |v-v'| [\varphi]_{C^1} [K(v,v') - K (v',v)] \dd v' , \\
&\leq C \Lambda [\varphi]_{C^1}. \qedhere
\end{align*}
\end{proof}

\begin{remark}
  Lemma \ref{lem:g-cancellation} tells us in particular
  that when $K$ is anti-symmetric, the operator $\Lv f$ is well
  defined pointwise. The same cannot be said for a symmetric
  kernel of $s \geq 1/2$. When $K$ is a symmetric kernel assuming only
  \eqref{e:Kabove-signed}, the value of $\Lv f(v)$ is not necessarily
  defined pointwise, even if $f$ is smooth. It is only through $\Esym$
  that we can define $\Lv$ as an operator from $H^s$ to $H^{-s}$.
\end{remark}

\subsection{Estimate focusing on the smoothness of only one function}

In this section we obtain an estimate for $\Ex(\varphi,g)$ taking
maximum advantage of the smoothness of $\varphi$, and not so much on
the smoothness of $g$.

\begin{lemma}[Second upper bound for $\Ex$] \label{lem:upper-2} Let
  $K$ satisfy \eqref{e:Kabove-signed} and
  \eqref{e:Kcancellation0G}. If $s \geq 1/2$, we also assume
  \eqref{e:Kcancellation1G}. For any two functions $g \in H^s (\R^d) \cap L^1 (\R^d)$
  and $\varphi \in C^2$ with $g \ge 0$ and any $\eps>0$, we have
\begin{equation}\label{e:Kupper2}
\Ex(\varphi,g)
\le \eps \|g\|_{\dot H^s}^2 + C \eps^{-1}
  \|\varphi\|_{C^1}^2 |\{v \in \R^d : g(v) >0 \}| + C \|\varphi\|_{C^2} \|g\|_{L^1}.
\end{equation}
\end{lemma}
\begin{proof}
Recall that $\Ex =  \Esym +  \Eskew$. We estimate each term separately.

In order to estimate $\Esym$, we apply the following elementary identity
\[ |g(v) - g(v')| \leq (\chi_{g>0}(v) + \chi_{g>0}(v')) |g(v) - g(v')|;\]
we then get
\begin{align*} 
 \Esym (\varphi,g) &\leq  \iint |\varphi(v) - \varphi(v')|\chi_{g>0}(v) |g(v) - g(v')| K(v,v') \dd v' \dd v, \\
&\leq \eps \iint (g(v) - g(v'))^2 K(v,v') \dd v' \dd v 
+  (4\eps)^{-1} \iint (\varphi(v) - \varphi(v'))^2 \chi_{g>0}(v) K(v,v') \dd v' \dd v, \\
&= \eps \Esym(g,g) + (4\eps)^{-1} \int \chi_{g>0}(v) \left( \int (\varphi(v) - \varphi(v'))^2 K(v,v') \dd v' \right) \dd v,\\
\intertext{using Lemma \ref{lem:upper-sym} and the assumption \eqref{e:Kabove-signed},}
&\leq \eps C \|g\|_{\dot H^s}^2 + C\eps^{-1} \|\varphi\|_{C^1}^2  \int \chi_{g>0} \dd v .
\end{align*} 

As far as $\Eskew$ is concerned, we first rewrite it as follows
\begin{align*}
 \Eskew(\varphi,g)  =& \frac14 \iint (\varphi(v) - \varphi(v')) (g(v)+g(v'))  (K(v,v')-K(v',v)) \dd v' \dd v   \\
 =& \frac12\iint (\varphi(v) - \varphi(v')) g(v)  \left( K (v,v')-K(v',v)  \right)\dd v'\dd v  \\
=& \frac12\int g(v) \left\{ PV \int_{\R^d} (\varphi(v) - \varphi(v'))  \left( K (v,v')-
K(v',v)  \right)\dd v' \right\} \dd v , \\
\intertext{using Lemma \ref{lem:g-cancellation},}
&\leq C \|\varphi\|_{C^2} \int  g(v) \dd v.
\end{align*}
Combining the upper bounds for $\Esym$ and $\Eskew$, we conclude the proof.
\end{proof}

\subsection{Commutator estimates}

\begin{lemma}[Commutator estimate for $s \in (0,1/2)$] \label{l:commutator-s<1/2} Let us assume $s \in (0,1/2)$
  and that $K$ satisfies \eqref{e:Kabove-signed}.  Let $D$ be a closed set and $\Omega$ open so that
  $D \Subset \Omega \subset \R^d$. Let $\varphi$ be a smooth function
  supported in $D$ and $f \in H^s(\Omega) \cap L^\infty(\R^d)$. We
  have the following commutator estimate
\[ \Lv [\varphi f] - \varphi \Lv f= h_1 + h_2,\]
with
\begin{align*}
\|h_1\|_{L^\infty(\R^d)} &\lesssim \|\varphi\|_{L^\infty} \|f\|_{L^\infty(\R^d)} d(D,\R^d\setminus \Omega)^{-2s}, \\
\|h_1\|_{L^2(\R^d \setminus \Omega)} &\lesssim \|\varphi\|_{L^\infty}  \|f\|_{L^2(D)} d(D,\R^d\setminus \Omega)^{-2s}, \\
\| h_2\|_{L^2(\R^d)} &\lesssim \|\varphi \|_{C^1} \|f\|_{L^2(\Omega)}.
\end{align*}
Moreover, $h_2 = 0$ outside $\Omega$. Whenever $\Omega = \R^d$, we can consider $d(D,\R^d\setminus \Omega) = +\infty$ and $h_1 = 0$.
\end{lemma}
\begin{proof}
From the formula \eqref{e:Lv}, we get
\[ C[\varphi,f](v) := \Lv [\varphi f](v) - \varphi(v) \Lv f(v) = \int_{\R^d} f(v') (\varphi(v') - \varphi(v)) K(v,v') \dd v'.\]

Let $r = d {(D,\R^d \setminus \Omega)}/2$, and let
$E = D + B_r$. Thus, we have $D \Subset E \Subset \Omega$, with
$d(D,\R^d \setminus E) = r$ and $d(D, \R^d \setminus \Omega) = r$.

We define
\[ h_1(v) = \int_{\R^d \setminus B_r(v)} f(v') (\varphi(v') - \varphi(v)) K(v,v') \dd v', \qquad h_2(v) = \int_{B_r(v)} f(v') (\varphi(v') - \varphi(v)) K(v,v') \dd v'.\]

From \eqref{e:Kabove-signed}, for any value of $v \in \R^d$, we have
\[ |h_1(v)| \leq 2\|f\|_{L^\infty} \|\varphi\|_{L^\infty}  \Lambda r^{-2s}\]
which is the first inequality. 

When $v \notin D$, we have $\varphi(v) = 0$. Therefore, the integrand in $C[\varphi,f](v)$ is nonzero only for $v' \in D$. We thus have for $v \notin \Omega \supset D$,
\[ h_1(v) = \int_{D} f(v') \varphi(v')  K(v,v') \dd v'\]
Therefore
\begin{align*} 
 \int_{\R^d \setminus \Omega} h_1(v)^2 \dd v &= \int_{\R^d \setminus \Omega} \left( \int_D f(v') \varphi(v')  K(v,v') \dd v'\right)^2 \dd v,\\
&\leq \|\varphi\|_{L^\infty}^2 \int_{\R^d \setminus \Omega} \left( \int_D f(v')^2 |K(v,v')| \dd v'\right)\left( \int_D |K(v,v')| \dd v'\right) \dd v, \\
\intertext{using \eqref{e:Kabove-signed},}
&\leq \|\varphi\|_{L^\infty}^2  \Lambda r^{-2s} \int_{\R^d \setminus \Omega} \int_D f(v')^2 |K(v,v')| \dd v' \dd v, \\
&\leq \|\varphi\|_{L^\infty}^2 \Lambda r^{-2s} \int_D f(v')^2 \left( \int_{|v-v'| > r} |K(v,v')| \dd v \right) \dd v = \Lambda^2 r^{-4s} \|\varphi\|_{L^\infty}^2 \|f\|_{L^2(D)}^2.
\end{align*}
This gives us the second inequality.

In order to estimate $\|h_2\|_{L^2}$, we use Cauchy Schwarz.
\begin{align*}
  \|h_2\|_{L^2}^2 &= \int_E \left( \int_{B_r(v)} f(v') (\varphi(v') - \varphi(v)) K(v,v') \dd v' \right)^2 \dd v, \\
                  &\leq \int_E \left( \int_{B_r(v)} f(v')^2 |\varphi(v') - \varphi(v)|\, |K(v,v')| \dd v' \right) \left( \int_{B_r(v)} |\varphi(v') - \varphi(v)| \, |K(v,v')| \dd v' \right) \dd v, \\
\end{align*}

Since $\varphi$ is bounded and $C^1$, then \eqref{e:Kabove-signed} implies (note that $s < 1/2$), 
\begin{align*} 
 \int_{\R^d} |\varphi(v') - \varphi(v)| \, |K(v,v')| \dd v' \lesssim \|\varphi\|_{C^1} \qquad \text{for every value of } v \in \R^d, \\
 \int_{\R^d} |\varphi(v') - \varphi(v)| \, |K(v,v')| \dd v \lesssim \|\varphi\|_{C^1} \qquad \text{for every value of } v' \in \R^d.
\end{align*}
Therefore,
\begin{align*}
  \|h_2\|_{L^2}^2 &\lesssim  \|\varphi\|_{C^1} \int_E \left( \int_{B_r(v)} f(v')^2 |\varphi(v') - \varphi(v)|\, |K(v,v')| \dd v' \right) \dd v, \\
                  &\lesssim  \|\varphi\|_{C^1} \int_\Omega f(v')^2 \left( \int_{E \cap B_r(v')} |\varphi(v') - \varphi(v)|\, |K(v,v')| \dd v \right) \dd v',\\
                  &\lesssim \|\varphi\|_{C^1}^2 \|f\|_{L^2(\Omega)}^2.
\end{align*}

\end{proof}

\begin{lemma}[Commutator estimate for $s \in
  [1/2,1)$]
  \label{l:commutator-s>1/2} Let us assume $s \in [1/2,1)$ and that
  $K$ satisfies \eqref{e:Kabove-signed} and \eqref{e:Kcancellation1G}.
  Let $D$ be a closed set, and $\Omega$ open so that $D \Subset \Omega \subset \R^d$. Let $\varphi$ be a smooth
  function supported in $D$ and
  $f \in H^s(\Omega) \cap L^\infty(\R^d)$. We have the following
  commutator estimate
\[ \Lv [\varphi f] - \varphi \Lv f = h_1 + h_2 + (-\Delta)^{s/2} h_3,\]
with
\begin{align*}
\|h_1\|_{L^\infty(\R^d)} &\lesssim \|\varphi\|_{L^\infty} \|f\|_{L^\infty(\R^d)} (d(D,\R^d\setminus \Omega)+d(v,D))^{-2s}, \\
\|h_1\|_{L^2(\R^d \setminus \Omega)} &\lesssim \|\varphi\|_{L^\infty} \|f\|_{L^2(D)} d(D,\R^d\setminus \Omega)^{-2s}, \\
\| h_2\|_{L^2(\R^d)} &\lesssim \|\varphi \|_{C^2} \|f\|_{H^s(\Omega)}, \\
\| h_3\|_{L^2(\R^d)} &\lesssim \|\varphi \|_{C^2} \|f\|_{L^2(\Omega)}.
\end{align*}
Moreover, $h_2 = 0$ outside $\Omega$. Whenever $\Omega = \R^d$, we can
consider $d(D,\R^d\setminus \Omega) = +\infty$ and $h_1 = 0$.
\end{lemma}

\begin{proof}
  We define $h_1$ and $\tilde h_2$ by the expressions of $h_1$ and
  $h_2$ in the proof of Lemma \ref{l:commutator-s<1/2}. The estimates
  for $h_1$ follow identically. We will split
  $\tilde h_2 = h_2 + (-\Delta)^{s/2} h_3$, and need to prove the
  estimate for each term.

Note that, by construction, $\tilde h_2(v) = 0$ for any $v \notin E$.

Let us write $K$ as the sum of its symmetric plus antisymmetric parts:
$K = K_s + K_a$. We start by estimating the antisymmetric
contribution.

Because of Lemma \ref{lem:g-cancellation}, we have that
$\|\Lv^a \varphi \|_{L^\infty} \lesssim \|\varphi\|_{C^2}$. Then
\begin{align*}
\|\tilde h_2^a\|_{L^2(\Omega)} &:= \left\vert \int_{B_r(v)} f(v') (\varphi(v') - \varphi(v)) K_a(v,v') \dd v' \right\vert_{L^2(\Omega)}, \\
&\leq\left\vert \int_{B_r(v)} (f(v')-f(v)) (\varphi(v') - \varphi(v)) K_a(v,v') \dd v' \right\vert_{L^2(\Omega)} + C \|\varphi\|_{C^2} \|f\|_{L^2(\Omega)} .
\end{align*}

With respect to the first term, we apply Cauchy-Schwarz and Lemma \ref{lem:upper-sym} to obtain
\begin{align*}
\bigg\vert \int_{B_r(v)} (f(v')&-f(v)) (\varphi(v') - \varphi(v)) K_a(v,v') \dd v' \bigg\vert_{L^2(\Omega)}^2  \\
&\leq \int_{E} \left( \int_{B_r(v)} (f(v')-f(v))^2 |K_a(v,v')| \dd v' \right) \left( \int_{B_r(v)} (\varphi(v')-\varphi(v))^2 |K_a(v,v')| \dd v' \right) \dd v, \\
&\lesssim \|\varphi\|_{C^2} \iint_{\Omega \times \Omega} (f(v')-f(v))^2 |K_a(v,v')| \dd v'\dd v \lesssim \|\varphi\|_{C^2}  \|f\|_{\dot H^s(\Omega)}^2.
\end{align*}

Therefore, we conclude the estimate for the antisymmetric contribution
$\|\tilde h_2^a\|_{L^2(\R^d)} \leq C \|\varphi\|_{C^2}
\|f\|_{H^s(\Omega)}$.

Now we need to analyse the contribution of $K_s$ to $\tilde h_2$,
which we call $\tilde h_2^s$. We estimate it by duality. Let
$g \in H^s(\R^d)$, recall that $\tilde h_s^2$ is supported in $E$ and
consider
\begin{align*} 
 \int_E \tilde h_2^s(v) g(v) \dd v &= \int_E \int_{B_r(v)} g(v) f(v') (\varphi(v') - \varphi(v)) K_s(v,v') \dd v' \dd v , \\
&= \frac 12 \int_E f(v) \left( \int_{B_r(v)}  (g(v) - g(v'))  (\varphi(v') - \varphi(v)) K_s(v,v') \dd v' \dd v \right) \\
&\phantom{=} + g(v) \left( \int_{B_r(v)}  (f(v') - f(v))  (\varphi(v') - \varphi(v)) K_s(v,v') \dd v' \dd v \right) \dd v.
\end{align*}
Applying the Cauchy-Schwarz inequality and Lemma \ref{lem:upper-sym} as above, we get
\[ \int_\Omega \tilde h_2^s(v) g(v) \dd v \lesssim \|\varphi\|_{C^2}
\left( \|f\|_{\dot H^s}\|g\|_{L^2} + \|f\|_{L^2} \|g\|_{\dot H^s}
\right).\]
Therefore, $\tilde h_2^s$ can be written as a sum
$\hat h_2^s + (-\Delta) h_3$ with
$\| \hat h_2^s\|_{L^2(\R^d)} \lesssim \|\varphi\|_{C^2}
\|f\|_{H^s(\Omega)}$
and
$\| h_3\|_{L^2(\R^d)} \lesssim \|\varphi\|_{C^2} \|f\|_{L^2(\Omega)}$.

We finish the proof by letting $h_2 = \tilde h_2^a + \hat h_2^s$.
\end{proof}

\section{Reduction to global kernels and weak solutions}
\label{sec:reduction}

The assumptions \eqref{e:Klower}, \eqref{e:Kabove},
\eqref{e:Kcancellation0} and \eqref{e:Kcancellation1} are given in
terms of values of $v \in B_\Radius$ only. It is natural that if we consider the equation
\eqref{e:main} to hold for $v \in B_1$ and we intend to prove local
regularity estimates, it should be useless to make assumptions for
$K(v,v')$ when $v \notin B_2$. It is confortable for the proofs of a
few lemmas (in particular the results in Section \ref{sec:bilinear}
above and Lemma \ref{lem:gee} below) to have a kernel that is globally
defined and satisfies all these assumptions for all values of $v$ and
$v'$. In this section we explain how to extend a kernel to the full
space in order to have that.

\subsection{Reduction to global kernels}

\begin{prop}[A kernel defined globally]\label{p:global-kernel}
  Assume that $K: B_\Radius \times \R^d \to \R$
  satisfies \eqref{e:Klower}, \eqref{e:Kabove},
  \eqref{e:Kcancellation0} and \eqref{e:Kcancellation1}.  There exists a kernel
  $\tilde K: \R^d \times \R^d \to \R$  satisfying the following global version of assumptions
  \eqref{e:Klower}, \eqref{e:Kabove}, \eqref{e:Kcancellation0} and
  \eqref{e:Kcancellation1}.
\begin{itemize}
\item $\tilde K(v,v') = K(v,v')$ whenever $v$ and $v'$ belong to $B_{2\Radius/3}$. Moreover $\tilde K(v,v') \geq 0$ for all $v,v' \in \R^d$ and for all $v \in B_{\Radius/2}$,
\begin{equation} \label{e:Kerrortail}
 \int_{\R^d} |K(v,v') - \tilde K(v,v')| \dd v' \leq C \Lambda.
\end{equation}
\item For any function $f \in H^s(\R^d)$,
\begin{equation}
\label{e:KlowerG} \lambda 
\|f\|_{\dot H^s}^2
\leq - \int_{\R^d} \tilde{\Lv} f(v) \,  f(v)  \dd v  + \Lambda \|f\|_{L^2(\R^d)}^2.
\end{equation}
Here $\tilde \Lv$ is the integro-differential operator corresponding
to the kernel $\tilde K$.
\item The assumptions \eqref{e:Kabove-signed},
  \eqref{e:Kcancellation0G} and \eqref{e:Kcancellation1G} hold for
  $\tilde K$ with a constant $C \Lambda$ instead of $\Lambda$, where
  $C$ depends on $s$, $\bar R$, and dimension only.
\end{itemize}
\end{prop}
\begin{proof}
  Let $\eta: \R^d \to [0,1]$ be a smooth radial function so that
  $\eta = 1$ in $B_{3\Radius/4}$ and $\eta = 0$ outside $B_{7\Radius/8}$. We define
\[ \tilde K(v,v') = \eta(v) \eta(v') K(v,v') + \Lambda (1-\eta(v) \eta(v')) |v-v'|^{-d-2s}.\]

Note that even though $K(v,v')$ is not defined when $v \notin B_\Radius$,
since we have the factor $\eta(v)=0$ there, there is no ambiguity in
the definition of $\tilde K(v,v')$.

The first item in the Proposition is obvious by construction. We start by checking \eqref{e:Kabove-signed}. For any $v \in \R^d$ and
$r>0$, we have
\begin{align*} \int_{\R^d \setminus B_r(v)} \tilde K (v,v') \dd v' &=\int_{\R^d \setminus B_r(v)} \eta(v) \eta(v') K (v,v') +  \Lambda (1-\eta(v) \eta(v')) |v-v'|^{-d-2s} \dd v', \\
&\leq \eta(v) \int_{\R^d \setminus B_r(v)} K(v,v') \dd v' + (1-\eta(v)) \Lambda \int_{\R^d \setminus B_r(v)} |v-v'|^{-d-2s} \dd v' \lesssim \Lambda.\end{align*}

For any $v' \in \R^d$ and $r>0$, we do almost the same computation
\begin{align*} 
\int_{\R^d \setminus B_r(v')} \tilde K (v,v') \dd v &=\int_{\R^d \setminus B_r(v)} \eta(v) \eta(v') K (v,v') +  \Lambda (1-\eta(v) \eta(v')) |v-v'|^{-d-2s} \dd v', \\
&\leq \eta(v') \int_{B_\Radius \setminus B_r(v')} K(v,v') \dd v + (1-\eta(v')) \Lambda \int_{\R^d \setminus B_r(v')} |v-v'|^{-d-2s} \dd v \lesssim \Lambda.
\end{align*}

This justifies \eqref{e:Kabove-signed}. We now verify \eqref{e:Kcancellation1G}. This only applies when
$2s \geq 1$. Given any $r > 0$, we compute
\begin{align*}
\bigg \vert PV \int_{B_r(v)}
   & (v-v') \big( \tilde K(v,v') -  \tilde K (v',v) \big) \dd v' \bigg \vert \\
& = \eta(v) \left \vert PV \int_{ B_r(v) \cap B_\Radius} 
    (v-v') \eta(v') \left( K(v,v') - K (v',v)\right) \dd v' \right \vert,\\ 
&\leq \eta(v) \Bigg( \eta(v) \Lambda (1+ \min(r, (\Radius - |v|))^{1-2s} ) \\ 
& \qquad  + \left \vert \int_{  B_r(v) \cap B_\Radius}   (v-v') (\eta(v')-\eta(v)) \left( K(v,v') - K (v',v)\right) \dd v' \right \vert \Bigg), \\ 
\intertext{Note that $\eta(v)=0$ if $|v| > 7\Radius/8$, therefore $\eta(v) (\Radius - |v|)^{1-2s} \leq C$ for some constant depending on $\Radius$.} 
&\leq C\eta(v) \left( \Lambda (1 + r^{1-2s}) + \int_{ B_r(v) \cap B_\Radius}  
    |v-v'|^2 | K(v,v') - K (v',v)| \dd v' \right),  \\
& \leq C \Lambda (1+r^{1-2s} ). 
\end{align*}
This proves \eqref{e:Kcancellation1G}.

We now move on to \eqref{e:Kcancellation0G}. When $s \in [0,1/2)$ the
proof is similar to the computation above for \eqref{e:Kcancellation1G}. Indeed,
\begin{align*}
\bigg\vert PV \int_{\R^d} \left( \tilde K(v,v') - \tilde K (v',v)\right) & \dd v' \bigg \vert = \eta(v) \left \vert PV \int_{B_\Radius} \eta(v') \left(  K(v,v') - K (v',v)\right) \dd v' \right \vert,\\ 
&\leq \eta(v) \left( \Lambda + \left \vert PV \int_{B_\Radius} (\eta(v')-\eta(v)) \left(  K(v,v') - K (v',v)\right) \dd v' \right \vert \right),\\ 
&\leq \eta(v) \left( \Lambda + C \int_{B_\Radius} |v'-v| | K(v,v') - K (v',v)|) \dd v' \right) \leq C \Lambda \eta(v).
\end{align*}
The last inequality follows from \eqref{e:Kabove} because $s \in [0,1/2)$.

In the case $s \in [1/2,1)$, we modify the estimate of the last line. We have
\begin{align*}
\bigg\vert PV \int_{\R^d} \left( \tilde K(v,v') - \tilde K (v',v)\right) & \dd v' \bigg \vert
\leq \eta(v) \left( \Lambda + \left \vert PV \int_{B_\Radius} (\eta(v')-\eta(v)) \left(  K(v,v') - K (v',v)\right) \dd v' \right \vert \right),\\ 
&\leq \eta(v) \bigg( \Lambda + \left \vert \nabla \eta(v) \cdot PV \int_{B_\Radius}  (v'-v) \left(  K(v,v') - K (v',v)\right) \dd v' \right \vert \\
& \qquad \qquad +\int_{B_\Radius} C|v'-v|^2 |K(v,v') - K (v',v)| \dd v' \bigg) \leq C \Lambda \eta(v).
\end{align*}
For the last inequality, we apply \eqref{e:Kabove} and \eqref{e:Kcancellation1}.

We now justify \eqref{e:KlowerG}. We see that
\begin{align*}
- \int \tilde \Lv f(v) f(v) \dd v &= \tilde \Esym(f,f) + \tilde \Eskew(f,f), \\
&= \iint |f(v) - f(v')|^2 \tilde K(v,v') \dd v \dd v' + \int f(v)^2 \left( PV \int (\tilde K(v,v') - \tilde K(v',v)) \dd v' \right) \dd v, \\
&\geq \iint_{\R^{2d}} |f(v) - f(v')|^2 \tilde K(v,v') \dd v \dd v' -  \Lambda \|f\|_{L^2}^2.
\end{align*}

Let $3\Radius/4 < r_1 < r_2 < \Radius$ so that $\eta(v) < 2/3$ if $|v| > r_1$ and $\eta(v) > 1/3$ if $|v| < r_2$. The first term in the definition of $\tilde K(v,v')$ of bounded below by $K(v,v')/9$ when both $v$ and $v'$ belong to $B_{r_2}$. When $v$ and $v'$ do not belong to $B_{r_1}$,  we can estimate $\tilde K(v,v')$ from below by $\Lambda |v-v'|^{-d-2s}/3$. If $v$ and $v'$ belong to $B_{r_2} \setminus B_{r_1}$, the value of $\tilde K(v,v')$ is bounded below by the sum of the two previous terms. We have,
\begin{align*}
\iint_{\R^{2d}} |f(v) - f(v')|^2 \tilde K(v,v') \dd v \dd v' &\geq \frac 19 \iint_{B_{r_2} \times B_{r_2}} |f(v) - f(v')|^2 K(v,v') \dd v \dd v' \\
&\qquad \qquad + \frac \Lambda 3 \iint_{\R^{2d} \setminus (B_{r_1} \times B_{r_1})} |f(v) - f(v')|^2 |v-v'|^{-d-2s} \dd v \dd v'.
\end{align*}

We need to estimate the first term using \eqref{e:Klower}. Let $\varphi$ be a smooth radial function so that $\varphi = 1$ in $B_{r_1}$ and $\varphi=0$ outside $B_{r_2}$.  Using Lemma \ref{lem:g-cancellation} after some arithmetic manipulations, we see that
\begin{align*} 
 \E(\varphi f, \varphi f) &= \iint_{B_\Radius \times B_\Radius} \varphi(v) \varphi(v') (f(v)-f(v'))^2 K(v,v') \dd v' \dd v \\
&\qquad \qquad + 2 \int_{B_\Radius} f(v)^2 \varphi(v) \left( PV \int_{B_\Radius} (\varphi(v) - \varphi(v')) (K(v,v') - K(v',v)) \dd v' \right) \dd v, \\
&\qquad \qquad + 2 \int_{B_\Radius} \varphi(v)^2 f(v)^2 \left( \int_{\R^d \setminus B_\Radius} K(v,v') \dd v' \right) \dd v, \\
&\leq \iint_{B_{r_2} \times B_{r_2}} |f(v) - f(v')|^2 K(v,v') \dd v \dd v' + C \|f\|_{L^2}^2.
\end{align*}
Combining the last three displayed inequalities with \eqref{e:Klower}, we obtain
\begin{align*}
\E(f,f) &= - \int \tilde \Lv f(v) f(v) \dd v, \\
&\geq \iint_{\R^{2d}} |f(v) - f(v')|^2 \tilde K(v,v') \dd v \dd v' -  \Lambda \|f\|_{L^2}^2, \\
&\geq \frac 19 \iint_{B_{r_2} \times B_{r_2}} |f(v) - f(v')|^2 K(v,v') \dd v \dd v' \\
&\qquad \qquad + \frac \Lambda 3 \iint_{\R^{2d} \setminus (B_{r_1} \times B_{r_1})} |f(v) - f(v')|^2 |v-v'|^{-d-2s} \dd v \dd v' -  \Lambda \|f\|_{L^2}^2, \\
&\geq  \frac 19 \E(\varphi f, \varphi f) - C \|f\|_{L^2}^2 + \frac \Lambda 3 \iint_{\R^{2d} \setminus (B_{r_1} \times B_{r_1})} |f(v) - f(v')|^2 |v-v'|^{-d-2s} \dd v \dd v', \\
&\geq \min(\lambda/9,\Lambda/3) \left( \iint_{\R^d \times \R^d} |f(v) - f(v')|^2 |v-v'|^{-d-2s} \dd v \dd v' \right) - C \|f\|_{L^2}^2.
\end{align*}
\end{proof}

The extended kernel $\tilde K$ can be used to reduce many results to
the case of globally defined kernels. The following results, which we
will need later, are examples.
\begin{cor}[The operator $\Lv$ maps $H^s$ into
  $H^{-s}$] \label{c:upperbound} Assume $K: B_\Radius \times \R^d \to \R$ is
  a non-negative kernel that satisfies \eqref{e:Kabove} and
  \eqref{e:Kcancellation0}; if $s \ge 1/2$, we also assume that $K$
  satisfies \eqref{e:Kcancellation1}. For any $f \in H^s(\R^d)$ and
  $g \in H^s (\R^d)$ supported in $B_{\Radius/2}$,
\begin{equation} \label{e:Kupper}
\Ex (f,g) = -\int_{B_{\Radius/2}} \Lv f(v) g(v) \dd v \leq \Lambda \|f\|_{H^s(\R^d)} \|g\|_{H^s(\R^d)}
\end{equation}
for some positive constant $\Lambda$ depending on dimension.
\end{cor}

\begin{cor}[Second upper bound for $\Ex$] \label{cor:upper-2L} Let
  $K$ satisfy \eqref{e:Kabove}, \eqref{e:Kcancellation0}. If
  $s \geq 1/2$, we also assume \eqref{e:Kcancellation1}. For any two
  functions $g \in H^s (B_{\Radius/2}) \cap L^1 (\R^d)$ and $\varphi \in C^2$, both compactly supported in $B_{\Radius/2}$,
  with $g \ge 0$ and any $\eps>0$, we have
\begin{equation}\label{e:Kupper2L}
\Ex(\varphi,g)
\le \eps \|g\|_{H^s}^2 + C \eps^{-1}
  \| \nabla \varphi\|_{L^\infty}^2 |\{v \in \R^d : g(v) >0 \}| + C \|\varphi\|_{C^2} \|g\|_{L^1} + C \eps \|g\|_{L^2}^2.
\end{equation}
\end{cor}
{

\begin{cor} \label{c:commutator-s<1/2}
Let $K$ satisfy \eqref{e:Kabove}, $D$, $\Omega$, $\varphi$ and $f$ be as in Lemma \ref{l:commutator-s<1/2}. Assume that $B_{\Radius/2} \supset \Omega$. We extend the operator $\tilde L$ as in Proposition \ref{p:global-kernel}. Then, 
\[ \tilde L[\varphi f] - \varphi Lf = h_1 + h_2,\]
where $h_1$ and $h_2$ satisfy the same estimates as in Lemma \ref{l:commutator-s<1/2}.
\end{cor}

\begin{cor} \label{c:commutator-s>1/2}
Let $K$ satisfy \eqref{e:Kabove} and \eqref{e:Kcancellation1}, $D$, $\Omega$, $\varphi$ and $f$ be as in Lemma \ref{l:commutator-s>1/2}. Assume that $B_{\Radius/2} \supset \Omega$. We extend the operator $\tilde L$ as in Proposition \ref{p:global-kernel}. Then, 
\[ \tilde L[\varphi f] - \varphi Lf = h_1 + h_2 + (-\Delta)^{s/2} h_3,\]
where $h_1$, $h_2$ and $h_3$ satisfy the same estimates as in Lemma \ref{l:commutator-s>1/2}.
\end{cor}

The justifications of the two lemmas above are almost identical. We explain the latter one.

\begin{proof}[Proof of Corollary \ref{c:commutator-s>1/2}]
Let $\bar K$ be the extended kernel according to Proposition \ref{p:global-kernel}.

Applying Lemma \ref{l:commutator-s>1/2}, we obtain that
\[ \tilde L [ \varphi f] - \varphi \tilde L f = \tilde h_1 + h_2 + (-\Delta)^{s/2} h_3.\]

For this corollary, we want to replace $\varphi \tilde L f$ by $\varphi Lf$. Since $\varphi$ is supported in $D$, these two expressions only differ when $v \in D$. In this case, we have
\[ |\varphi(v) \tilde L f(v) - \varphi(v) Lf(v)| = \left\vert \varphi(v) \int_{\R^d} [f(v) - f(v')] \left( K(v,v') - \tilde K(v,v') \right) \dd v \right\vert \leq C \varphi(v) \|f\|_{L^\infty} \delta^{-2s}.\]
This difference is absorbed by the term $h_1$ by setting $h_1 = \tilde h_1 + \varphi(v) \tilde L f(v) - \varphi(v) Lf(v)$.
\end{proof}

} 

\subsection{Definition of weak solutions}

We now discuss the concept of weak solutions. In order to justify the
definition we are going to give below, we start with the following
preparatory lemma.
\begin{lemma}[The bilinear form $\E$ in the local
  case] \label{l:restricted-form} Let
  $\supp \varphi \Subset B_{\Radius/2}$  and $\varphi \in H^s (\R^d)$. Assume $K$ satisfies
  \eqref{e:Kabove}, \eqref{e:Kcancellation0} and
  \eqref{e:Kcancellation1}. Then for all $f \in L^\infty(\R^d \setminus B_{\Radius/2}) + H^s(\R^d)$
\[ \E(f,\varphi) \leq C \|f\|_{L^\infty(\R^d \setminus B_{\Radius/2}) + H^s(\R^d)} \|\varphi\|_{H^s(\R^d)},\]
where the constant $C$ depends on $\Lambda$, $d$, $s$ and the support of $\varphi$. Here,
\[\|f\|_{L^\infty(\R^d \setminus B_{\Radius/2}) + H^s(\R^d)} = \inf\left\{ \|f_1\|_{L^\infty(\R^d \setminus B_{\Radius/2})} + \|f_2\|_{H^s(\R^d)} : f = f_1 + f_2 \text{ and } f_1 = 0 \text{ in } B_{\Radius/2}\right\}.
\]
More precisely, the inequality holds for smooth functions, and
therefore it allows the bilinear form to be extended to the
appropriate spaces of functions.
\end{lemma}

Note that the restriction $f \in L^\infty(\R^d \setminus B_{\Radius/2}) + H^s(\R^d)$ imposes some fractional Sobolev regularity in $B_{\Radius/2}$ but not so much outside. In particular, any function $f \in H^s(B_{\Radius/2 + \eps}) \cap L^\infty(\R^d \setminus B_{\Radius/2+\eps})$ is in this space.

\begin{proof}
  As mentioned above, we assume for the proof that both $f$ and
  $\varphi$ are smooth. Afterwards, the inequality is obtained by
  density when $f \in L^\infty(\R^d \setminus B_{\Radius/2}) + H^s(\R^d)$ and
  $\varphi \in H^s(\R^d)$ is compactly supported in $B_{\Radius/2}$

Let $f = f_1 + f_2$ as in the definition of the norm in $L^\infty(\R^d \setminus B_{\Radius/2}) + H^s(\R^d)$. Applying Corollary \ref{c:upperbound},
$|\E(f_2, \varphi)| \lesssim \|f_2\|_{H^s} \|\varphi\|_{H^s}$.  We are left to compute $\E(f_1, \varphi)$. We have
\begin{align*}
\E(f_1, \varphi) &= \lim_{\eps \to 0} \iint_{|v'-v|>\eps} (f_1(v')  - f_1(v) ) \varphi(v) K(v,v') \dd v' \dd v, \\
&= \lim_{\eps \to 0} \int_{\supp \varphi} \left( \int_{\R^d \setminus B_\eps(v)} f_1(v') K(v,v')  \dd v'\right) \varphi(v) \dd v, \\
&= \int_{\supp \varphi} \left( \int_{\R^d \setminus B_\delta(v)} f_1(v') K(v,v')  \dd v'\right) \varphi(v) \dd v. \\
\intertext{Here $\delta$ is the distance between the support of $\varphi$ and $\R^d \setminus B_{\Radius/2}$.}
&\leq \Lambda \delta^{-2s} \|f_1\|_{L^\infty} \left( \int_{\supp \varphi} \varphi(v) \dd v \right) \leq  C \Lambda \delta^{-2s} \|f_1\|_{L^\infty} \| \varphi\|_{H^s}.
\end{align*}
\end{proof}

Another way to describe Lemma \ref{l:restricted-form} is that $\Lv$ is
a bounded operator from $L^\infty(\R^d \setminus B_{\Radius/2}) + H^s(\R^d)$ to
$H^{-s}(B_{\bar R})$. We will use this to define the concept of weak
solution.
\begin{defn}[Weak solutions] Assume $K$ satisfies \eqref{e:Kabove}, \eqref{e:Kcancellation0} and \eqref{e:Kcancellation1}.
Given the cylinder $Q = (0,T) \times B_{(\Radius/2)^{1+2s}} \times B_{\Radius/2}$, 
We say that a function
$f : [0,T] \times B_{(\Radius/2)^{1+2s}} \times \R^d \to \R$ is a \emph{subsolution} of
\eqref{e:main} in the cylinder $Q$ \label{weak}  if
\[ 
f \in C^0 ((0,T), L^2 (B_{(\Radius/2)^{1+2s}} \times B_{\Radius/2})) \cap L^2 ((0,T) \times B_{(\Radius/2)^{1+2s}}, 
L^\infty(\R^d \setminus B_{\Radius/2}) + H^s(\R^d) ),
\]
\[ 
f_t + v \cdot \nabla_x f \in L^2((0,T) \times B_{(\Radius/2)^{1+2s}} ,H^{-s}(B_{\Radius/2})), 
\]
and for all non-negative test function
$\varphi \in L^2 ((0,T) \times B_{(\Radius/2)^{1+2s}} ,
H^s (\R^d))$ so that for every $t$ and $x$, $\varphi(t,x\cdot)$ is compactly supported in $B_{\Radius/2}$,
\begin{equation} \label{e:weak-subsolution}
\iiint (f_t+ v\cdot \nabla_x f) \varphi + \iint \E(f,\varphi) - \iiint
h \varphi \le 0.
\end{equation}

A function $f$ is a \emph{supersolution} of \eqref{e:main} in $Q$ if $-f$ is a
subsolution of \eqref{e:main} in $Q$. A function $f$ is a \emph{solution} of \eqref{e:main} in $Q$ if it is
both a sub- and a supersolution.  
\end{defn}
\begin{remark}
Assuming that $f \in C^0 ((0,T), L^2 (B_{(\Radius/2)^{1+2s}} \times B_{\Radius/2}))$ and $f \in L^2 ((0,T) \times B_{(\Radius/2)^{1+2s}}, 
H^s (B_{\Radius/2}))$ is rather natural in view of the energy estimates one can easily get from the coercivity assumption.

\end{remark}

Note that the bilinear form $\iint \E(f,\varphi)$ in \eqref{e:weak-subsolution} is well defined because of Lemma \ref{l:restricted-form}.

\section{The first lemma of De Giorgi}
\label{sec:fdl}

This section is devoted to the first  intermediate result in the
proof of the weak Harnack inequality.  It is referred
to as the first lemma of De Giorgi. It consists in controlling a local
pointwise bound in the interior of a cylinder by an integral quantity in the
cylinder. Its proof (see Subsection~\ref{sub:1DG}) relies on a global
energy estimate (See Subsection~\ref{sub:global}). 

For degenerate integral equations, the situation is different than for
equations of second order. It is not true that the maximum of a
nonnegative subsolution can be bounded by above by a multiple of its $L^2$ norm. One needs to impose an extra global restriction (in this case we assume $0 \leq f \leq 1$ globally). This is because of nonlocal effects, since the positive values of the function outside of the domain of the equation may \emph{pull} the maximum upwards. The strong Harnack inequality fails in general. This fact is well documented and there are counterexamples (see \cite{bogdan2005}).

\subsection{Energy estimates}
\label{sub:global}

The proof of the first lemma of De Giorgi relies on an iteration of
energy estimates applied to certain truncated functions. For kinetic
equations, the energy estimate naturally gives us some regularization
with respect to the $v$ variable. We use the fractional Kolmogorov equation to
translate this regularization in $v$ to a higher degree of integrability of the function.

\begin{lemma}[Global energy inequality and gain of integrability]\label{lem:gee}
Assume $\tilde K$, and its corresponding
  operator $\tilde \Lv$, satisfy \eqref{e:KlowerG},
  \eqref{e:Kabove-signed}, \eqref{e:Kcancellation0G} and
  \eqref{e:Kcancellation1G}. Let $G \geq 0$ be a weak sub-solution of
\begin{equation}\label{e:ivp-bis} 
\begin{cases}
  (\partial_t + v \cdot \nabla_x) G - \tilde \Lv G \leq H_1 {+ (-\Delta)_v^{s/2} H_2} &  \text{ in } [0,T] \times \R^{2d}, \\
  G(0,x,v) = G_0(x,v) & \text{ in } \R^{2d}
\end{cases} 
\end{equation}
with a  source {terms $H_1 , H_2 \in L^2 ([0,T] \times \R^{2d})$.}
Then, 
\begin{equation}\label{e:nrj-G}
\sup_{\tau \in [0,T]} \| G(\tau)\|^2_{L^2 ( \R^{2d})} + \| G\|^2_{L^2([0,T]\times \R^d, \dot{H}^s(\R^{d}) )} \le C \left( \| G_0\|^2_{L^2(\R^{2d})} + \| H_1 \|^2_{L^2 ([0,T] \times \R^{2d})} {+ \| H_2 \|^2_{L^2 ([0,T] \times \R^{2d})} } \right).
\end{equation}
Moreover, there exists $p>2$ (only depending on dimension and $s$) such that 
\begin{equation}\label{e:gain}
\| G \|^2_{L^{p} ([0,T] \times \R^{2d})} \le C \left( \| G_0\|^2_{L^2(\R^{2d})} + \| H_1 \|^2_{L^2 ([0,T] \times 
\R^{2d})} {+ \| H_2 \|^2_{L^2 ([0,T] \times 
\R^{2d})}} \right),
\end{equation}
for some constant $C$ depending on $\lambda$, $\Lambda$, $d$, $p$, $s$, and $T$. 
\end{lemma}
\begin{proof}
Multiplying the equation by $G$ and integrating on the time interval
$[0,\tau]$ for $\tau \in [0,T]$, we get
\[
\frac12 \| G (\tau)\|^2_{L^2(\R^{2d})} + \int_{0}^\tau \int_{\R^d} \Ex(G,G) \dd x \dd t  \le \frac12 \| G_0\|^2_{L^2(\R^{2d})} + 
\int_{[0,T]\times \R^{2d}}  {(H_1+ (-\Delta)^{s/2}_v H_2)} G .
\]
Using \eqref{e:KlowerG} from Proposition~\ref{p:global-kernel}, we have
\begin{equation} \label{e:e1}
\begin{aligned}
\frac12 \| G (\tau)\|^2_{L^2(\R^{2d})} + \int_{0}^\tau \int_{\R^d} \lambda \|G\|_{\dot H^s}^2 - \Lambda \|G\|_{L^2}^2 \dd x \dd t  \le & \frac12 \| G_0\|^2_{L^2(\R^{2d})} \\
&+ { \int_0^T  \|H_1(t)\|_{L^2} \|G(t)\|_{L^2} + \|H_2(t)\|_{L^2} \|G(t)\|_{\dot H^s} \dd t }.
\end{aligned}
\end{equation}
{Therefore,
\[
\frac12 \| G (\tau)\|^2_{L^2(\R^{2d})} + \int_{0}^\tau \int_{\R^d} - \frac{\Lambda}2 \|G\|_{L^2}^2 \dd x \dd t  \le \frac12 \| G_0\|^2_{L^2(\R^{2d})} + 
C \int_0^T  \|H_1(t)\|_{L^2}^2 + \|H_2(t)\|_{L^2}^2\dd t .
\]
}

Integrating against $\exp(-\Lambda \tau/2)$ with respect to $\tau$ yields
\[
\| G\|^2_{L^2 ([0,T] \times \R^{2d})} 
\le C \left( \| G_0\|^2_{L^2(\R^{2d})} + \| H_1 \|^2_{L^2 ([0,T] \times 
\R^{2d})} + { \| H_2 \|^2_{L^2 ([0,T] \times 
\R^{2d})}} \right).
\]
Using this information back into \eqref{e:e1}, we finally get 
\begin{equation}\label{e:nrj-G-0}
\sup_{\tau \in [0,T]} \| G(\tau)\|^2_{L^2 ( \R^{2d})} + \| G\|^2_{L^2_{t,x}\dot{H}^s_v([c,b] \times \R^{2d} )}
\le C \left( \| G_0\|^2_{L^2(\R^{2d})} + \| H_1 \|^2_{L^2 ([0,T] \times 
\R^{2d})}  { + \| H_2 \|^2_{L^2 ([0,T] \times 
\R^{2d})} } \right).
\end{equation}
The function $G$ is also a subsolution of the fractional Kolmogorov equation with an appropriate right hand side
\[ G_t + v \cdot \nabla_x G + (-\Delta)^s G \leq (-\Delta)^s G + \tilde{\Lv} G + H_1 + {(-\Delta)^{s/2} H_2}.\]
Thus, $G$ is smaller or equal to the exact solution of this equation. Theorem~\ref{t:upperbound} ensures that
$\tilde \Lv G \in L^2([0,T] \times \R^d, H^{-s}(\R^d))$.  
We then can apply Proposition~\ref{l:kolmogorov-youngs} to $G$ with $h=H_1 {+ (-\Delta)^{s/2} H_2} + \tilde{L_v} G + (-\Delta)^s G$
so that
\[\|h\|_{L^2([0,T] \times \R^{d}, H^{-s}(\R^d) )} \le \|H_1\|_{L^2} {+ \|H_2\|_{L^2}} + C \| G \|_{L^2_{t,x}\dot{H}^s_v([0,T] \times \R^{2d})}. \]
and get \eqref{e:gain}.
\end{proof}

Let us analyze a localized version of the energy dissipation.
\begin{lemma}[Local energy dissipation]\label{lem:caccio}
  Let $f$ be a subsolution of \eqref{e:main} in
  $[0,T] \times B_{R^{1+2s}} \times B_R$ with $h=0$. Assume
  $0 \leq f \leq 1$ almost everywhere in
  $[0,T] \times B_{R^{1+2s}} \times \R^{d}$. Assume $K$ satisfies \eqref{e:Klower},
  \eqref{e:Kabove}, \eqref{e:Kcancellation0} and
  \eqref{e:Kcancellation1} with $\Radius = 2R$.  Then, for any
  $\delta \in (0,1)$, we have
\begin{multline} \label{e:energy-g}
 \sup_{t \in [0,T]} \iint_{B_{(R-\delta)^{1+2s}} \times B_{R-\delta}} f(t,x,v)^2 \dv \dx  +  \int_0^T \int_{B_{(R-\delta)^{1+2s}}} \|f\|_{H^s(B_{R-\delta})}^2   \dd x \dt \\
  \leq  \iint_{B_{R^{1+2s}} \times B_R}  f(0,x,v)^2 \dv \dx + C \delta^{-2}  |\{ f> 0 \} \cap [0,T] \times B_{R^{1+2s}} \times B_R| .
 \end{multline}
\end{lemma}

\begin{remark}
  The factor in $\delta^{-2}$ can be improved in terms of $s$
  (probably to $\delta^{-2s}$). The optimal power is irrelevant for
  the rest of our proof.
\end{remark}
\begin{proof}
  Let $\varphi: \R^{2d} \to [0,1]$ be $C^\infty$, supported in
  $B_{R^{1+2s}} \times B_R$, so that $\varphi = 1$ in
  $B_{(R - \delta)^{1+2s}} \times B_{R -\delta}$. It is not hard to
  check that we can construct such $\varphi$ with
  $\|\varphi\|_{C^2} \leq \delta^{-2}$.

Let $g = (\varphi + f-1)_+$ and $\tilde g = (1-\varphi - f)_-$
We use $g$ as a test function for \eqref{e:main}  and obtain for a.e. $t \in [0,T]$,
\begin{align*} 
  0 &\geq \iint (f_t + v \grad_x f)  g \dd v \dd x + \int \Ex(f,g) \dd x, \\
    &= -\frac12  \frac{d}{dt} \iint g^2 \dd v \dd x  + \int \Ex(g,g) \dd x - \int \Ex(\tilde g, g) \dd x 
- \int \Ex(\varphi,g) \dd x  + \int g  ( \partial_t \varphi +  v \cdot \grad_x \varphi) \dd v \dd x.
\end{align*}
We used the fact that
$\nabla_x (g^2) = 2 g (-\nabla_x \varphi + \nabla_x f)$.  Remarking that
$\Ex (\tilde g,g) \le 0$ and using \eqref{e:Klower} and \eqref{e:Kupper2L} from Corollary~\ref{cor:upper-2L} yields for any $t_0 \in [0,T]$,
\begin{multline*}
\frac12  \iint g^2(t_0,x,v) \dv \dx  +  \lambda \int_{0}^{t_0} \int \|g\|_{H^s}^2  \dd x \dt \\
\leq \iint g^2 (0,x,v) \dv \dx + \eps \int_{0}^{t_0} \int \|g\|_{H^s}^2  \dd x \dt + C \eps^{-1} [\varphi]_{C^1}^2 
|\{g >0\} \cap \{ 0 \le t \le t_0\}| +  C \| \varphi\|_{C^2} \|g\|_{L^1} + C \|g\|_{L^2}^2\\
+ \int_{0}^{t_0}  \iint g  (\partial_t \varphi  + v \cdot \grad_x \varphi )  \dv \dx \dt.
\end{multline*}
Recall that $\|\varphi\|_{C^1} \lesssim \delta^{-1}$ and $\|\varphi\|_{C^2} \lesssim \delta^{-2}$. Also $g(t,x,v) \in [0,1]$ for all $(t,x,v)$, therefore $\|g\|_{L^1}$ and $\|g\|_{L^2}^2$ are both bounded by $|\{ g > 0\}| \le  |\{ f> 0 \} \cap [0,T] \times B_{R^{1+2s}} \times B_R|$. Therefore, taking supremum in $t_0$,
\begin{multline*}
\sup_{t \in [0,T]} \frac12  \iint g^2(t,x,v) \dv \dx  +  \int_0^T \int \|g\|_{H^s}^2  \dd x \dt \\
\leq \iint g^2 (0,x,v) \dv \dx + \eps \int_{0}^{T} \int \|g\|_{H^s}^2  \dd x \dt +\left(  \eps^{-1} \delta^{-2} +  \delta^{-2} + 2\right)
|\{ f> 0 \} \cap [0,T] \times B_{R^{1+2s}} \times B_R|.
\end{multline*}
Note that $g = f$ in $B_{(R-\delta)^{1+2s}} \times B_{R -\delta}$,
$g \leq f$ everywhere, and $g = 0$ outside of
$B_{R^{1+2s}} \times B_R$. We thus conclude the proof picking $\eps>0$
small.
\end{proof}

\begin{lemma}[Local gain of integrability]\label{lem:gain-int}
  Let $f$ be a subsolution of \eqref{e:main} in
  $[0,T] \times B_{R^{1+2s}} \times B_R$ with $h=0$. Assume
  $0 \leq f \leq 1$ almost everywhere in
  $[0,T] \times B_{R^{1+2s}} \times \R^{d}$. Assume $K$ satisfies
  \eqref{e:Klower}, \eqref{e:Kabove}, \eqref{e:Kcancellation0} and
  \eqref{e:Kcancellation1} with $\Radius = 2R$. Then for any
  $\delta \in (0,1)$ and $\delta <R$,
\begin{multline} \label{e:energy-g2-improved}
\left(  \int_{0}^T \iint_{B_{(R -\delta)^{1+2s}} \times B_{R -\delta}} f^p \dt \dv \dx   \right)^{2/p} \\
 \leq \delta^{-2} \int_{B_{R^{1+2s}} \times B_R} f(0,x,v)^2 \dv \dx 
+ C \delta^{-4} \left \vert \{f>0\}  \cap ([0,T] \times B_{R^{1+2s}} \times B_R) \right \vert
\end{multline}
where $p>2$ is some universal constant (explicit).
\end{lemma}
\begin{remark}
The exponents in the factors $\delta^{-2}$ and $\delta^{-4}$ are most certainly not optimal. This is not important for the rest of our proof.
\end{remark}
\begin{proof}
  Let us start by the following simple observation. Wherever
  $f(t,x,v) = 0$, we have $f_t + v \cdot \nabla_x f=0$ (a.e.) and
  $\Lv f \geq 0$. In particular, the following equation also holds and
  contains slightly more information than
  \eqref{e:main}.
\begin{equation} \label{e:subsol-recharged}
 f_t + v \cdot \nabla_x f - \Lv f \leq - (\Lv f) \chi_{\{f=0\}} = -\left( \int_{\R^d} f(v') K(v,v') \dd v' \right) \chi_{\{f=0\}}.
\end{equation}

Let us call 
\[ N := \delta^{-2} \int_{B_{R^{1+2s}} \times B_R} f(0,x,v)^2 \dv \dx 
+ C \delta^{-4} \left \vert \{f>0\}  \cap ([0,T] \times B_{R^{1+2s}} \times B_R) \right \vert.\]

From Lemma \ref{lem:caccio}, we know that
\[ \int_0^T \int_{B_{(R-\delta/2)^{1+2s} } } \|f\|_{H^s(B_{R-\delta/2})}^2   \dd x \dt  \leq \delta^2 N.\]

 Let
$\varphi: \R^{2d} \to [0,1]$ be $C^\infty$, supported in $B_{(R-\delta/2)^{1+2s}} \times B_{R-\delta/2}$, so that $\varphi = 1$ in $B_{(R - \delta)^{1+2s}} \times B_{R - \delta}$. It is not hard to check that we can construct such $\varphi$ with $\|D\varphi\|_{L^\infty} \lesssim \delta^{-1}$ and $\|D^2\varphi\|_{L^\infty} \lesssim \delta^{-2}$. 

Let us analyse what equation the funtion $g = \varphi f$ satisfies. Combining {Corollaries \ref{c:commutator-s<1/2} and \ref{c:commutator-s>1/2}} with \eqref{e:subsol-recharged}, we have
\[ [\partial_t + v \cdot \nabla_x - { \tilde L_v}] g \leq f (v \cdot \nabla_x \varphi) - \varphi (\Lv f) \chi_{\{f=0\}} - h_1 - h_2 - (-\Delta)^{s/2} h_3 \qquad \text{in } [0,T] \times B_{(R-\delta/2)^{1+2s}} \times B_{R-\delta/2}.\]
We want to verify that the right hand side belongs to $L^2([0,T]\times \R^d, H^{-s}(\R^d))$ with norm bounded above by $N$.

{
Following the proofs of Lemmas \ref{l:commutator-s<1/2} and \ref{l:commutator-s>1/2} and Corollaries \ref{c:commutator-s<1/2} and \ref{c:commutator-s>1/2}, we have
\[ h_1 = \int_{\R^d \setminus B_{\delta/2}(v)} \varphi(v) f(v) [K(v,v') - \tilde K(v,v')] + f(v') (\varphi(v') K(v,v') - \varphi(v) \tilde K(v,v')) \dd v'.\]
Therefore, at the points in $\{f=0\}$ we have
\[ - \varphi (\Lv f) \chi_{\{f=0\}} - h_1 \leq \int_{\R^d \setminus B_{\delta/2}(v)} f(v') \left( (\varphi(v) - 1) K(v,v') - \varphi(v') \tilde K(v,v') \right) \dd v' \leq 0.\]
}
This allows us to simplify the equation to
\begin{equation} \label{e:gint1}
 [\partial_t + v \cdot \nabla_x - \tilde{L}x_v] g \leq f (v \cdot \nabla_x \varphi) - h_1 \chi_{\{f>0\}}- h_2 - (-\Delta)^{s/2} h_3 \qquad \text{in } [0,T] \times B_{(R-\delta/2)^{1+2s}} \times B_{R-\delta/2}.
\end{equation}

Corollaries \ref{c:commutator-s<1/2} and \ref{c:commutator-s>1/2} tell us that
\[ \| h_2\|_{L^2} , {\|h_3 \|_{L^2([0,T] \times \R^d, L^2(\R^d) ) }} \lesssim \delta^{-2} \|f\|_{L^2( [0,T] \times B_{R^{1+2s}}, H^s(B_R) ) } \leq N.\]

Since $(v \cdot \nabla_x \varphi)$ is bounded and supported in $B_{(R-\delta/2)^{1+2s}} \times B_{R-\delta/2}$, and $0\leq f \leq 1$, we clearly have $\| f (v \cdot \nabla_x \varphi) \|_{L^2} \leq N$. Likewise $\|h_1 \chi_{\{f>0\} }\|_{L^2} \leq N$.

We conclude the proof applying Lemma \ref{lem:gee} to \eqref{e:gint1}.
\end{proof}

\subsection{De Giorgi's iteration}
\label{sub:1DG}

This subsection is devoted to the proof of the following lemma. 
\begin{lemma}[First lemma of De Giorgi]\label{lem:linfty}
  Let $\tilde Q = [-\tau, 0 ] \times B_{R_1^{1+2s}} \times B_{R_1}$ and 
  $\hat Q = [-\hat \tau, 0 ] \times B_{R_2^{1+2s}} \times B_{R_2}$ with
  $0 < \tilde{\tau} < \hat \tau$ and
  $R_1 \leq R_2$.  There exists
  $\eps_0>0$ (depending on $\tau$, $\hat \tau$, $R_1$, $R_2$, dimension, $s$, $\lambda$ and $\Lambda$) such that for all supersolution $f$ of
$f_t + v \cdot \nabla_x f - \Lv f \geq 0$ in $\hat Q$ such that
  $f \ge 0$ almost everywhere in $[-\hat \tau, 0] \times \R^{2d}$ and
\begin{equation}\label{assum:lem-1deg}
 \int_{\hat Q} (2 - f)_+^2 \dt \dv \dx \le \eps_0, 
\end{equation}
we have 
\[ f \ge 1 \quad \text{ a.e. in } \tilde Q. \]
\begin{figure}[ht]
\includegraphics[height=4cm]{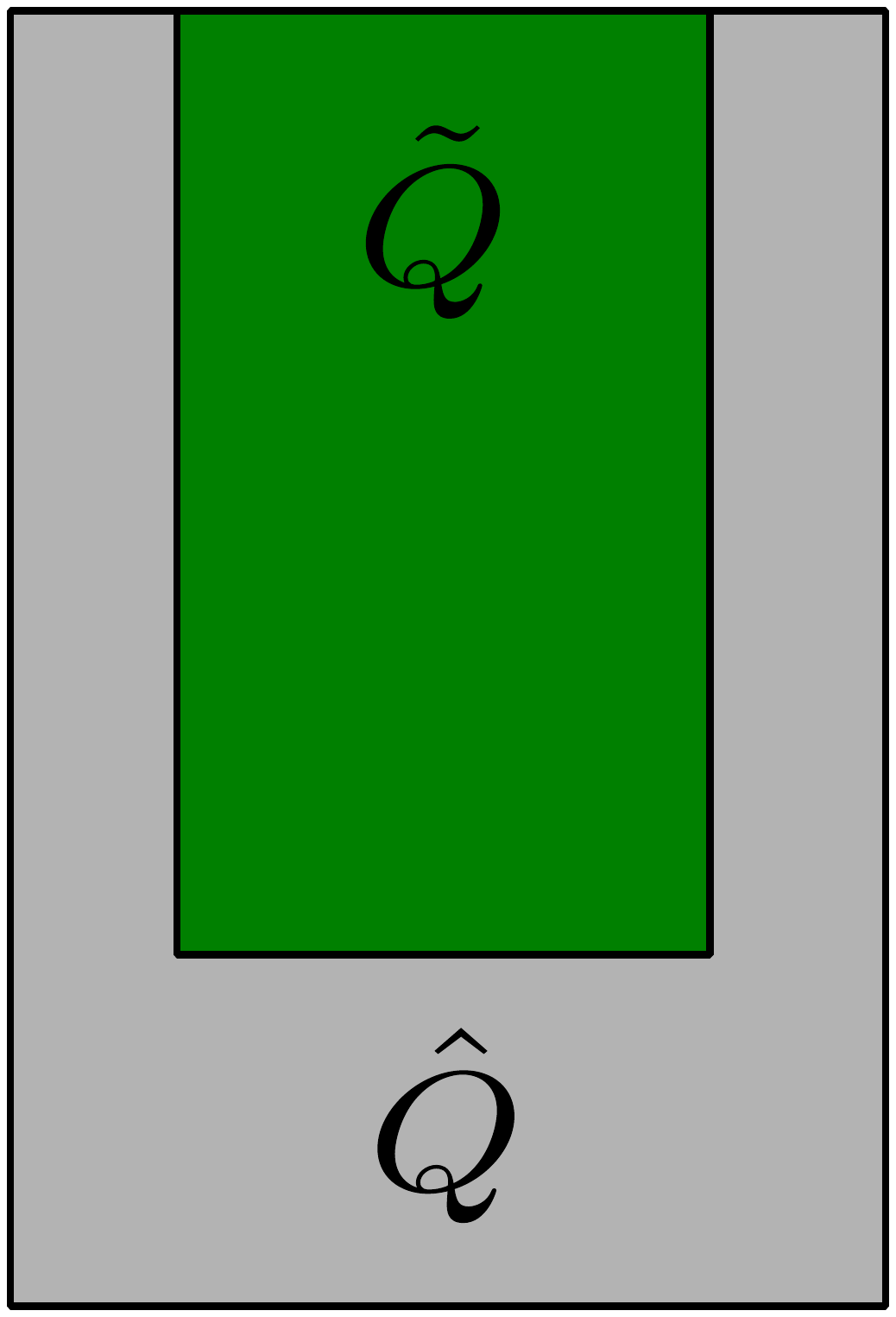}
\caption{The cylinders $\hat Q$ and $\tilde Q$}
\end{figure}
\end{lemma}
After Lemma~\ref{lem:gain-int}, the proof of Lemma~\ref{lem:linfty} follows by the relatively standard De Giorgi's iteration.

\begin{proof}[Proof of Lemma~\ref{lem:linfty}]
Let us consider the sequences
\begin{align*}
\ell_k = 1 + 2^{-k}, \qquad r_k = R_1 + (R_2-R_1) 2^{-k}, \qquad t_k = \tau - 2^{-k} (\hat \tau- \tau).
\end{align*}
We define
\[ A_k := \int_{t_k}^0 \iint_{B_{r_k^{1+2s}} \times B_{r_k}} (\ell_k-
f)_+^2 \dv \dx \dt.\]
The assumption \eqref{assum:lem-1deg} tells us that $A_0 \le
\eps_0$.
The strategy of De Giorgi's iteration is to prove that $A_k \to 0$ as
$k \to \infty$ provided that $\eps_0$ is sufficiently small. The
conclusion clearly follows from that.

In order to prove that $A_k$ converges towards $0$, we are going to prove that 
\begin{equation}\label{eq:ak}
A_{k+1} \le C 2 ^{Ck} A_k^{1+ \eps}
\end{equation}
for some $\eps >0$. 

We first pick $t_{k+\frac12} \in [t_k,t_{k+1}]$ such that 
\[ \iint_{B_{r_k^{1+2s}} \times B_{r_k}} (\ell_k - f
(t_{k+\frac12},x,v))_+^2 \dv \dx \le \frac{1}{t_{k+1}-t_k}
\int_{t_k}^{t_{k+1}} \iint_{B_{r_k^{1+2s}} \times B_{r_k}}
(\ell_k-f)_+^2\dv \dx \dt \le C 2^k A_k.\]

Note that $(\ell_{k+1} - f)_+$ is a subsolution with values in $[0,2]$
(in particular half of it takes values in $[0,1]$). We then apply
Lemma~\ref{lem:gain-int}, and obtain the following inequality (note
that $\ell_{k+1} \leq \ell_k$)
\begin{equation}\label{e:interm}
 \left(  \int_{t_{k+\frac12}}^0 \int_{B_{r_{k+1}^{1+2s}} \times B_{r_{k+1}}} (\ell_{k+1} - f)_+^p \dv \dx  \dt  \right)^{2/p} 
 \leq   C 4^k  A_k  + C 16^k | \{ f < \ell_{k+1}\} \cap ([t_{k+1/2},0] \times B_{r_{k}^{1+2s}  } \times B_{r_{k}}) | .
\end{equation}
We now estimate $ | \{ f < \ell_{k+1}\} \cap ([t_{k+1/2},0] \times B_{r_{k}} \times B_{r_{k}})|$ in terms of $A_k$. We use Chebyshev inequality and get
\begin{equation} \label{eq:ak-mes}
\begin{aligned}
| \{ f < \ell_{k+1}\} \cap ([t_{k+1/2},0] \times B_{r_{k}^{1+2s} } \times B_{r_{k}}) | &= | \{ ( \ell_{k} - f)_+ > 2^{-k-1} \} \cap([t_{k+1/2},0] \times B_{r_{k}^{1+2s} } \times B_{r_{k}}) |, \\
&\leq 16^{k+1}  A_k.
\end{aligned}
\end{equation}

Combining \eqref{e:interm} and \eqref{eq:ak-mes}, we get
\[
 \left(  \int_{t_{k+1}}^0 \iint_{B_{r_{k+1}^{1+2s} } \times B_{r_{k+1}}} (\ell_{k+1}-f)_+^p \dt \dv \dx   \right)^{\frac2p} 
 \leq   C 2^{8k}  A_k
\]
(we used that $t_{k+\frac12} \le t_{k+1} \le 0$). 
We can now combine this estimate  with \eqref{eq:ak-mes} and get 
\begin{align*}
 A_{k+1} & \le \left(  \int_{t_{k+1}}^0 \iint_{B_{r_{k+1}^{1+2s} } \times B_{r_{k+1}}} (\ell_{k+1}-f)_+^p \dt \dv \dx   \right)^{\frac2p}  | \{ f < \ell_{k+1}\} \cap ([t_{k+1},0] \times B_{r_{k+1}^{1+2s} } \times B_{r_{k+1}}) | ^{1 -\frac2p} \\
& \le C 2^{8k} A_k^{1+ \frac{2-p}p}.
\end{align*}
This yields \eqref{eq:ak} with $\eps = \frac{2-p}p >0$. The proof is now complete. 
\end{proof}

\section{Barrier functions for $s < 1/2$}
\label{sec:barriers}

A remarkable difference between the range $s < 1/2$ and $s \geq 1/2$
is that, in the former, the integral expression in the definition of
$\Lv f(v)$ is computable pointwise for all smooth functions $f$
provided that $K$ satisfies the first line in \eqref{e:Kabove}. The reason for this is simply that from the Lipschitz continuity of
$f$ we get
\begin{equation}\label{e:lip}
\int_{B_{2r}(v) \setminus B_r(v)} |f(v) - f(v')| K(v,v') \dd v' \leq r
\|f\|_{\Lip} \int_{B_{2r(v)} \setminus B_r(v)} K(v,v') \dd v' \leq \Lambda
\|f\|_{\Lip} r^{1-2s}.
\end{equation}
This is summable for $r = 2^{-k}$ as $k$ ranges accross the natural
numbers when $s < 1/2$.

If we assumed further than $K$ is symmetric in
the \emph{non-divergence} sense $K(v,v+h) = K(v,v-h)$, then the same
analysis as above would hold for $s \in (0,1)$ and $f \in C^{1,1}$ (instead of
$f \in \Lip$) and the results in this section could be extended to the full range $s \in (0,1)$. Note
that the Boltzmann kernel satisfies this symmetry, but we do not make that assumption in Theorems \ref{thm:whi} and
\ref{thm:holder}.

We build barrier functions using crucially the assumption \eqref{e:Knondegeneracy}.

\begin{lemma}[Existence of barriers] \label{l:barrier} For any $r>0$, $R>0$, $\tau > 0$ and
  $T>0$, there exist constants $\theta>0$ and $R_1 > 0$, and a
  function $\varphi : [0,\infty) \times \R^d \times \R^d \to [0,1]$
  such that
\begin{itemize}
\item we have $\varphi \in C^{1,1}([0,\infty) \times
  \R^{2d})$;
  moreover, $\varphi$ is smooth in the open set $\{\varphi > 0\}$;
\item for any kernel $K(t,x,v)$ that satisfies \eqref{e:Kabove}
  and \eqref{e:Knondegeneracy} with $\Radius = R_1$, and all
  $(t,x,v) \in \Omega \subset [0,\infty) \times \R^{2d}$, we have
\[ \varphi_t + v \cdot \nabla_x \varphi - \Lv \varphi \leq 0 \qquad \text{ in } \Omega;\]
\item at the initial time, the support of $\varphi(0,\cdot,\cdot)$ is contained in $B_{r^{1+2s}} \times B_r$;
\item we have the following lower bound: $\varphi(t,x,v) \geq \theta$ if $t \in [\tau,T]$, $x \in B_{R^{1+2s}}$ and $v \in B_R$;
\item the function $\varphi(t,x,v)$ vanishes if $t \in [0,T]$ and $(x,v) \notin B_{R_1^{1+2s}} \times B_{R_1}$.
\end{itemize}

The function $\varphi$ depends on $r$, $R$, $\tau$, $T$, dimension
$d$, $\lambda$, $\Lambda$ and $s$ (which should be in $(0,1)$). The radius $R_1$ depends on $r$, $R$, $\tau$, $T$, dimension
$d$, and $s$ (but not $\lambda$ and $\Lambda$).
\end{lemma}

Lemma \ref{l:barrier} will be proved by the end of this section. We remark that we only use the first line in \eqref{e:Kabove}.

It is convenient to define the extremal (Pucci type) operators which
correspond to the supremum and infimum of all values of $\Lv f(v)$ for
any kernel $K$ satisfying \eqref{e:Kabove} and
\eqref{e:Knondegeneracy}.

Let us say that a nonnegative kernel $K : \R^d \to [0,+\infty]$ belongs to the class
$\kernels$ if
\[
K \in \kernels \Leftrightarrow \begin{cases}
\int_{\R^d \setminus B_r} K(w) \dd w \leq \Lambda r^{-2s}, \\
\inf_{|e|=1} \int_{B_r} (w\cdot e)_+^2 K(w) \dd w \geq \lambda r^{2-2s}.
\end{cases}
\]
Correspondingly, we define the extremal operators $\Mp$ and $\Mm$.
\begin{align*} 
 \Mp f(v) &= \sup \left\{ \int_{\R^d} (f(v') - f(v)) K(v'-v) \dd v' : K \in \kernels \right\},\\
 \Mm f(v) &= \inf \left\{ \int_{\R^d} (f(v') - f(v)) K(v'-v) \dd v' : K \in \kernels \right\}.
\end{align*}

Note that the infimum and supremum are taken only with respect to a family of translation invariant linear operators, whose kernels depend only on $v'-v$. However, the kernel which achieves the extremal value will be different at every value of $v$. Therefore, effectively, the operators $\Mp f$ and $\Mm f$
correspond to the supremum and infimum value of $\Lv f$ for all kernels $K(v,v')$ satifying the first line in \eqref{e:Kabove} and
\eqref{e:Knondegeneracy}.

We start by pointing out a simple continuity property of $\Mp$ and $\Mm$. 
\begin{lemma} \label{l:Mm-continuous}
Let $f$ and $g$ be two bounded functions that are Lipschitz in $B_r(v)$, then
\[|\Mm f(v) - \Mm g(v)| \leq C_r \left(\|f-g\|_{L^\infty(\R^d)} + \|f-g\|_{\Lip(B_r(v))} \right).\]
The same holds for $\Mp$.
\end{lemma}
\begin{remark}
  Note that the norm $\|f-g\|_{L^\infty(\R^d)}$ can be
  weighted. Indeed, the same estimate holds with
  $\|(1+|v|)^{-\sigma} (f(v)-g(v))\|_{L^\infty(\R^d)}$ instead provided
  that $\sigma < 2s$. 
\end{remark}
\begin{proof}
  It is enough to notice that each linear operator in the infimum of
  the definition of $\Mm$ satisfies the continuity estimate.
\end{proof}

\begin{cor} \label{c:Mm-continuity} If $f_n$ is a sequence of
  functions so that $f_n \to f$ uniformly in $\R^d$ and $f_n \to f$ in
  $\Lip(\Omega)$, then $\Mp f_n$ and $\Mm f_n$ converge to $\Mp f$ and
  $\Mm f$ uniformly in compact sets of $\Omega$.
\end{cor}
The following is perhaps not strictly a corollary of Lemma \ref{l:Mm-continuous}, since it requires a slightly sharper analysis (but standard and elementary).

\begin{cor} \label{c:Mm-continuous}
  Let $f$ be a bounded continuous function in $\R^d$ and Lipchitz in
  some open set $\Omega$. The functions $\Mm f$ and $\Mp f$ are
  continuous in $\Omega$.
\end{cor}
Since the operators $\Mp$ and $\Mm$ are a supremum and infimum of
linear ones, then they are also sub- and super-additive
respectively. That means that for any $f$ and $g$,
\[ \Mm (f+g)(v) \geq \Mm f(v) + \Mm g(v), \qquad \Mp (f+g)(v) \leq \Mp f(v) + \Mp g(v).\]
\begin{lemma}[The function $\varphi_1$] \label{l:b1} Let $\varphi_1 : \R^d \to [0,1]$ be a
  nonnegative, radially symmetric function, so that
\begin{itemize}
\item $\{ \varphi_1 > 0\} = B_1$, $\varphi_1 \in C^2(B_1)$,
  $\varphi_1 =1$ in $B_{1/2}$, and
  $v \cdot \nabla \varphi_1(v) \leq 0$;
\item $\varphi_1 \in C^2 (B_1)$ and $\varphi \in C^{1,1}(\R^d)$; more
  precisely, there is a discontinuity of $D^2 \varphi_1$ on
  $\partial B_1$ so that
  $\lim_{r \to 1^-} D^2 \varphi_1 (re) = e \otimes e$ for any $|e|=1$.
\end{itemize}
Then, there exist two constants  $\delta>0$ and  $\theta > 0$ so that
\[ 
\Mm \varphi_1(v) \geq \theta \text{ for any $v \in B_1$ so that } \varphi_1(v) < \delta.
\]
\end{lemma}
\begin{remark}
We can choose any function $\varphi_1 (x) = \Psi (|x|)$ with
    $\Psi$ non-increasing in $\R$, positive and $C^2$ in $[0,1]$,
    supported in $[0,1]$, $\Psi \equiv 1$ in $[0,1/2]$, and
    $\Psi'(1)=0$ and $\Psi'' (1) = 1$.
\end{remark}
\begin{proof}
  Since $\Mm \varphi_1$ is continuous in $B_1$, it is enough to prove
  that $\Mm \varphi_1$ is strictly positive on $\partial B_1$. From
  radial symmetry, we are left to show that $\Mm \varphi_1(e) > 0$ for
  $e = (1,0,\dots,0)$.

Let $\eps > 0$. From the super-additivity of $\Mm$, we have
\[ 
\Mm \varphi_1(e) \geq \Mm (\varphi_1 \chi_{B_\eps(e)}) (e) + \Mm
(\varphi_1 \chi_{\R^d\setminus B_\eps(e) } ) (e).
\]

For any $K \in \kernels$, since $K\geq 0$ and $\varphi_1 \geq 0$, we have
\[
  \int_{\R^d} ((\varphi_1 \chi_{\R^d\setminus B_\eps(e)}) (v')
- (\varphi_1 \chi_{\R^d\setminus B_\eps(e)})(e)))K(v'-e) \dd v' 
= \int_{\R^d \setminus B_\eps(e)} \varphi_1(v') K(v'-e) \dd v' \geq 0.
\]
Therefore $ \Mm (\varphi_1 \chi_{\R^d\setminus B_\eps(e) } )   (e) \geq 0$.

We now show that $\Mm (\varphi_1 \chi_{B_\eps(e)}) (e)$ is bounded
below for $\eps>0$ small. Essentially this follows because
$\varphi_1(v')$ is approximately $((v'-e) \cdot (-e))_+^2$ in
$B_\eps(e)$.

Indeed, let the scaled function $\varphi_\eps$ be
\[ \varphi_\eps(w) = \begin{cases}
\eps^{-2} \varphi_1(e+\eps w) & \text{ if } |w| < 1, \\
0 & \text{ if } |w| \geq 1.
\end{cases}\]
Thus
\[
\Mm (\varphi_1 \chi_{B_\eps(e)})(e) = \eps^{2-2s} \Mm \varphi_\eps (0) .
\] 
From the definition of $\varphi_1$ ,we know that
\[ \varphi_\eps(w) \to q(w) := \begin{cases}
(-w \cdot e)_+^2 & \text{ for } |w| < 1, \\
0 & \text{ for } |w| \geq 1,
\end{cases}
\]
uniformly in $\R^d$ and also in $\Lip(B_{1/2})$. Therefore, using
Corollary \ref{c:Mm-continuity},
\[\Mm \varphi_\eps(0) \to \Mm q(0) \geq \lambda.\]
The last inequality comes from the non-degeneracy condition~\eqref{e:Knondegeneracy}.

Therefore, choosing $\eps$ sufficiently small,
\[ 
\Mm \varphi_1(e) \geq \Mm (\varphi_1 \chi_{B_\eps(e)}) (e) \geq
\frac \lambda 2 \eps^{2-2s} > 0.
\] 
This concludes the proof.
\end{proof}
\begin{lemma}[The function $\varphi_2$] \label{l:b2} Let $t_0 > (0,1)$ be arbitrary and $\varphi_1$
  be a function as in Lemma \ref{l:b1}. Let
  $A = \left( 5 +\frac 1 {2s}\right)$. Let us define the function
  $\varphi_2: \R^{2d} \to [0,1]$ to be
  \[ \varphi_2(x,v) := \varphi_1(x) \varphi_1(v-Ax). \]
  There exists a constant $\delta>0$ so that if at some point $(x,v)$,
  $\varphi_2(x,v) < \delta$, then
  \begin{equation} \label{e:b2} \left(-1-\frac 1 {2s} \right) x \cdot
    \nabla_x \varphi_2 - \frac 1 {2s} v \cdot \nabla_v \varphi_2 + t_0
    \left( v \cdot \nabla_x \varphi_2 - \Mm_v \varphi_2 \right) \leq
    0.
\end{equation}
\end{lemma}
\begin{proof}
  Since $\min \varphi_2 = 0$, then $\Mm_v \varphi_2 \geq 0$ wherever
  $\varphi_2 = 0$. Thus, the inequality is trivial wherever
  $\varphi_2 = 0$. We are left to verify it at points where
  $\varphi_2 > 0$. Note that this is a bounded set since there
  $|x| \leq 1$ and $|v| \leq A|x|+1 \leq A+1$.

We expand the left hand of \eqref{e:b2}, in terms of $\varphi_1$,
  $x$ and $v$, as the sum of two terms $T_1+T_2$, where
\begin{align*}
T_1 &= \varphi_1(x) \left\{ \nabla \varphi_1(v-Ax) \cdot \left( A \left( 1+\frac 1 {2s} \right) x 
- \left( t_0 A + \frac 1 {2s} \right) v \right) - t_0 \Mm \varphi_1(v-Ax) \right\}, \\
T_2 &= \varphi_1(v-Ax) \nabla \varphi_1(x) \cdot \left\{ -\left(1+\frac 1 {2s} \right) x + {t_0} v \right\}.
\end{align*}

We first claim that 
\begin{equation}\label{e:d1}
 \text{there exist } \delta_1 >0  \text{ such that } T_1 \le 0 \text{ if }
\varphi_1(v-Ax) < \delta_1.
\end{equation}
Using Lemma \ref{l:b1}, we pick $\delta_1$ sufficiently small so that
$\Mm \varphi_1(v-Ax) \geq \theta$ whenever
$\varphi_1(v-Ax) < \delta_1$. Thanks to the continuity of
$\nabla \varphi_1$, we pick $\delta_1$ smaller if necessary so that
whenever $\varphi_1(v-Ax) < \delta_1$,
\[ 
\nabla \varphi_1(v-Ax) \cdot \left( A \left( 1+\frac 1 {2s} \right) x
  - \left( t_0 A + \frac 1 {2s} \right) v \right) - t_0 \Mm \varphi_1(v-Ax) < -\frac {t_0 \theta} 2.
\]
Therefore, we have 
\[ T_1 \leq -\frac{t_0 \theta \varphi_1(x)} 2 \text{ whenever } \varphi_1(v-Ax) < \delta_1.\]
In particular, \eqref{e:d1} holds true. 

We next claim that 
\begin{equation}\label{e:d2}
 \text{there exist } \delta_2 >0  \text{ such that } T_2  \text{ if }
\varphi_1(x) < \delta_2.
\end{equation}
Because of the second derivative of $\varphi_1$ of $\partial B_1$, we
have the following expansion
\[ \nabla \varphi_1(x) = -(1-|x|) \frac x {|x|} + O((1-|x|)^2 ).\]
Whenever $\varphi_1(v-Ax) > 0$, also $v \in B_1(Ax)$, and therefore 
\[ \nabla \varphi_1(x) \cdot \left\{ -\left(1+\frac 1 {2s} \right) x +
  v \right\} \leq (1-|x|) (-4|x|+1) + C(1-|x|)^2 < -(1-|x|) +
C(1-|x|)^2.\]
Thus, 
\[ 
T_2 \leq -\varphi_1(v-Ax) (1-|x|)/2 \le 0 \text{  whenever }
\varphi_1(x) < \delta_2
\]
and $\delta_2$ is sufficiently small. In particular, \eqref{e:d2} holds true.

In view of \eqref{e:d1} and \eqref{e:d2}, $T_1 + T_2 \le 0$ if 
$\varphi_1(v-Ax) < \delta_1$ and $\varphi_1(x) < \delta_2$. 

Let us analyse the case $\varphi_1(x) \geq \delta_2$; in this case
consider $\varphi_1(v-Ax) < \delta_{11} < \delta_1$ so that
\[ T_1 + T_2 \leq - t_0 \delta_2 \theta/2 + C \delta_{11}.\]
Picking $\delta_{11}$ sufficiently small (depending on the previous
choice of $\delta_2$), we assure $T_1 + T_2 < 0$ in this case.

We are left with the case $\varphi_1(v-Ax) \geq \delta_1$. In this case we have
for $\varphi_1(x) < \delta_{21}$, 
\begin{align*}
T_1 + T_2 &\leq C \varphi_1(x) -\varphi_1(v-Ax) (1-|x|)/2, \\
&\leq C \delta_{21} -\delta_1 (1-|x|)/2 < 0, 
\end{align*}
provided $|x|$ is sufficiently close to $1$, which follows if
$\varphi_1(x) < \delta_{21} < {\delta_2}$ with $\delta_{21}$
sufficiently small.

Finally, we finish the proof picking $\delta = \delta_{11} \delta_{21}$ to ensure that at
least one of the three cases above holds.
\end{proof}
\begin{lemma}[The function $\varphi_3$] \label{l:b3}
Let $\varphi_2$ be the function from Lemma \ref{l:b2} and $t_0> 0$. The function $\varphi_3(t,x,v)$ given by
\[ 
\varphi_3(t,x,v) = \frac {t_0^p}{(t+t_0)^p} \varphi_2 
\left( \left( \frac {t_0} {t+t_0}\right)^{1+\frac 1{2s}} x  , 
\left( \frac {t_0} {t+t_0} \right)^{\frac 1{2s}} v \right),
\]
is a subsolution of the equation
\[ \partial_t \varphi_3 + v \cdot \nabla_x \varphi_3 - \Mm \varphi_3 \leq 0,\]
provided that $p$ is sufficiently large (depending on $\varphi_1$, $\lambda$, $\Lambda$, $s$ and $d$, but not $t_0$).
\end{lemma}
\begin{proof}
We write the equation in terms of $\varphi_2$. We have
\begin{align*}
\partial_t \varphi_3 + v \cdot \nabla_x \varphi_3 - \Mm \varphi_3 &= \frac{t_0^p}{(t+t_0)^{p+1}} \bigg\{ -p \varphi_2(X,V) \\
& \qquad +  \left(-1-\frac 1 {2s} \right) X \cdot \nabla_x \varphi_2(X,V) - \frac 1 {2s} V \cdot \nabla_v \varphi_2(X,V) \\
&\qquad + t_0 V \cdot \nabla_x \varphi_2(X,V) - t_0 \Mm_v \varphi_2(X,V) \bigg\},
\end{align*}
where $X = (t_0  / (t+t_0))^{1+\frac 1{2s}} x$ and $V = (t_0 / (t+t_0))^{\frac 1{2s}} v$.

Let $\delta>0$ be as in Lemma \ref{l:b2}, so that the right hand side
is non-positive when $\varphi_2 < \delta$. We choose $p$ large so that
the term $p \varphi_2 \geq p \delta$ is larger than all the others
terms when $\varphi_2 \geq \delta$.  Thus, the right hand side is
never positive.
\end{proof}
\begin{proof}[Proof of Lemma \ref{l:barrier}]
  Note that $\varphi_3(0,x,v) = \varphi_2(x,v)$, where $\varphi_2$ and
  $\varphi_3$ are the functions in Lemmas \ref{l:b2} and \ref{l:b3}
  respectively. Note that these fuctions depends on the choice of $t_0$ which will be made below. Also, the value of $p$ depends on $t_0$. The function $\varphi_2$ is supported in
  $B_1 \times B_{A+1}$. We must rescale $\varphi_3$ in order to obtain
  a function so that $\varphi(0,x,v)$ is supported in
  $B_r \times B_r$. We pick $\rho>0$ small and let
\[ \varphi(t,x,v) = \varphi_3(\rho^{-2s} t, \rho^{-2s-1} x, \rho^{-1} v),\]
so that $\rho (A+1) \leq r$. 
This ensures the first three items in Lemma \ref{l:barrier}. Indeed, the function $\varphi$ satisfies
\[ \varphi_t + v \cdot \nabla_x \varphi - \Mm \varphi \leq 0.\]
In particular, also 
\[ \varphi_t + v \cdot \nabla_x \varphi - \Lv \varphi \leq 0,\]
since $\Lv \varphi \geq \Mm \varphi$ in $\Omega$.

In order to obtain the lower bound in $[\tau,T] \times B_R$, we are
going to choose the parameter $t_0$ accordingly. Note that the value
of $t_0$ does not affect $\varphi(0,x,v)$.

From the construction of $\varphi_1$ and $\varphi_2$, we have
$\varphi_2(x,v) = 1$ whenever $|x|<\frac 1 {4A}$ and $|v| < 1/4$.  Picking $t_0$ sufficiently small, for
$(t,x,v) \in [\tau,T] \times B_{R^{1+2s}} \times B_R$, we have
\begin{align*} 
 \rho^{-2s-1} \left( \frac {t_0} {t+t_0}\right)^{1+\frac 1{2s}} |x| &\leq \rho^{-2s-1} \left( \frac {t_0} {\tau}\right)^{1+\frac 1{2s}} R^{1+2s}  < \frac 1 {4A}, \\
\rho^{-1} \left( \frac {t_0} {t+t_0} \right)^{\frac 1{2s}} |v| &\leq \rho^{-1} \left( \frac {t_0} {\tau} \right)^{\frac 1{2s}} R < \frac  14.
\end{align*}
Therefore, when
$(t,x,v) \in [\tau,T] \times B_R \times B_R$, we have
\[ 
\varphi(t,x,v) = \frac{t_0^p}{(\rho^{-2s} t + t_0)^{p}} \geq \frac{t_0^p}{(\rho^{-2s} T +t_0)^{p}} =: \theta > 0.
\]
This justifies the fourth item in Lemma \ref{l:barrier}.

Finally, for the last item, we just pick $R_1$ sufficiently large. The
function $\varphi_2$ is supported in $B_1 \times B_{1+A}$. Depending
on our choices of $t_0$ and $\rho$ above, the function
$\varphi(t,\cdot,\cdot)$ is supported inside $B_{R_1} \times B_{R_1}$
for all $t \in [0,T]$. This achieves the construction of the barrier. 

Note that the only parameters in this construction that depend on $\lambda$ and $\Lambda$ are $p$ and $\theta$.
\end{proof}

\section{The intermediate-value  lemma for $s \ge \frac12$}
\label{sec:intermediate}

This section is devoted to the statement and proof of a version of De
Giorgi's isoperimetric lemma in the case $s \ge \frac12$. It is
inspired by the compactness method in \cite{gimv}. However, unlike
\cite{gimv}, we do not use averaging lemmas. Instead, the analysis of
the fractional Kolmogorov equation plays a critical role.

The first lemma of this section concerns a supersolution of the
equation \eqref{e:main}. In this case we add a nonnegative measure to
the right hand side in order to have an exact solution. The purpose of
this lemma is to provide a basic control of the total measure that we
add.
\begin{lemma}[A priori estimate on a nonnegative measure] \label{l:rhs-measure}
Let $Q = [0,T] \times B_{R^{1+2s} } \times B_R$, $f : [0,T] \times \R^d \times \R^d \to [0,1]$ be supported in $Q$. Assume also that
\[ f_t + v \cdot \nabla_x f + (-\Delta)^s f \geq h \qquad \text{in } [0,T] \times \R^d \times \R^d,\]
for some $h \in L^2([0,T] \times \R^d, H^{-s}(\R^d))$. Then
\[ f_t + v \cdot \nabla_x f + (-\Delta)^s f = \tilde h + \mu \qquad \text{in } [0,T] \times \R^d \times \R^d,\]
where $\mu$ is a nonnegative measure supported in $[0,T] \times B_{(2R)^{1+2s} } \times B_{2R}$ such that
\[ \mu(Q) \leq C (1+\|h\|_{L^2_{t,x} H^{-s}_v} ),\]
and $\tilde h = h$ in $[0,T] \times B_{(2R)^{1+2s} } \times B_{2R}$ and
\[ \|\tilde h\|_{L^2_{t,x}H^{-s}_v } \leq C ( 1+\|h\|_{L^2_{t,x} H^{-s}_v}  ).\]
\end{lemma}
\begin{proof}
Note that for $(x,v) \notin B_{R^{1+2s}} \times B_R$, $f_t + v \cdot \nabla_x f = 0$. Moreover, 
\[ |(-\Delta)^s f(t,x,v)| = c \left\vert\int_{B_R} f(t,x,w) |w-v|^{-d-2s} \dd w \right \vert \leq (|v|-R)^{-d-2s} |B_R| \chi_{|x| \leq R}.\]
Therefore, $f_t + v \cdot \nabla_x f + (-\Delta)^s f = (-\Delta)^s f$ is an $L^2$ function outside of $B_{R^{1+2s}} \times B_{3R/2}$.

Let $\varphi : \R^{2d} \to [0,1]$ be a smooth bump function so that $\varphi=1$ in $B_{R^{1+2s} } \times B_{3R/2}$ and $\varphi = 0$ outside of $B_{(2R)^{1+2s} } \times B_{2R}$.

We first need to justify that there is $\tilde h \in L^2([0,T] \times \R^d, H^{-s}(\R^d) )$ so that
\[ \tilde h = h \varphi + (1-\varphi)  (-\Delta)^s f. \]

We clearly have that $(1-\varphi) (-\Delta)^s f$ is in $L^2_{t,x}H^{-s}_v$. We are left to justify that $h \varphi \in L^2([0,T] \times \R^d, H^{-s}(\R^d) )$.  This follows by duality once we observe that for every $g \in L^2_{t,x}H^s_v$, we also have $\varphi g \in H^s(\R^d)$. 

With this definition of $\tilde h$, we still have
\[ f_t + v \cdot \nabla_x f + (-\Delta)^s f \geq \tilde h \qquad \text{in } [0,T] \times \R^d \times \R^d,\]
with equality for $(x,v) \notin B_{(2R)^{1+2s} } \times B_{2R}$.

Let $\mu$ be the nonnegative measure, supported in $[0,T] \times B_{(2R)^{1+2s} } \times B_{2R}$, defined by
\[ \mu := f_t + v \cdot \nabla_x f + (-\Delta)^s f - \tilde h.\]

In order to estimate the total measure of $\mu$, we test it against a test function which is identically one on its support. Let $\tilde \varphi = 1$ in $B_{(2R)^{1+2s} } \times B_{2R}$ and be supported in $B_{(3R)^{1+2s} } \times B_{3R}$. We have
\begin{align*} 
 \mu([0,T] \times \R^d \times \R^d) &= \int_{[0,T] \times \R^{2d}} \tilde \varphi \dd \mu, \\
&= \int_{[0,T] \times \R^{2d}} \tilde \varphi \left( f_t + v \cdot \nabla_x f + (-\Delta)^s f - \tilde h \right) \dd v \dd x \dd t, \\
&= \int_{\R^{2d}} (f(T,x,v) - f(0,x,v)) \tilde \varphi(x,v) \dd v \dd x \\
& \phantom{=} + \int_{[0,T] \times \R^{2d}} \left\{ 
[-v\cdot \nabla_x \tilde \varphi + (-\Delta)^s \tilde \varphi]  \,  f - \tilde \varphi \tilde h  \right\}\dd v \dd x \dd t \leq C.
\qedhere
\end{align*}
\end{proof}
\begin{lemma} [Intermediate sets for the Kolmogorov
  equation] \label{l:intermediate-set-global} Let $s \in [1/2,1)$. Let
  $f : [0,T] \times \R^d \times \R^d \to [0,1]$. Assume $f$ is a
  supersolution of the fractional Kolmogorov equation
\[ f_t + v \cdot \nabla_x f + (-\Delta)^s f \geq  h \qquad \text{ in } [0,T] \times \R^d \times \R^d,\]
where $h \in L^2([0,T] \times \R^d, H^{-s}(\R^d))$. Let $r_1 > 0$, $r_2 > 0$, $0 < r_3 < r_4$ and $0 < t_1 < t_2 < T$ such that $r_3/2 > (r_1^{1+2s}+r_2^{1+2s})/(t_2-t_1)$. We define
\begin{align*}
Q^1 &= [0,t_1] \times B_{r_1^{1+2s}} \times B_{r_1}, \\
Q^2 &= [t_2,T] \times B_{r_2^{1+2s}} \times B_{r_2}, \\
Q^3 &= [0,T] \times B_{r_3^{1+2s}} \times B_{r_3}, \\
Q^4 &= [0,T] \times B_{r_4^{1+2s}} \times B_{r_4}.
\end{align*}
Let us assume that $f$ is supported in $Q^4$ and $f \in L^2([0,T] \times \R^d, H^s(\R^d)) \cap C([0,T], L^2(\R^d \times \R^d))$. For every pair of positive numbers $\delta_1,\delta_2$, there exist $\theta>0$ and $\mu>0$ so that whenever 
\[ |\{f=1\} \cap Q^1| \geq \delta_1 \qquad \text{and} \qquad |\{f \leq \theta \} \cap Q^2| \geq \delta_2,\]
then
\[ |\{\theta < f < 1\} \cap Q^3| \geq \mu.\] 
Here, the constants $\theta$ and $\mu$ depend on $\delta_1$, $\delta_2$, $\|h\|_{L^2_{t,x}H^{-s}_v}$, $\|f\|_{L^2_{t,x}H^s_v}$, $t_1$, $t_2$, $T$, $r_1$, $r_2$, $r_3$, $r_4$, $s$ and $d$.
\end{lemma}
\begin{figure}[ht]
\setlength{\unitlength}{1in} 
\begin{picture}(2.667 ,2.000)
\put(0,0){\includegraphics[height=2.000in]{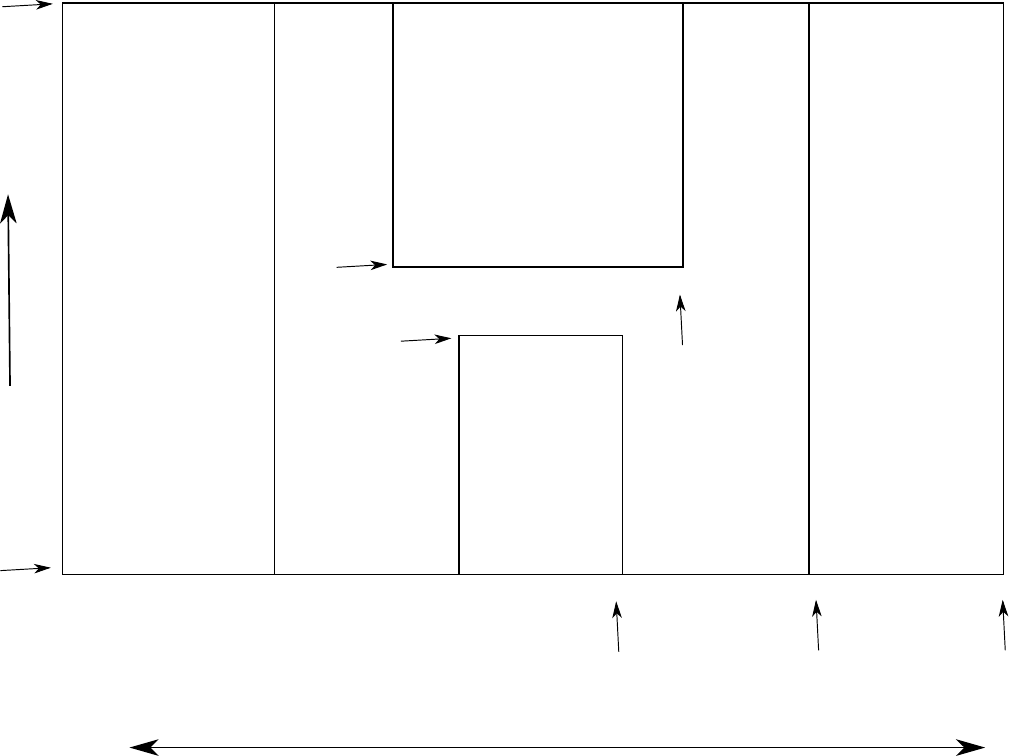}}
\put(1.6,0.192){$r_1$}
\put(2.081,0.192){$r_3$}
\put(1.740,0.983){$r_2$}
\put(2.6,0.192){$r_4$}
\put(0.935,1.056){$t_1$}
\put(0.777,1.244){$t_2$}
\put(-0.1,1.931){$T$}
\put(-0.32,0.44){$t=0$}
\put(-0.049,1.174){$t$}
\put(1.293,-0.116){$(x,v)$}
\put(1.3,1.5){$Q^2$}
\put(1.3,0.6){$Q^1$}
\put(1.8,0.6){$Q^3$}
\put(2.4,0.6){$Q^4$}
\end{picture}
\caption{The geometric setting of Lemma \ref{l:intermediate-set-global}.}
\end{figure}
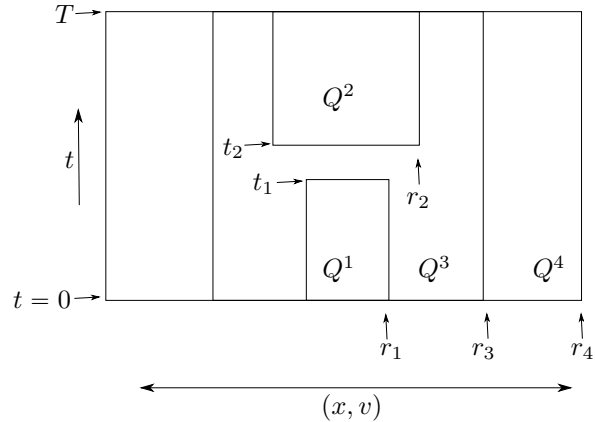

\begin{proof}
Assume the contrary. Then, there is a sequence of functions $f_i$, uniformly bounded in $L^2([0,T] \times \R^d, H^s(\R^d))$, $h_i$ uniformly bounded in $L^2([0,T] \times \R^d, H^{-s}(\R^d))$, and sequences of positive numbers $\theta_i \to 0$ and $\mu_i \to 0$ so that all hypotheses in the lemma hold, however
\begin{equation} \label{e:cpt1}
 \begin{aligned} 
 |\{f_i=1\} \cap Q^1|  &\geq\delta_1, \\
 |\{f_i \leq \theta_i\} \cap Q^2| &\geq \delta_2, \\
|\{\theta_i < f_i < 1\} \cap Q^3| &< \mu_i.
\end{aligned} 
\end{equation}
We will find a contradiction \emph{by compactness}. That is, we will
find a subsequence that converges and find a limit function $f_\infty$
which only takes the values $1$ and $0$ in $Q^3$. We will derive a
contradiction there provided $s \in [1/2,1)$.

According to Lemma \ref{l:rhs-measure}, there are measures $\mu_i$,
supported in $[0,T] \times B_{(2r_4)^{1+2s}} \times B_{2r_4}$, and
modified right hand sides $\tilde h_i$ so that
\[ [\partial_t + v \cdot \nabla_x + (-\Delta)^s] f_i = \tilde h_i + \mu_i.\]
Moreover, $\mu_i ([0,T] \times \R^d \times \R^d) \leq C$, $\|\tilde h_i\|_{L^2_{t,x} H^{-s}_v} \leq C$.

Let us write $\tilde h_i = h_1^i + (-\Delta)^{s/2} h_2^i$ for $h_1^i$
and $h_1^i$ in $L^2([0,T] \times \R^{2d})$, with
$\|h_1^i\|_{L^2} \leq C$ and $\|h_2^i\|_{L^2} \leq C$.

Up to extracting a subsequence, we can assume that
$f_i(0,\cdot,\cdot)$ converges weakly in $L^2(\R^{2d})$, $f_i$,
$h_1^i$ and $h_2^i$ converge weakly in $L^2([0,T] \times \R^{2d})$ to
$f_\infty$, $h_1^\infty$ and $h_2^\infty$, and $\mu_i$ converges
weakly-$\ast$ in the space of Radon measures
$\mathbb M([0,T] \times \R^{2d})$ to some measure $\mu_\infty$.

Using the formula \eqref{e:duhamel}, we can write
$f_i = T_0 f_i(0,\cdot,\cdot) + T_1 \mu_i + T_2 h_1^i + T_3
h_2^i$.
Here, the operators $T_0 : L^2(B_{r_4}) \to L^1(Q^4)$,
$T_1: \mathbb M([0,T] \times \R^{2d}) \to L^1(Q^4)$,
$T_2,T_3: L^2([0,T] \times \R^{2d}) \to L^2(Q^4) $ and are given by
\begin{align*}
T_0 f_0 &:=   f_0  \ast_t J(t,\cdot,\cdot). \\
T_1 \mu &:=\int_0^t \mu(t) \ast_t  J(t-\tau,\cdot,\cdot) \dd \tau, \\
T_2 h_1 &:=\int_0^t h_1(t) \ast_t  J(t-\tau,\cdot,\cdot) \dd \tau, \\
T_3 h_2 &:=\int_0^t h_2(t) \ast_t  (-\Delta)^{s/2} J(t-\tau,\cdot,\cdot) \dd \tau.
\end{align*}
Note that $T_1$, $T_2$ and $T_3$ are exactly convolutions in all variables $(t,x,v)$ with respect to the natural Lie group structure. Also $T_0$ is the same as $T_1$ applied to a singular measure concentrated on $t=0$ with marginal density $f_0$.

The operators $T_1$, $T_2$ and $T_3$ are compact simply because they are convolutions with the $L^1$ functions $J$ and $(-\Delta)^{s/2} J$. Therefore $f_i = T_0 f_i(0,\cdot,\cdot) + T_1 \mu_i + T_2 h_1^i + T_3 h_2^i$ converges strongly in $L^1(Q^4)$ to some function $f_\infty$. Since we have $0 \leq f_i \leq 1$, then in fact $f_i$ converges strongly to $f_\infty$ in $L^p(Q^4)$ for any $p \in [1,+\infty)$.

The function $f_\infty$ solves, in the sense of distributions,
\[ [\partial_t + v \cdot \nabla_x + (-\Delta)^s] f_\infty \geq h_1^\infty + (-\Delta)^{s/2} h_2^\infty .\]
Moreover, since $f_i \to f_\infty$ in $L^1$, from \eqref{e:cpt1} we deduce that
\begin{equation} \label{e:cpt2}
 \begin{aligned} 
 |\{f_\infty=1\} \cap Q^1|  &\geq\delta_1, \\
 |\{f_\infty=0\} \cap Q^2| &\geq \delta_2, \\
|\{0 < f_\infty < 1\} \cap Q^3| &= 0.
\end{aligned} 
\end{equation}
Then, $f_\infty$ only takes the values $0$ and $1$, almost everywhere
in $Q^3$. Moreover, we have
$\|f_\infty\|_{L^2([0,T] \times \R^d, H^s(\R^d))} \leq C$. Thus
$f_\infty(t,x,\cdot) \in H^s(B_{r_3})$ almost everywhere in
$[0,T]\times B_{r_3^{1+2s}}$.  Since $s \geq 1/2$, this implies that
$f(t,x,\cdot)$ is either constant $1$ or constant $0$ in $B_{r_3}$ for
$(t,x) \in [0,T]\times B_{r_3^{1+2s}}$. From this point on, we write
$f_\infty(t,x) := f_\infty(t,x,v)$ provided that $(t,x,v)$ is
restricted to $Q^3$. Note that $(-\Delta)^s f_\infty$ is not constant
in $Q^3$ due to the nonlocality of $(-\Delta)^s$.

Let $\varphi : \R^d \to [0,+\infty)$ be a smooth bump function supported in $B_{r_3 /2}$, such that
\[ \int_{\R^d} \varphi(v) \dd v = 1, \qquad \int_{\R^d} \varphi(v) \, v \dd v = 0.\]

For any $v_0 \in B_{r_3/2}$ and $(t,x) \in [0,T] \times B_{r_3}$, we have
\[ f_\infty(t,x) = \int_{\R^d} f_\infty(t,x) \varphi(v-v_0) \dd v.\]
Therefore, using the equation
\begin{align*} 
 \partial_t f_\infty(t,x) &\geq \int_{\R^d} \big[ -v \cdot \nabla_x f_\infty(t,x) - (-\Delta)^s f_\infty(t,x,v) + h_1^\infty(t,x,v) + (-\Delta)^{s/2} h_2^\infty(t,x,v) \big] \varphi(v-v_0) \dd v, \\
&= -v_0 \cdot \nabla_x f_\infty(t,x) + \int_{\R^d} \left\{ (-f_\infty(t,x,v) + h_2^\infty(t,x,v)) (-\Delta)^s \varphi(v-v_0) + h_1^\infty \varphi(v-v_0) \right\} \dd v.
\end{align*}
Thus, for any $v_0 \in B_{r_3/2}$, $f_\infty(t,x)$ satisfies the transport equation
\[ \partial_t f_\infty + v_0 \cdot \nabla_x f_\infty \geq H_{v_0}(t,x).\]
where $H_{v_0}$ is the function in $L^2([0,T] \times B_{r_3^{1+2s}})$ given by
\[ H_{v_0}(t,x) = \int_{\R^d} (-f_\infty(t,x,v) + h_2^\infty(t,x,v)) (-\Delta)^s \varphi(v-v_0) + h_1^\infty \varphi(v-v_0) \dd v.\]

From \eqref{e:cpt2}, we know that there exist some $\tau_1 \in [0,t_1]$ and $\tau_2 \in [t_2,T]$ so that
\begin{align*}
 |\{x : f_\infty(\tau_1,x)=1\} \cap B_{r_1}|  &\geq \frac{\delta_1} {t_1}, \\
 |\{x : f_\infty(\tau_2,x)=0\} \cap B_{r_2}| &\geq \frac{\delta_2} {T-t_2}.
\end{align*}
Let $S_1 = \{x : f_\infty(\tau_1,x)=1\} \cap B_{r_1^{1+2s}}$ and $S_2 = \{x : f_\infty(\tau_2,x)=0\} \cap B_{r_2^{1+2s}}$. Since
\[ \| \chi_{S_1} \ast \chi_{-S_2} \|_{L^1} = |S_1| |S_2| \geq \frac{\delta_1 \delta_2} {t_1(T-t_2)},\]
then, there exists one vector $w_0 \in B_{r_1^{1+2s}+r_2^{1+2s}}$ such that
\[ | S_1 \cap (S_2-w_0)| \geq \frac{\delta_1 \delta_2} {t_1(T-t_2) |B_{r_1^{1+2s}+r_2^{1+2s}}|} =: c_0. \]
Let $v_0 = w_0 / (\tau_2 - \tau_1)$. We have $|v_0| \leq |w_0| / (t_2-t_1) \leq r_3/2$.

Since the right hand side $H_{v_0} \in L^2([0,T] \times B_{r_3^{1+2s}})$, in particular, for almost all $x \in S_1 \cap (S_2-w_0)$, the function $t \mapsto H_{v_0}(t,x+(t-\tau_1) v_0)$ is in $L^2(\tau_1,\tau_2)$. 

Because of the transport equation that $f_\infty$ satisfies in $Q^3$, we have
\[ \frac{\dd }{\dd t} f_\infty(t,x+(t-\tau_1) v_0) \geq H_{v_0}(t,x+(t-\tau_1) v_0).\]
In particular, for almost every $x \in B_{r_1}$, there is a constant $C(x)>0$ so that
\[ f_\infty(\tilde t_2,x+(\tilde t_2-\tau_1) v_0) - f_\infty(\tilde t_1,x+(\tilde t_1-\tau_1) v_0) \geq - C(x) (\tilde t_2 - \tilde t_1)^{1/2}, \qquad \text{ for any } t_1 < \tilde t_1 < \tilde t_2 < t_2.\]
However, since $f_\infty(t,x+(t-\tau_1) v_0)$ only takes the values $0$ and $1$, and $f_\infty(\tau_1,x)=1$ for every $x \in S_1 \cap (S_2-w_0)$, then $f_\infty(t,x+(t-\tau_1) v_0)=1$ for every $x \in S_1 \cap (S_2-w_0)$ and $t \in [\tau_1 , T]$.

We arrive to a contradiction since $f_\infty(\tau_2,x+(\tau_2 - \tau_1) v_0)=f_\infty(\tau_2,x+w_0)=0$ for every $x \in S_1 \cap (S_2-w_0)$.
This achieves the proof. \end{proof}
\begin{lemma} [Intermediate sets for local super-solutions] \label{l:intermediate-set-local}
Let $s \in [1/2,1)$. Let $r_1$, $r_2$, $r_3$, $r_4$, $t_1$, $t_2$, $T$, $Q^1$, $Q^2$, $Q^3$ and $Q^4$ be like in Lemma \ref{l:intermediate-set-global}.
Let $f : [0,T] \times B_{r_4} \times \R^d \to [0,+\infty)$. Assume $f$ is a supersolution of
\[ f_t + v \cdot \nabla_x f - \Lv f \geq 0 \qquad \text{in } Q^4.\]
For every pair of positive numbers $\delta_1,\delta_2$, there exists $\theta>0$ and $\mu>0$ so that whenever 
\[ |\{f\geq 1\} \cap Q^1| \geq \delta_1 \qquad \text{and} \qquad |\{f \leq \theta \} \cap Q^2| \geq \delta_2,\]
then
\[ |\{\theta < f < 1\} \cap Q^3| \geq \mu.\] 
Here, the constants $\theta$ and $\mu$ depend on $\delta_1$, $\delta_2$, $t_1$, $t_2$, $T$, $r_1$, $r_2$, $r_3$, $r_4$, $\lambda$, $\Lambda$, $s$ and $d$.
\end{lemma}
\begin{proof}
By replacing $f$ with $\min(f,1)$ (see Lemma~\ref{lem:convex} in Appendix), we can assume that $0 \leq f \leq 1$ everywhere.

Let $\rho>0$ so that $2 \rho > r_4 -r_3$. 

Applying Lemma \ref{lem:caccio} to $1-f$, we obtain that $f \in L^2([0,T] \times B_{(r_4 - \rho)^{1+2s}}, H^s(B_{r_4-\rho}))$, with
\[ \int_0^T \int_{B_{r_4 - \rho}} \|f\|_{H^s(B_{r_4 - \rho} ) }^2 \dd x \dd t \leq C,\]
for some constant $C$ depending only on $r_4$, $\rho$, $d$, $\lambda$, $\Lambda$ and $s$.

Let $\varphi : \R^d \times \R^d \to [0,1]$ be a smooth bump function supported in $B_{(r_4-2\rho)^{1+2s}} \times B_{r_4-2\rho}$ and such that $\varphi=1$ in $B_{r_3^{1+2s}} \times B_{r_3}$. We now have
\begin{equation} \label{e:ivl1}
 \int_0^T \int_{\R^d} \|\varphi f\|_{H^s(\R^d)}^2 \dd x \dd t \leq C.
\end{equation}

From a direct computation, we get
\[ [\partial_t + v \cdot \nabla_x - \Lv] (\varphi f) \geq (v \cdot \nabla_x \varphi) f - \left( \Lv (\varphi f) - \varphi \Lv f \right) \qquad \text{in } [0,T] \times \R^d \times \R^d. \]
The term $(v \cdot \nabla_x \varphi) f$ is  bounded by one, and supported in $B_{(r_4-2\rho)^{1+2s}} \times B_{r_4-2\rho}$. The second term is a commutator, which is also bounded in $L^2([0,T]\times \R^d, H^{-s}(\R^d))$ because of Lemmas \ref{l:commutator-s<1/2} and \ref{l:commutator-s>1/2}. Let
\[ h_0 := (v \cdot \nabla_x \varphi) f - \left( \Lv [\varphi f] - \varphi \Lv f \right) \in L^2([0,T] \times  \R^d, H^{-s}(\R^d) ). \]
Now, we rewrite the equation for $\varphi f$ as a fractional Kolmogorov equation
\[  [\partial_t + v \cdot \nabla_x + (-\Delta)^s_v] (\varphi f) \geq h_0 + (-\Delta)^s_v (\varphi f) + \Lv (\varphi f).\]
Because of \eqref{e:ivl1}, there is a function $h_1 \in L^2([0,T] \times  \R^d \times \R^d)$ such that $(-\Delta)^s_v (\varphi f) = (-\Delta)^{s/2}_v h_1$.

Also because of \eqref{e:ivl1} and applying Corollary~\ref{c:upperbound}, $\Lv (\varphi f)$ belongs to $H^{-s}(\R^d)$.

Summarizing, $(\varphi f)$ is a supersolution to a fractional Kolmogorov equation with a right hand side in $L^2([0,T] \times \R^d, H^{-s}(\R^d) )$,
\[  [\partial_t + v \cdot \nabla_x + (-\Delta)^s_v] (\varphi f) \geq h_0 + (-\Delta)^{s/2}_v h_1 + \Lv(\varphi f) \qquad \text{in } [0,T] \times \R^d \times \R^d. \]

We finish the proof applying Lemma \ref{l:intermediate-set-global} with $r_4 - 2\rho$ instead of $r_4$.
\end{proof}
\begin{lemma} [Propagation in
  measure] \label{l:propagation-in-measure} Under the same assumptions
  as in Lemma \ref{l:intermediate-set-local}, For every pair of
  positive numbers $\delta_1,\delta_2$, there exists $\theta>0$ so
  that whenever
\[ |\{f\geq 1\} \cap Q^1| \geq \delta_1 \qquad \text{then} \qquad |\{f \leq \theta \} \cap Q^2| < \delta_2,\]
Here, the constant $\theta$ depends on $\delta_1$, $\delta_2$, $t_1$, $t_2$, $T$, $r_1$, $r_2$, $r_3$, $r_4$, $\lambda$, $\Lambda$, $s$ and $d$.
\end{lemma}
\begin{proof}
Let $\tilde \theta $ and $\mu>0$ be the values from Lemma \ref{l:intermediate-set-local}. In this lemma we choose $\theta := \tilde \theta^{k}$ where $k$ is the smallest integer larger than $|Q^3| / \mu$.

Assume the conclusion of the lemma was not true. Then for all values of $j = 0, 1, \dots, k-1$, the function $\tilde \theta^{-j} f$ would satisfy the hypothesis of Lemma \ref{l:intermediate-set-local}. Therefore, for every $j = 0, 1, \dots, k-1$,
\[ |\{\tilde \theta^{j+1} < f < \tilde \theta^{j}\} \cap Q^3| \geq \mu.\] 
This is clearly impossible since all these are disjoint sets contained in $Q^3$, so their measures cannot add up to more than $|Q^3|$.
\end{proof}

\section{The propagation lemma}
\label{sec:growth}

We call \emph{propagation lemma}  a result that says that as soon as a
(super)solution is above a large constant in most of a cylinder, then
it is bounded from below by $1$, say, for later times. 

The difference between this propagation lemma and the first De Giorgi
lemma proved in Section \ref{sec:fdl} lies in the sets of points where
the estimates hold. Essentially, the propagation lemma is the result
of De Giorgi's first lemma, combined with a propagation of the lower
bound to later times and larger sets. This propagation is obtained
using the barrier function of section \ref{sec:barriers} when
$s \in (0,1/2)$ or the intermediate-value lemma from Section
\ref{sec:intermediate} when $s \in [1/2,1)$.

\begin{lemma}[Propagation lemma] \label{lem:growth} There exist
  $R_1>0$ (large, only depending on dimension and $s$),
  $\delta >0$ (small, universal) and $M \ge 1$ (large, universal) such
  that for $\tilde T = 2^{2s}$, if $f$ is a supersolution
\[ f_t + v \cdot \nabla_x f -\Lv f \geq 0 \qquad \text{in } [-1,\tilde T] \times B_{ R_1^{1+2s} } \times B_{R_1},\]
which is non-negative in $[-1,\tilde T] \times \R^{2d}$ and such that
\begin{equation}\label{assum:f-meas}
 |\{ f > M \} \cap Q_1| \ge (1-\delta) |Q_1| 
\end{equation}
then $f \ge 1$ in $\tilde Q$ where
$\tilde Q = [0,\tilde T] \times B_{2^{2s+1}} \times B_2$
(see Figure~\ref{fig:growth}).
\end{lemma}
\begin{figure}[ht]
\setlength{\unitlength}{1in} 
\begin{picture}(1.475 ,1.000)
\put(0,0){\includegraphics[height=1.000in]{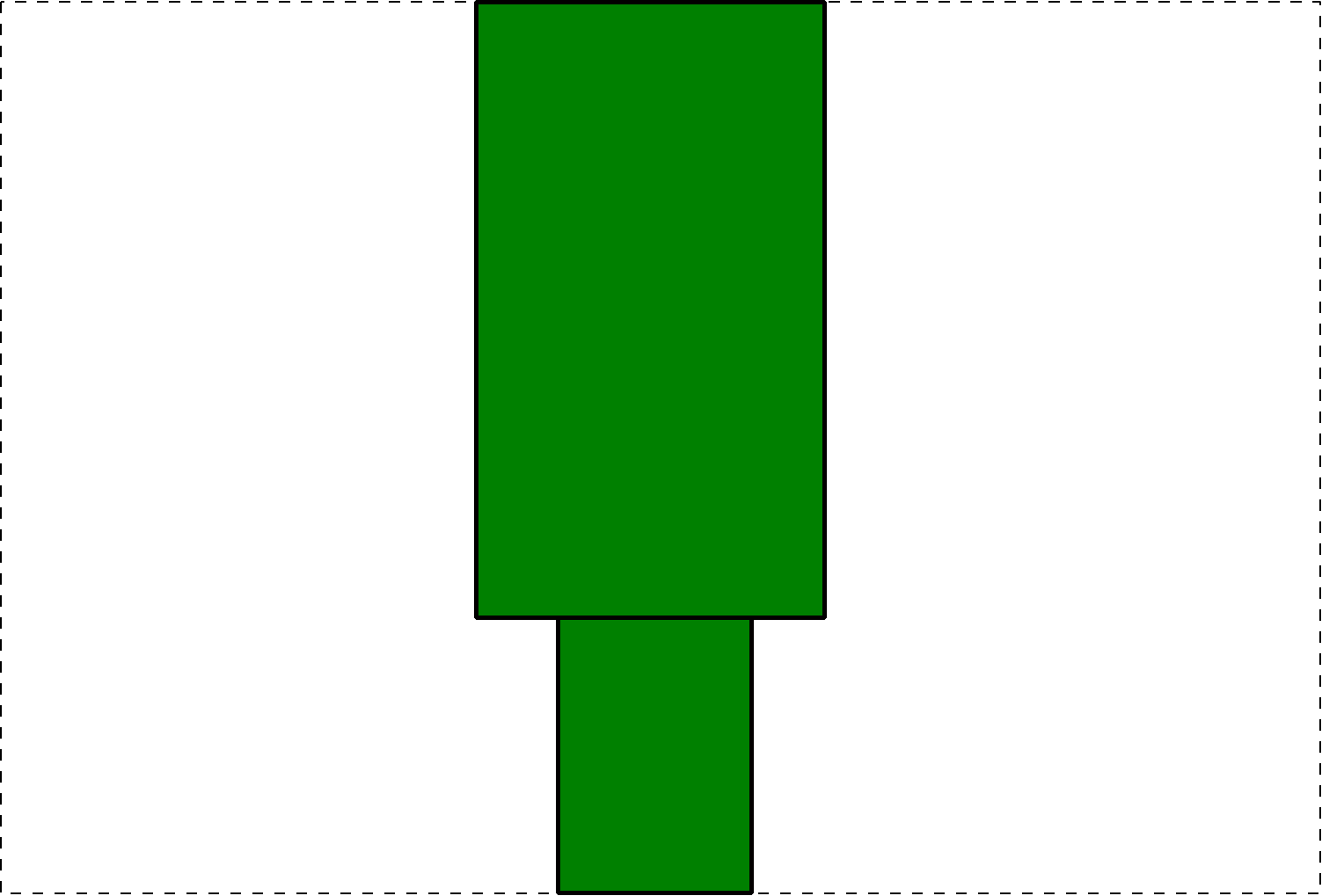}}
\put(0.634,0.651){$\tilde Q$}
\put(0.615,0.119){$Q_1$}
\end{picture}
\caption{Geometric setting of the propagation Lemma~\ref{lem:growth}.} 
\label{fig:growth}
\end{figure}
\begin{proof}
  We will prove the equivalent result that if
  $ |\{ f < 1 \} \cap Q_1| < \delta$ then $f \geq 1/M$ in $\tilde
  Q$.
  The proof will be different depending on whether $s \in (0,1/2)$ or
  $s \in [1/2,1)$.

  Let us start with the case $s \in (0,1/2)$. We combine De Giorgi's
  first lemma with a barrier function.

  We first apply Lemma \ref{lem:linfty} to $2f$, shifted in time, with
  $\hat Q = [-1,-1/2] \times B_1 \times B_1$ and
  $\tilde Q = [-3/4,-1/2] \times B_{1/2} \times B_{1/2}$. For $\delta$
  suffiently small, we obtain that $f \geq 1/2$ in $\tilde Q$. In
  particular $f(-1/2,x,v) \geq 1/2$ for all
  $(x,v) \in B_{1/2} \times B_{1/2}$.

  Let $\varphi$ be the barrier of Lemma \ref{l:barrier} with
  $T = 3/2$, $\tau = 1/2$ and $r=1/2$. Lemme \ref{l:barrier} also
  gives us the value of $R_1$. We apply the comparison principle to
  get that $f \geq \frac 12 \varphi(t+1/2,\cdot,\cdot)$ in
  $[-1/2,T] \times B_{R_1} \times B_{R_1}$ and conclude the proof. In
  this case $M = 2 / \theta$, where $\theta>0$ is the constant from
  Lemma \ref{l:barrier}.

  For the case $s \in [1/2,1)$, we combine the intermediate set lemma
  with De Giorgi's first lemma.

  We apply Lemma \ref{l:propagation-in-measure} to $f(t-1,x,v)$, with
  $r_1 = 1$, $r_2 = 3$, $r_3=4$, $r_4 = R_1 = 5$, $t_1 = 1/2$,
  $t_2 = 3/4$, $T = \tilde T+1$, arbitrary $\delta_1 = \delta>0$ and
  $\delta_2>0$ sufficiently small. We obtain that there is a
  $\theta_1>0$ so that
  \[ \left \vert \{f > \theta_1\} \cap ([-1/4,\tilde T] \times
    B_{3^{1+2s}} \times B_3) \right\vert < \delta_2.\]
  Then we apply Lemma \ref{lem:linfty} to $2f / \theta_1$ (again
  shifted in time) with
  $\hat Q = [-1/4,\tilde T] \times B_{3^{1+2s}} \times B_3$ and
  $\tilde Q = [0,\tilde T] \times B_{2^{1+2s}} \times B_2$. This
  concludes the proof.
\end{proof}

The propagation lemma implies the following corollaries.
\begin{cor}[Stacked propagation]\label{cor:stacked}
  Let $R_1$ and $\delta$ be the constants from Lemma \ref{lem:growth}.
  Let $k \ge 1$, $T_k =\sum_{i=1}^k 2^{2si}$ and $R_k = 2^k R_1$.  If
  $f$ is a supersolution of \eqref{e:main} with $h=0$ in
  $[-1,T_k] \times B_{R_k^{1+2s}} \times B_{R_k}$ and
\[| \{ f > M^k \} \cap Q_1| > (1-\delta) | Q_1|,\]
 then $f \geq 1$  in $Q[k] := [T_{k-1},T_k] \times B_{2^{(1+2s)k} } \times B_{2^k}$.
\end{cor}
\begin{figure}[ht]
\setlength{\unitlength}{1in} 
\setlength{\unitlength}{1in} 
\begin{picture}(3.679 ,1.500)
\put(0,0){\includegraphics[height=1.500in]{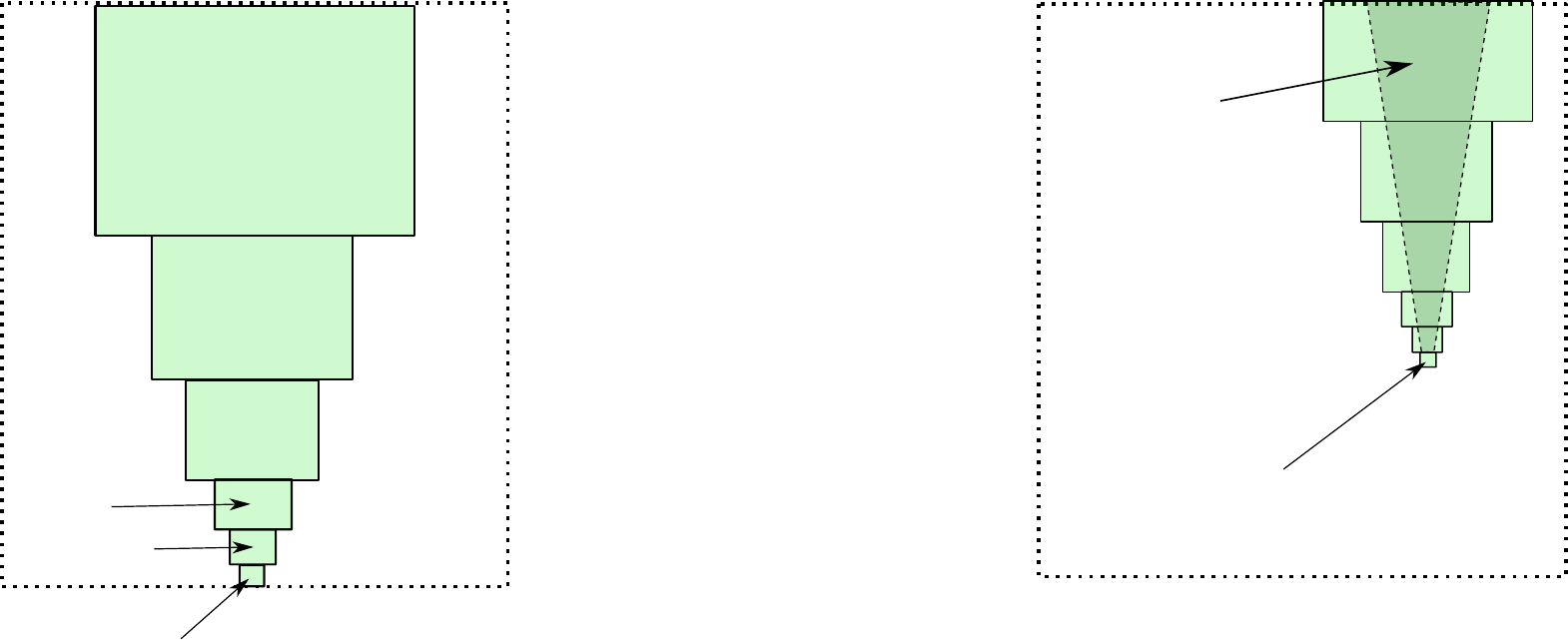}}
\put(2.171,1.205){The set $S$}
\put(2.431,0.289){$Q_r(t_0,x_0,v_0)$}
\put(0.290,-0.087){$Q_1$}
\put(0.104,0.142){$Q[1]$}
\put(-0.004,0.302){$Q[2]$}
\put(0.466,0.443){$Q[3]$}
\put(0.444,0.750){$Q[4]$}
\put(0.439,1.176){$Q[5]$}
\put(3.402,0.199){$Q_1$}
\end{picture}
\caption{Geometric setting of Corollaries \ref{cor:stacked} and \ref{cor:smoothpropagation}.}
\label{fig:stacked}
\end{figure}
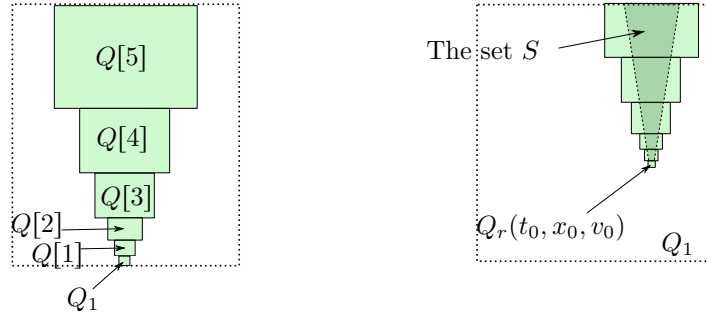

\begin{proof}
This is simply an iteration of Lemma \ref{lem:growth}. Indeed, getting $f \ge 1$
in $\tilde Q$ implies that
$|\{ \tilde f > M \} \cap \tilde Q| > (1-\delta) |\tilde Q|$ where
$\tilde f = M f$.  Choosing $\tilde Q$ as the new cylinder $Q$ in the
basic propagation lemma yields that $f$ is bounded from below by $M^{-1}$
in a new cylinder. Iterating this estimate, we get the corollary.
\end{proof}
\begin{cor}[Propagation of minima] \label{cor:smoothpropagation} Let $R_1$ and $\delta$ as in
  Lemma \ref{lem:growth}. Let $f$ be a supersolution of \eqref{e:main}
  with $h=0$ in $Q = [-1,0] \times B_{R_1^{1+2s}} \times B_{R_1}$. Let
  $Q_r(t_0,x_0,v_0) \subset Q_1$ such that
\[| \{ f > A \} \cap Q_r(t_0,x_0,v_0) | > (1-\delta) |Q_r|.\]
Then, there exists some $p>0$ and $c>0$ so that
\[ f(t,x,v) \gtrsim A \left( 1 + \frac{t-t_0}{r^{2s}} \right)^{-p}, \]
whenever $(t,x,v)$ belongs to the set
\[ \begin{aligned}
S = S(t_0,x_0,v_0):= \bigg\{ (t,x,v)  : t>t_0,  \, & |x-x_0 - (t-t_0) v_0| < \left( (1-2^{-2s})(t-t_0)+r^{2s} \right)^{1+\frac 1 {2s}}, \\
& \hspace{12mm}    \text{ and } |v-v_0| < \left( (1-2^{-2s})(t-t_0)+r^{2s} \right)^{\frac 1 {2s}}  \bigg\} .
\end{aligned}
\]
\end{cor}
\begin{proof}
Let $t_k =t_0 + \sum_{i=1}^k (2^ir)^{2s}$, $r_k = (2^k r)$ and
\[ \tilde Q[k] := Q_{r_k}(t_k,x_0+t_k v_0,v_0).\]
The change of variables
$(t,x,v) \mapsto (r^{2s}(t-t_0), r^{1+2s}(x-x_0 - (t-t_0) v_0), r
(v-v_0))$,
which preserves the equation, transforms the cylinder $Q_1$ into
$Q_r(t_0,x_0,v_0)$ and the cylinders $Q[k]$ of Corollary
\ref{cor:stacked} into $\tilde Q[k]$.

We can easily check that $S \subset \bigcup \tilde Q[k]$. Corollary
\ref{cor:stacked} tells us (after the change of variables above) that
$f \geq A / M^k$ in $\tilde Q[k]$. Observe that
$(t-t_0+r^{2s}) \approx (2^k r)^{2s}$ in $\tilde Q[k]$, therefore
\[ f(t,x,v) \geq A M^{-k} \gtrsim A \left( \frac{t-t_0+r^{2s}}{r^{2s}} \right)^{-p},\]
where $p = \frac{\log(M)}{\log(2^{2s})}$.
\end{proof}
\begin{remark}
  It is possible that in the proof of Corollary
  \ref{cor:smoothpropagation} some cylinder $\tilde Q[k]$ extends pass
  the time $t=0$ and thus it is not strictly contained in $Q$. This is
  not a problem since we are dealing with a parabolic equation and
  future values of $f$ do not affect earlier values. Indeed, we can
  readily verify that Lemma \ref{lem:growth} also holds for any value
  of $\tilde T \in (0,2^{2s})$. The only thing that matters is that
  $R_1$ is sufficiently large.
\end{remark}

\section{The ink-spots theorem for slanted cylinders}
\label{sec:ink-spot}

This section is dedicated to the statement and proof of a theorem
involving a covering argument in the flavor of Krylov-Safonov growing
ink spots theorem, or the Calder\'on-Zygmund decomposition. Such a
theorem is used in the proof of the weak Harnack inequality.  
The statement of the theorem involves stacked (and slanted) cylinders: 
\begin{equation} \label{e:Qdelayed}
  \bar Q^m(z_0,r) = \{(t,x,v) : 0 < t-t_0 \leq m r^{2s},  |v-v_0|<r,  
  |x-x_0-(t-t_0) v_0| < (m+2) r^{1+2s}\}
\end{equation}
(see Figure~\ref{fig:slanted}).
\begin{figure}[ht]
\includegraphics[height=5cm]{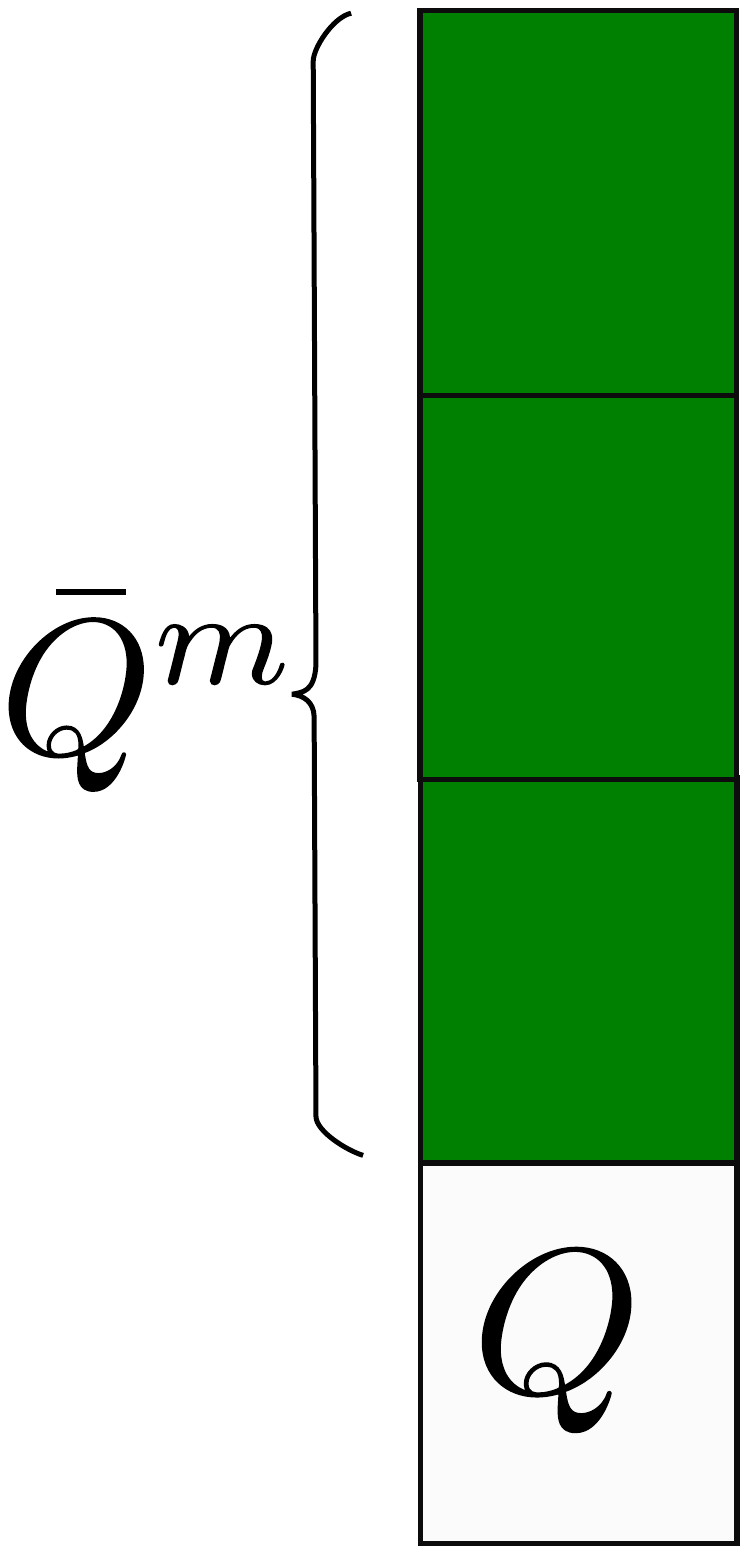}
\caption{Stacked cylinders}
\label{fig:slanted}
\end{figure}

The cylinder $\bar Q^m$ is a delayed version of $Q$. It starts
immediately at the end of $Q$. Its duration in time is $m$ times as
long as $Q$. Its radius in $x$ is $(m+2)$ times the radius in $Q$. It
shares the same values of velocities $v$ as $Q$.

\begin{thm}[The ink-spots theorem] \label{t:inkspots-slanted} Let
  $E \subset F$ be two bounded {measurable} sets. We
  make the following assumption for some constant $\mu \in (0,1)$.
\begin{itemize}
\item $E \subset Q_1$ and $|E| < (1-\mu) |Q_1|$.
\item If any cylinder $Q \subset Q_1$ {such that $\bar Q^m \subset Q_1$} satisfies
  $|Q \cap E| \ge (1-\mu) |Q|$, then $\bar Q^m \subset F$.
\end{itemize}
Then $|E| \leq \frac{m+1}m (1- c\mu) |F|$ for some constant $c \in (0,1)$
depending on $s$ and dimension only.
\end{thm}
There is no chance to adapt the Calder\'on-Zygmund decomposition to
this context. It would require that we split a larger piece into
smaller pieces of the same type. Even if we replace balls with cubes,
the different slopes, depending on the center velocities, make this
tiling condition impossible.

What we do is a variation of the growing ink-spots theorem. The
original construction by Krylov and Safonov can be found (in English)
in the Appendix 1 of \cite{krylov1987book}. Here, we have one extra
dimension, $x$, which plays a different role and presents additional
difficulies. The most significant difficulty is to go from a lower
bound on the measure of the union of disjoint cylinders $Q$ (Lemma
\ref{l:inkspots-notimeshift}) to a lower bound on the measure of the
union of their delayed versions $\bar Q^m$
(Theorem~\ref{t:inkspots-slanted}). The problem is that if the center
velocities of two cubes flow towards each other, they may create extra
overlaps in their delayed versions. This is addressed essentially in
Lemma \ref{l:expansion-overlap}, using that we expand the radius in
$x$ only by a fixed factor.

The values of $x$ that belong to a slanted cylinder $Q_r(t_0,x_0,v_0)$
change for different values of $t$. They are contained in a ball with
radius $r^{1+2s}$ which flows in the direction of $v_0$ and is shifted
a total distance $|v_0| r^{2s}$ from the initial to the end time. For
small values of $r$, the lenght of this shift is an order of magnitude
larger than the radius of the ball $r^{1+2s}$. Dealing with this shift
is non-trivial, and that is the main difference between the covering
argument described in this section and the usual ink-spots theorem.

The following corollary will be used when we need to confine both $E$
and $F$ to stay within a fixed cylinder.
\begin{cor}[Ink-spots theorem with leakage] \label{c:inkspots-leakage}
  Let $E \subset F$ be two bounded {measurable} sets. We make the following
  assumption for some constant $\mu \in (0,1)$.
\begin{itemize}
\item $E \subset Q_1$.
\item If any cylinder $Q \subset Q_1$ satisfies
  $|Q \cap E| \ge (1-\mu) |Q|$, then $\bar Q^m \subset F$ and also
  $Q = Q_r(t,x,v)$ for some $r < r_0$.
\end{itemize}
Then $|E| \leq \frac {m+1}m (1- c\mu) ( |F \cap Q_1| + C m r_0^{2s} )$ for
some constants $c$ and $C$ depending on $s$ and dimension only.
\end{cor}

\subsection{Stacked cylinders and scaling}

For any factor $k$, we define the scaled cylinder $kQ_r$ by
\[ kQ_r = Q_{kr}\left(\frac{k^{2s}-1}2 r^{2s} , 0,0 \right).\]
Here, we scaled the radius $r$ by a factor $k$ and kept the same
center of the cylinder. Note that the point $(t_0,x_0,v_0)$ in
$Q_r(t_0,x_0,v_0)$ refers to the top of the cylinder, not its
center. In order to keep the center fixed, we updated the top.

Consistently with the Lie group action, we define
\begin{alignat*}{2}
 kQ_r(t_0,x_0,v_0) &:= \mathcal{T}_{(t_0,x_0,v_0)} && kQ_r , \\ 
 &= \{(t,x,v) : && -\frac{k^{2s}+1}2r^{2s} < t-t_0 , \leq \frac{k^{2s}-1}2 r^{2s}, \\
& &&|v-v_0|<kr, \\
& && |x-x_0-(t-t_0) v_0| < (kr)^{1+2s}\}.
\end{alignat*}

Note that $|k Q_r(t_0,x_0,v_0)| = k^{2(sd+s+d)} |Q_r|$.

The first version of the growing ink-spots lemma uses essentially a
variation of the Vitali covering lemma together with a generalized
Lebesgue differentiation theorem.

\subsection{A generalized Lebesgue differentiation theorem}

In \cite{is}, a generalized Lebesgue differentiation theorem was
derived for parabolic cylinders.  Here, even though we have one additional
variable ($x$), the proof is the essentially the same. It relies on an
adaptation of Vitali's covering lemma (Lemma~\ref{l:vitali} below) and
a maximal inequality (Lemma~\ref{l:maximal} below). 
\begin{thm} \label{t:lebesgue}
Let $f \in L^1 (\Omega,  \dx \otimes \dv \otimes \dt )$ where $\Omega$ is an open set of $\R^{2d+1}$. 
Then for a.e. $(t,x,v) \in \Omega$, 
\[ \lim_{r \to 0^+} \fint_{Q_r(t,x,v)} | f -f(t,x,v)|  \dx \dv \dt  =0. \]
\end{thm}
Theorem~\ref{t:lebesgue} is obtained from Lemma~\ref{l:maximal}
exactly as in \cite{is}. For the reader's convenience we will provide
below a proof of the maximal inequality.

In our setting, the cylinders $Q_r(t_0,x_0,v_0)$ are not the balls of
any metric. The important properties of cylinders are explicitly given
by the following lemma.
\begin{lemma} \label{l:scaledcyl} Let $Q_{r_0}(t_0,x_0,v_0)$ and
  $Q_{r_1}(t_1,x_1,v_1)$ be two cylinders with nonempty
  intersection. Assume that $2r_0 \geq r_1$. Then
\[ Q_{r_1}(t_1,x_1,v_1) \subset k Q_{r_0}(t_0,x_0,v_0) ,\]
for some universal constant $k$ (it depends on $s$ only).
\end{lemma}
\begin{proof}
  Since all our definitions are invariant by the action of the Lie
  group, we can assume without loss of generality that
  $(t_0,x_0,v_0) = 0$ (the general case is reduced to this applying
  $\mathcal T_{(t_0,x_0,v_0)}^{-1}$).

We need to choose the constant $k$ so that
\begin{align*}
k &\geq 5, \\
k^{2s}  &\geq 1+2 \cdot 2^{2s},\\
k^{1+2s}  &\geq 1+2 \cdot 2^{1+2s}.
\end{align*}
The first inequality implies the other two when $s \geq 1/2$. The
second inequality implies the other two when $s \leq1/2$. In
particular the third inequality is always redundant. In any case, we
pick the smallest $k$ satisfying these inequalities, which depends
only on $s$.

Let $(t_2,x_2,v_2) \in Q_{r_0} \cap Q_{r_1}(t_1,x_1,v_1)$. Let $(t,x,v) \in Q_{r_1}(t_1,x_1,v_1)$. Then all the following hold
\begin{align*}
t &\leq t_1 < t_2 + r_1^{2s} \leq (2r_0)^{2s} \leq \frac{k^{2s}-1}2 r_0^{2s}, \\
t &\geq t_1-r_1^{2s} \geq t_2 - r_1^{2s} \geq - r_0^{2s} - (2r_0)^{2s} \geq -\frac{k^{2s}+1}2 r_0^{2s}, \\
|v| &\leq |v-v_1| + |v_1-v_2| + |v_2| \leq 2 r_1 + r_0 \leq 5 r_0 \leq k r_0, \\
|x| &\leq |x-x_1-(t-t_1) v_1| + |x_2-x_1-(t_2-t_1) v_1| + |x_2| \leq 2 r_1^{2s+1} + r_0^{2s+1} \leq k^{2s+1} r_0^{2s+1}.
\end{align*}
Thus, we get that $(t,x,v) \in k Q_{r_0}$ and we conclude the proof.
\end{proof}
\begin{lemma}[Vitali] \label{l:vitali} Let $\{Q_j\}_{j \in J}$ be an arbitrary
  collection of slanted cylinders with bounded radius. Then, there
  exists a disjoint countable subcollection $\{Q_{j_i}\}$ so that
\[ \bigcup_{j \in J} Q_j = \bigcup_{i=1}^\infty kQ_{j_i} .\]
\end{lemma}
The proof of Lemma \ref{l:vitali} is the same as the classical proof
of the Vitali coverling lemma using Lemma \ref{l:scaledcyl} instead of
the fact that in any metric space $B_{r_1}(x_1) \subset 5 B_{r_0}$ if
$B_{r_1}(x_1) \cap B_{r_0} \neq \emptyset$ and $r_1 \leq 2 r_0$.

We next define the maximal function $Mf$ as follows: for $(x,v,t) \in \Omega$,
\[ Mf (t,x,v) = \sup_{Q \ni (x,v,t)} \fint_{Q \cap \Omega} |f|\]
where the supremum is taken over cylinders of the form $(y,w,s) + R Q_1$. 
\begin{lemma}[The maximal inequality] \label{l:maximal} 
For all $\lambda > 0$, 
\[ |\{ Mf > \lambda \} \cap \Omega| \le \frac{C}\lambda \|f\|_{L^1(\Omega)}.\] 
\end{lemma}
\begin{proof}
For $(x,v,t) \in \{ Mf > \lambda \} \cap \Omega$, there exists a cylinder $Q \ni (x,v,t)$ 
such that 
\[ \int_{Q \cap \Omega} |f| \ge \frac\lambda2 |Q \cap \Omega|.\]
Then $\{ Mf > \lambda \} \cap \Omega$ is covered with cylinders $\{Q_j\}$
such that the previous inequality holds. From Lemma~\ref{l:vitali}, 
there exists a disjoint countable subcollection $\{Q_{j_i}\}$ so that 
\[ \{ Mf > \lambda \} \cap \Omega \} \subset  \bigcup_{i=1}^\infty kQ_{j_i} \]
for some integer $k$ only depending on $s$. 

We now write 
\begin{align*}
\int_{\Omega} |f| \ge & \int_{\Omega \cap \cup_i Q_{j_i}} |f| = \sum_i  \int_{\Omega \cap \ Q_{j_i}} |f| \\
\ge & \sum_i \frac\lambda2 |Q_{j_i} \cap \Omega| = \frac\lambda2 |\cup_i Q_{j_i} \cap \Omega|  
=  \frac\lambda{2k^{2(d+ds+s)}} |\cup_i k Q_{j_i} \cap \Omega|  \\
\ge & \frac\lambda{2k^{2(d+ds+s)}} |\{ Mf > \lambda \} \cap \Omega|.
\end{align*}
We obtain the desired inequality with $C = 2 k^{2(d+ds+s)}$. 
\end{proof}

\subsection{Preliminary version without time delay}

\begin{lemma} \label{l:inkspots-notimeshift} Let
  $E \subset F \subset Q_1$ be two measurable
  sets. Assume that there is a constant $\mu>0$ such that
\begin{itemize}
\item $|E| < (1-\mu) |Q_1|$. 
\item if any cylinder $Q \subset Q_1$ satisfies
  $|Q \cap E| \geq (1-\mu) |Q|$, then $Q \subset F$.
\end{itemize}
Then  $|E| \leq (1-c \mu) |F|$ for some constant $c$ depending on $s$
and dimension only.
\end{lemma}
\begin{proof}
  Using the generalized Lebesgue differentiation theorem (see
  \cite{is} for an adaptation of the classical Lebesgue
  differentiation theorem~\ref{t:lebesgue}, for almost all points
  $x \in E$, there is some cylinder $Q^x$ containing $x$ such that
  $|Q^x \cap E| \geq (1-\mu) |Q^x|$. For all Lebesgue points
  $x \in E$, let us choose a maximal slanted cylinder
  $Q^x \subset Q_1$ that contains $x$ and such that
  $|Q^x \cap E| \geq (1-\mu) |Q^x|$. Here
  $Q^x = Q_{\bar r}(\bar t, \bar x, \bar v)$ for some $\bar r$,
  $\bar t$, $\bar x$ and $\bar v$. From one of the assumptions, we
  know that $Q^x \neq Q_1$ for any $x$. The other assumption tells us
  that $Q^x \subset F$.

  We claim that $|Q^x \cap E| = (1-\mu)|Q^x|$.  Otherwise, for
  $\delta>0$ small enough, there would be a $\tilde Q$ such
  that $Q^x \subset \tilde Q \subset (1+\delta) Q^x$,
  $\tilde Q \subset Q_1$ and $|\tilde Q \cap E| > (1-\mu)|\tilde Q|$,
  contradicting the maximality of the choice of $Q^x$.

  The family of cylinders $Q^x$ covers the set $E$. By Lemma
  \ref{l:vitali}, we can select a finite subcollection of non
  overlapping cylinders $Q_j := Q^{x_j}$ such that
  $E \subset \bigcup_{j=1}^n k Q_j$.

  Since $Q_j \subset F$ and $|Q_j \cap E| = (1-\mu) |Q_j|$, we have
  that $|Q_j \cap F \setminus E| = \mu |Q_j|$. Therefore
\begin{align*}
|F \setminus E| &\geq \sum_{j=1}^n |Q_j \cap F \setminus E| \\
& \geq \sum_{j=1}^n \mu |Q_j| \\
& = k^{-2(d+ds+s)} \mu \sum_{j=1}^n |k Q_j|  \geq k^{-2(d+ds+s)} \mu |E|.
\end{align*}
We thus get
\[ |F| \ge (1+\tilde{c}\mu) |E| \]
with $\tilde{c}=k^{-2(ds+d+s)}$. Since $\tilde c \mu \in (0,1)$, this implies 
\[ |E| \le (1 - c \mu)|F|\]
with $c = \tilde{c} /2$. 
\end{proof}

\subsection{Stacked cylinders and leakage}

The following lemma can be deduced from Lemma 4.29 in \cite{is} (There
is a typo in the statement in that note, we embarrassingly apologize).
\begin{lemma}\label{lem:cover-1d}
Consider a (possibly infinite) sequence of intervals $(a_j-h_k,a_j]$. Then
\[
\left|\bigcup_k (a_k, a_k+ m h_k)\right| \geq \frac{m}{m+1} \left|\bigcup_k
  (a_k-h_k, a_k] \right|.
\]
\end{lemma}
\begin{proof}
We first assume that $k=1,\dots,N$ for some finite number $N$.

Let
\[ \bigcup_{k=1}^N (a_k, a_k+ m h_k) = \bigcup_\ell I_\ell,\]
for a disjoint family of intervals
$I_\ell$. Here, each $I_\ell$ is a union of intervals of the form
$(a_i,a_i+mh_i]$. Let $a_0 - h_0$ be the minimum of $a_i-h_i$ and
$a_1+mh_1$ be the maximum of $a_i+ mh_i$ respectively, for all $i$ so
that $(a_i,a_i+mh_i] \subset I_\ell$. Naturally, we have
\[ 
|I_\ell| \geq (a_1+mh_1) - a_0 \geq \frac m {m+1} ((a_1+mh_1) -
(a_0-h_0)) \geq \frac m {m+1} \left\vert \bigcup_{\{i: (a_i,a_i+mh_i]
    \subset I_\ell\}} (a_i-h_i,a_i] \right\vert,
\] 
Therefore
\begin{align*}
\left|\bigcup_{k=1}^N (a_k, a_k+ m h_k] \right| &= |\bigcup_\ell I_\ell|  = \sum_\ell |I_\ell|, \\
&\geq \frac m {m+1} \sum_\ell \left\vert \bigcup_{\{i: (a_i,a_i+mh_i] \subset I_\ell\}} (a_i-h_i,a_i] \right\vert,\\
&\geq \frac m {m+1} \left\vert \bigcup_{i=1,\dots,N} (a_i-h_i,a_i] \right\vert.
\end{align*}
It is now enough to let $N \to \infty$ to conclude.
\end{proof}
\begin{lemma} \label{l:expansion-overlap} Let $\{Q_j\}$ be a
  collection of slanted cylinders, and $\bar Q_j^m$ be the
  corresponding delayed versions as in \eqref{e:Qdelayed}. Then
\[\abs{\bigcup_{j} \bar Q_j^m} \geq \frac m {m+1} \abs{\bigcup_{j} Q_j}.\] 
\end{lemma}
\begin{proof}
  Because of Fubini's theorem, we know that for any set
  $A \subset \R \times \R^d \times \R^d$,
  \[ |A| = \int |\{(t,x) : (t,x,v) \in A\}| \dd v.\]
  Therefore, in order to prove the lemma, it is enough to show that
  for every $v \in \R^d$,
\begin{equation} \label{e:aim_in_v}
\abs{\left\{(t,x) : (t,x,v) \in \bigcup_{j} \bar Q_j^m \right\}} \geq \frac m {m+1} \abs{\left\{(t,x) : (t,x,v) \in \bigcup_{j} Q_j\right\}}.
\end{equation}
From now on, let $v$ be any fixed $v \in \R^d$.

Any cylinder $Q_j$ corresponds to $Q_{r_j}(t_j,x_j,v_j)$ for some
choice of $r_j>0$, $t_j \in \R$ and $x_j,v_j \in \R^d$. If
$|v-v_j|<r_j$, we have
\[
\left\{(t,x) : (t,x,v) \in \bar Q_j^m \right\} = \{ (t,x) : 
0< t-t_j \leq m r_j^{2s} ,|x-x_j - (t-t_j) v_j| < (m+2) r_j^{1+2s}\}.
\]
The set in the left hand side would be empty when $|v-v_j| \geq r_j$. 

When $|v-v_j| < r_j$, we have $|(t-t_j)(v_j-v)| < m
r_j^{1+2s}$.
Therefore, we can switch $v_j$ with $v$ in the last term, changing the
right hand side, and we obtain a smaller set.
\[
\left\{(t,x) : (t,x,v) \in \bar Q_j^m \right\} \supset \{ (t,x) : 0 < t-t_j \leq m r_j^{2s},
|x-x_j - (t-t_j) v| < 2 r_j^{1+2s}\}.
\]

Let $z = x-tv$. The change of variables $(t,x) \mapsto (t,z)$ has
Jacobian one. We will estimate the measure of the points $(t,z)$ so
that $(t,z+tv)$ belongs to the set above. Thus
\begin{equation} \label{e:is-b2t}
\abs{ \left\{(t,x) : (t,x,v) \in \bigcup_{j:|v-v_j|<r_j} \bar Q_j^m \right\}} \geq \bigg\vert\bigcup_{j:|v-v_j|<r_j} \{ (t,z) : 
0 < t-t_j \leq m r_j^{2s},
|z-x_j + t_j v| < 2r_j^{1+2s}\}\bigg\vert.
\end{equation}

Let $\tilde Q_j$ be the cylinders in $\R \times \R^d$ used in the
right hand side of the inequality above,
\[
\tilde Q_j = \{ (t,z) :  0 < t-t_j \leq m r_j^{2s},
|z-x_j + t_j v| < 2r_j^{1+2s}\}.
\]

Applying Fubini's theorem again,
\begin{align*} 
  \abs{ \bigcup_{|v-v_j|<r_j} \tilde Q_j  } = \int_{\R^d} 
\abs{ \bigcup_{\substack{ \{j : |v-v_j|<r_j, \\
 |z-x_j + t_j v| < 2r_j^{1+2s}\}} } (t_j, t_j+m r_j^{2s}]} \dd z
\end{align*}
Using lemma \ref{lem:cover-1d}, 
\begin{equation} \label{e:is-t2rhs}
 \begin{aligned} 
\abs{ \bigcup_{|v-v_j|<r_j} \tilde Q_j  } &\geq \frac m {m+1}\int_{\R^d} \abs{ \bigcup_{\substack{ \{j : |v-v_j|<r_j, \\  |z-x_j + t_j v| < 2r_j^{1+2s}\} } } (t_j-r_j^{2s},t_j] } \dd z, \\
&= \frac m {m+1} \abs{ \bigcup_{j:|v-v_j|<r_j} \{(t,x) : -r_j^{2s} < t-t_j \leq 0, |x-x_j-(t-t_j)v| \leq 2r_j^{1+2s} \} }, \\
&\geq \frac m {m+1} \abs{ \bigcup_{j:|v-v_j|<r_j} \{(t,x) : -r_j^{2s} < t-t_j \leq 0, |x-x_j-(t-t_j)v_j| \leq r_j^{1+2s} \} }, \\
&= \frac m {m+1} \abs{\left\{(t,x) : (t,x,v) \in \bigcup_{j} Q_j\right\}}.
\end{aligned}
\end{equation}
For the last inequality we used that if $-r_j^{2s} < t-t_j \leq 0$,
then $(t_j-t)|v-v_j| < r_j^{1+2s}$.

Combining \eqref{e:is-b2t} with \eqref{e:is-t2rhs}, we obtain
\eqref{e:aim_in_v} and finish the proof.
\end{proof}
We can now turn to the proof of the main theorem.
\begin{proof}[Proof of Theorem~\ref{t:inkspots-slanted}]
  Let $\mathcal Q$ be the collection of all cylinders $Q \subset Q_1$
  such that $|Q \cap E| \ge (1-\mu) |Q|$. Let
  $G := \bigcup_{Q \in \mathcal Q} Q$. By construction, the sets $E$
  and $G$ satisfy the hypothesis of the Lemma
  \ref{l:inkspots-notimeshift}. Therefore $(1-c\mu) |G| \geq |E|$.

From our hypothesis $\bigcup_{Q \in \mathcal Q} \bar Q^m \subset
F$. We conclude the proof applying Lemma \ref{l:expansion-overlap},
\[|F| \geq \abs{\bigcup_{Q \in \mathcal Q} \bar Q^m} \geq \frac m {m+1} \abs{\bigcup_{Q \in \mathcal Q} Q }
= \frac m {m+1} |G|. \qedhere\] 
\end{proof}
\begin{proof}[Proof of Corollary~\ref{c:inkspots-leakage}]
Note that the condition $|E| \leq (1-\delta) |Q_1|$ is implied by the second assumption when $r_0 < 1$. Moreover, the result is trivial for $r_0 \geq 1$ choosing $C$ sufficiently large.

  Let $\mathcal Q$ be the collection of all cylinders $Q \subset Q_1$
  such that $|Q \cap E| \ge (1-\mu) |Q|$. Let
  $G := \bigcup_{Q \in \mathcal Q} \bar Q^m$. From Theorem
  \ref{t:inkspots-slanted}, we have that
  $|E| \leq \frac m {m+1} (1-c\mu) |G|$. Moreover, our hypothesis tell
  us that $G \subset F$.

  In order to conclude the corollary, we will estimate the measure
  $G \setminus Q_1$ using the fact that each of the cubes
  $Q = Q_r(t,x,v) \subset Q_1$ has radius bounded by $r_0$. Recall
  that
\[ \begin{aligned}
\bar Q^m  &= \{(\bar t,\bar x,\bar v) : 0 < \bar t-t \leq m r^{2s}, \\
& \phantom{= \{(t,x,v) :} |\bar v-v|<r, \\
& \phantom{= \{(t,x,v) :} |\bar x-x-(\bar t-t) v| < (m+2) r^{1+2s}\}.
\end{aligned} \]
Since $Q \subset Q_1$, then $t<0$. So $\bar t \leq m
r_0^{2s}$.
Moreover, $|\bar v|<1$, since the velocities in $\bar Q^m$ are the
same as in $Q$. Also, $|x|<1$, so $|\bar x| \leq 1 + m
r_0^{2s}$.
Therefore,
$\bar Q^m \subset (-1,m r_0^{2s}] \times B_{1+m r_0^{2s}} \times
B_1$. The same thing applies to $G$.
\[ G \subset (-1,m r_0^{2s}] \times B_{1+m r_0^{2s}} \times B_1.\]
Therefore
$|F \cap Q_1| \geq |G \cap Q_1| \geq |G| - |G \setminus Q_1| \geq |G|
- C m r_0^{2s}$ and we conclude the proof.
\end{proof}

\section{Proofs of the main results}
\label{sec:main}

In this section we complete the proofs of our main results. At this
point, the main tools have already been established in previous
sections. The weak Harnack inequality is proved combining the propagation
lemma (Lemma \ref{lem:growth}) with our special version of the
ink-spots theorem (Theorem \ref{t:inkspots-slanted}). The structure of
this proof is inspired by the work of Krylov and Safonov
\cite{ks} for equations in nondivergence form.

\subsection{The weak Harnack inequality}
\label{sec:whi}

\begin{proof}[Proof of Theorem~\ref{thm:whi}]
  We choose $R_1$ to be the radius given in Lemma \ref{lem:growth}. We
  choose $r_0$ sufficiently small so that the set $S(t_0,x_0,v_0)$
  from Corollary \ref{cor:smoothpropagation} contains $Q_+$ for any
  $(t_0,x_0,v_0) \in Q_-$ and $r \in (0,r_0)$.

  Replacing $f$ and $h$ with $c f$ and $ch$ where the constant $c$ is choosen as follows
\[
c = (2\inf_{Q^+}f +  2\|h\|_{L^\infty (Q_1) } )^{-1},
\]
  we reduce to the case where $\inf_{Q^+} f \le 1/2$ and
  $\|h\|_{L^\infty (Q_1)} ) \le 1/2$. 

  We can further reduce to the case $\inf_{Q^+} f \le 1$ and $h =
  0$. Indeed, if the function $f$ is a supersolution of
  \[ f_t + v \cdot \nabla_x f - \Lv f \geq -1/2,\]
  then the function $\tilde f(t,x,v) = f(t,x,v) + (t+1)/2$ is a
  nonnegative function in $[-1,0] \times \R^{2d}$ which is a
  supersolution to \eqref{e:main} with $h=0$. Moreover,
  $ \inf_{Q^+} \tilde f \le \inf_{Q^+} f + 1/2 \le 1$ and
  $f^\eps \leq \tilde f^\eps$.

\medskip

  The proof relies on the application of the propagation
  lemmas~\ref{lem:growth} and Corollary \ref{cor:smoothpropagation}.  The constants
  $M,\delta$ in the remainder of the proof are chosen so that these
  propagation lemmas can be applied.

  We are going to prove that in this case
\[ \int_{Q^-} f^\eps (t,x,v) \dv \dx \dt \le \tilde{C}_{\text{w.h.i.}} .\]
 In order to do so,
it is enough to prove that 
\begin{equation}\label{eq:Mk}
\forall k \ge 1, \qquad |\{ f > \bar M^k \} \cap Q^- | \le C_{\text{w.h.i.}} (1- \delta')^k 
\end{equation} 
for some universal constants $\bar M \ge 1$, $C_{\text{w.h.i.}} \ge 1$ and  $\delta' \in (0,1)$. 

Estimate~\eqref{eq:Mk} is proved by induction.  For $k=1$, we simply
 choose
$C_{\text{w.h.i.}}$ and $\delta'$ so that
\[
|Q^-| \le C_{\text{w.h.i.}} (1-\delta)  \quad \text{ and } \quad
\delta' \le \delta.
\]
Note that by choosing a larger constant $C_{\text{w.h.i.}}$ we can make sure the inequality holds for arbitrarily many values of $k$.

Assume that the inequality holds true up to rank $k \ge 1$ and let us
prove it for $k+1$. We want to apply
Corollary~\ref{c:inkspots-leakage} of the growing ink spots
theorem~\ref{t:inkspots-slanted} with $\mu = \delta$, some integer
$m \ge 1$ (to be fixed later, only depending on $\delta$), and 
\[\bar M = M^m\] 
where $\delta$ and $M$ given by the
propagation lemma~\ref{lem:growth}, and 
\[ E = \{ f > \bar M^{k+1} \} \cap Q^- \quad \text{ and } \quad F = \{ f >\bar M^k \} \cap Q_1. \]

The sets $E$ and $F$ are bounded and measurable and
$E \subset F \subset Q_1$. We consider a cylinder $Q=Q_r(z_0) \subset Q^-$  (in particular $r \in (0,r_0)$) such that
$|Q \cap E | > (1-\delta) |Q|$, that is to say
\begin{equation}\label{e:assum}
 | \{ f > \bar M^{k+1} \} \cap Q^- | > (1- \delta)|Q|. 
\end{equation}
We now prove that $r$ is small. Since we have $\inf_{Q^+} f \leq 1$
and $S(t_0,x_0,v_0)$ contains $Q_+$, 
Corollary~\ref{cor:smoothpropagation} yields 
\[ \bar M^{k+1} \left( 1 + \frac{1 - 2r_0^{2s}}{r^{2s}} \right)^{-p} \lesssim 1.\]
Therefore $r \lesssim \bar M^{-k/(2sp)}$. In particular $\bar Q^m \subset Q_1$ (at least for $k$ large).

Now, we want to prove that, $\bar Q^m \subset F$, that is to say, 
\begin{equation}\label{e:conc}
 \bar Q^m \subset \{ f > \bar M^k \}.
\end{equation}
This follows simply from Corollary \ref{cor:stacked}, with $k=m$ to the function
$\tilde f = \bar M^{-k} f \circ \mathcal{T}_{z_0}$.

Applying Corollary~\ref{c:inkspots-leakage} to $E$ and $F$ with $\mu = \delta$ and $r_0 = \bar M^{-k/(2sp)}$, we get 
\[
|\{ f > \bar M^{k+1} \} \cap Q^-| \le \frac{m+1}m(1-c \delta) \left\{
  | \{ f > \bar M^k \} \cap Q^- | + Cm \bar M^{-k/p} \right\}
\]
where $C= C (s,d)$.
We now use the induction hypothesis and get 
\[ |\{ f > \bar M^{k+1} \} \cap Q^-| \le \frac{m+1}m(1-c \delta) \left\{
  C_{\text{w.h.i.}} (1-\delta')^k + C m \bar M^{-k/p} \right\} .\]
Choosing $\delta'$ smaller than $\bar M^{-1/p}$, we have
\[ |\{ f > \bar M^{k+1} \} \cap Q^-| \le C_{\text{w.h.i.}} \frac{m+1}m(1-c
\delta) \left\{ 1 + C_{\text{w.h.i.}}^{-1} Cm \right\}
(1-\delta')^k. \]
We next pick $m$ large enough (depending on $\delta$) and then $C_{\text{w.h.i.}}$ large enough
(depending on $\delta$ and $m$) so that
\[ \frac{m+1}m (1-c\delta) \left\{ 1 + C_{\text{w.h.i.}}^{-1}
  C m \right\} \le 1 - (c/2)\delta. \]
Now imposing $\delta' < (c/2)\delta$, we get the desired inequality:
 \[ |\{ f > \bar M^{k+1} \} \cap Q^-| \le C_{\text{w.h.i.}} (1-\delta')^{k+1}.\]
This achieves the proof of Estimate~\eqref{eq:Mk} and of the theorem. 
\end{proof}

\subsection{The H\"older estimate}
\label{sec:holder}

In order to prove Theorem~\ref{thm:holder}, we first prove two
preparatory results, Lemma~\ref{l:weak-harnak-with-some-negative} has
the flavor of a weak Harnack inequality, but for supersolutions that
can take (controled) negative values. The lemma then implies
Corollary~\ref{c:improvement_of_oscillation} which is concerned with the
improvement of oscillation of solutions with small forcing terms.
\begin{lemma} \label{l:weak-harnak-with-some-negative}
Let $r_0$ and $R_1$ as in Theorem \ref{thm:whi} and $\rho = r_0/R_1$. Let $\tilde Q_-$ be
\[ \tilde Q_- := Q_{\rho} (-R_1^{-2s}+\rho^{2s},0,0).\]
Let $f: (-1,0] \times B_1 \times \R^d \to \R$ be a function satisfying the following assumptions.
\begin{itemize}
\item $f_t + v \cdot \nabla_x f \geq \Lv f  + h$ in $Q_1$, with $h \geq -\eps_0$;
\item For $t \in (-1,0]$, $x \in B_1$, $v \in B_2$, $f(t,x,v) \in [0,1]$;
\item For  $t \in (-1,0]$, $x \in B_1$ and $v \in \R^d \setminus B_2$, $f(t,x,v) \geq -\left(\frac{|v|}2\right)^{\alpha_0}  +1$;
\item $|\{f \geq 1/2\} \cap \tilde Q_-| \geq \frac 12 |\tilde Q_-|.$
\end{itemize}
If $\alpha_0>0$, $\eps_0>0$ and $\delta>0$ are sufficiently
small, then
\[ f \geq \delta \text{ in } Q_\rho.\]
\end{lemma}
\begin{proof}
We can assume that $h \le 0$ without loss of generality. 
  Let us first scale the function by defining
  $\tilde f(t,x,v) = f(R_1^{-2s} t, R_1^{-1-2s} x, R_1^{-1} v)$. This
  function satisfies the equation
  \[ \partial_t \tilde f + v \cdot \nabla_x \tilde f - \tilde \Lv f
  \geq -\eps_0 R_1^{-2s} \geq -\eps_0,\]
  in $Q_{R_1}$. The rescaled kernel in $\tilde \Lv$ satisfies the same
  assumptions \eqref{e:Klower}, \eqref{e:Kabove},
  \eqref{e:Kcancellation0}, \eqref{e:Knondegeneracy} if $s < 1/2$, and
  \eqref{e:Kcancellation1} if $s \geq 1/2$. Note that $Q_{R_1}$
  contains $(-1,0] \times B_{R_1^{1+2s}} \times B_{R_1}$.

 Let $\tilde f_+ = \max(\tilde f,0)$. This function satisfies the following equation in $Q_{R_1}$,
\begin{align*} 
 \partial_t \tilde f_+ + v \cdot \nabla_x \tilde f_+ &\geq \int_{\R^d} (\tilde f_+(w) - \tilde f_+(v)) K(v,w) \dd w 
+  h - \int_{|w|\ge 2R_1} \tilde f_-(w) K(v,w) \dd w, \\
&\geq \int_{\R^d} (\tilde f_+(w) - \tilde f_+(v)) K(v,w) \dd w  - 2\eps_0
\end{align*}
provided $\alpha_0$ is small.

Applying Theorem \ref{thm:whi} (the weak Harnack inequality), we get
\begin{align*} 
\tilde  f_+ &\geq \left( \int_{\tilde{Q}_-} \tilde f^\eps \right)^{1/\eps} - 2 \eps_0 && \text{in } Q_+ \\ 
 &\geq \frac 12 |\tilde{Q}_-|^{1/\eps} - 2 \eps_0 && \\ 
 &\geq \delta && \text{ for $\eps_0$ and $\delta$ sufficiently small.}
\end{align*}
Rescaling back to $f$, we finish the proof.
\end{proof}
\begin{cor} \label{c:improvement_of_oscillation}
Let $f$ be a solution of \eqref{e:main} in $Q_1$ with $|h| \leq \eps_0$. Assume that
\[ \osc_{(-1,0] \times B_1 \times B_R} f \leq \left(\frac R 2\right)^{\alpha_0} \quad \text{ for all } R \ge 2.\]
Then
\[ \osc_{Q_\rho} f \leq 1-\delta.\]
Here, $\eps_0>0$, $\delta>0$, $\alpha_0>0$ and $\rho>0$ are the same
constants as in Lemma \ref{l:weak-harnak-with-some-negative}.
\end{cor}
\begin{proof}
  Let $a = \essinf_{(-1,0] \times B_1 \times B_2} f$ and
  $b = \esssup_{(-1,0] \times B_1 \times B_2} f$. The values of
  $f(t,x,v)$ are either above or below the middle value $(a+b)/2$ in
  at last half of the points in $\tilde Q_-$. Thus, one of the following
  inequalities holds.
\[ \left\vert \left\{ f  \geq \frac{a+b}2 \right\} \cap \tilde Q_- \right\vert \geq \frac 12 |\tilde Q_-| 
\qquad \text{or} \qquad \left\vert \left\{ f  \leq \frac{a+b}2 \right\} \cap \tilde Q_- \right\vert \geq \frac 12 |\tilde Q_-|.\]
Assume the former. The opposite case would follow from the same proof upside down.

Consider the function 
\[ \bar f(t,x,v) = 1-b + f(t,x,v).\]
This choice is made so that $\esssup_{(-1,0] \times B_1 \times B_2} \bar f = 1$.

Since $\osc_{(-1,0] \times B_1 \times B_2} \bar f \leq 1$, then
$\bar f \in [0,1]$ for $(t,x,v) \in (-1,0] \times B_1 \times B_2$.

Since
$\osc_{(-1,0] \times B_1 \times B_R} \bar f \leq \left(\frac R2
\right)^{\alpha_0}$
for $R \geq 2$ and
$\esssup_{(-1,0] \times B_1 \times B_2} \bar f = 1$, then
$\bar f(t,x,v) \geq 1 - \left(\frac {|v|} 2 \right)^{\alpha_0} $ for
$t \in (-1,0]$, $x\in B_1$ and $v \in \R^d \setminus B_2$.

Thus, $\bar f$ satisfies the hypothesis of
Lemma~\ref{l:weak-harnak-with-some-negative}, $\bar f \in [\delta,1]$
in $B_\rho$, and the corollary follows.
\end{proof}
\begin{proof}[Proof of Theorem~\ref{thm:holder}]
  Without loss of generality we assume $\|f\|_{L^\infty((-1,0] \times B_{1} \times \R^d )} \leq 1$ and
  $\|h\|_{L^\infty(Q_1)} \leq \eps_0$, where $\eps_0$ is the constant from
  Lemma \ref{l:weak-harnak-with-some-negative}. Otherwise, we replace
  $f$ by
\[\tilde f(t,x,v) =  \frac 1 {\|f\|_{L^\infty((-1,0] \times B_{1} \times \R^d )} + \|h\|_{L^\infty(Q_1)} / \eps_0} f(t,x,v).\]
We want to prove that there exists some universal constant $C$ so that for all $r >0$, 
\[ \osc_{Q_r} f \leq C r^\alpha.\]
We choose $\alpha < \min( \alpha_0, \ln(1-\delta)/\ln(\rho/2))$,
where $\rho$, $\delta$ and $\alpha_0$ are the constants from Lemma
\ref{l:weak-harnak-with-some-negative}.

Let $A(r) := \osc_{Q_r} f = \esssup_{Q_r} f - \essinf_{Q_r} f$. It is
a monotone increasing function. We cannot assume a priori that $A$ is
a continuous function, but it is always left continuous. Since $|f|\le 1$,
we also have $A (r) \le 2$ for all $r>0$. Hence, we can choose $C$ large
enough so that $A(r) \le C r^{\alpha}$ for all $r \ge \rho$.  

Assume the theorem is not true, then let
\[ r_0 := \sup\{ r : A(r) > C r^\alpha\} \in (0,\rho).\]
Since $A(r)$ is left continuous, $A(r_0) \geq C r_0^\alpha$.

Let $f_0$ be the rescaled function
\[ f_0(t,x,v) = \frac 1C \left( \frac{\rho}{2r_0}\right)^\alpha
f\left( \left(\frac{r_0}\rho  \right)^{2s} t, \left(\frac{r_0}\rho 
  \right)^{2s+1} x, \frac{r_0}\rho v \right).\]

Since $A(r) \leq C r^\alpha$ for $r > r_0$, 
\[ \osc_{Q_R} f_0 \leq (R/2)^\alpha \qquad \text{ for } R > \rho.\]
In particular, since $\alpha \leq \alpha_0$, we can apply Corollary
\ref{c:improvement_of_oscillation} and obtain that
\[ \osc_{Q_\rho} f_0 \leq 1-\delta.\]
Therefore, in terms of the original function $f$,
\[ A(r_0) \leq C \left( \frac{2r_0}{\rho} \right)^\alpha (1-\delta).\]
This contradicts that $A(r_0) \geq C r_0^\alpha$ since
$\alpha < \ln(1-\delta)/\ln(\rho/2)$, and we finish the proof.
\end{proof}

\appendix
\section{New proofs of known estimates and technical lemmas}
\label{sec:appendix}

\subsection{The coercivity estimate for the Boltzmann kernel}
\label{subsec:lowerbound}

In this appendix we give a geometric proof of the following coercivity estimate for the Boltzmann equation. It says that the Boltzmann kernel satisfies the assumption \eqref{e:Klower}.

\begin{prop} \label{p:lowerbound-app}
Assume that the function $f$ satisfies the inequalities
\begin{align*}
M_1 \leq \int_{\R^d} f(v) \dd v &\leq M_0, \\
\int_{\R^d} |v|^2 f(v) \dd v &\leq E_0, \\
\int_{\R^d}  f(v) \ln f(v) \dd v &\leq H_0.
\end{align*}
Assume also that $f \ast |\cdot|^\gamma$ is bounded by some constant $K_0$. This bound is controlled by $M_0$ and $E_0$ if $\gamma \in [0,2]$, and it is an extra assumption when $\gamma<0$.

Let $g: \R^d \to \R$ be a function supported in $B_R$. Then
\[ -\int_{\R^d} \Lv g(v) g(v) \dd v \geq \lambda \|g\|_{\dot H^s}^2 - \Lambda \|g\|_{L^2}^2.\]
The constants $\lambda$ and $\Lambda$ depend only on $M_0$, $M_1$, $E_0$, $H_0$, $K_0$ the dimension $d$ and the radius $R$.

In particular, for an appropriately larger constant $\Lambda$,
\[ \int_{\R^d} Q(f,g)(v) g(v) \dd v \leq -\lambda \|g\|_{\dot H^s}^2 + \Lambda \|g\|_{L^2}^2.\]
\end{prop}

Note that the extra assumptions about the boundedness of $f \ast |\cdot|^\gamma$ comes from the usual condition for the classical cancellation Lemma to give us a bounded function. It is the same assumption as in Lemma \ref{l:Kcancellation0}.

The constant $c$ above may go to zero as $R \to \infty$ depending on the value of $\gamma$. The precise optimal rate for this can be easily deduced from the proof. We explain this in remark \ref{r:rate}

The proofs of Proposition \ref{p:lowerbound-app} that can be found in
the Boltzmann literature are done using Fourier analysis. Here, we
present a relatively elementary proof based on a direct computation
and a geometric argument in physical variables.

We define $K_f$, $Q_1(f,g) = \Lv g$ and $Q_2(f,g) = c (f\ast |\cdot|^\gamma) \, g$ as described in Section \ref{sec:boltzmann}.


In \cite{luis}, there is an estimate for $K_f$ in terms of a simplified integral expression. It says that 
\begin{equation} \label{e:Kfsimplified}
 \begin{aligned}
K_f(v,v') &\approx \left( \int_{\left\{w \cdot (v'-v) = 0 \right \}} f(v+w)\ |w|^{\gamma + \angulars + 1}  \dd w \right) |v'-v|^{-d-\angulars}. 
\end{aligned} 
\end{equation}

\subsubsection{Lower bounds for $K_f$ in a cone of directions}

We obtain a lower bound for $K_f(v,v')$ in a symmetric cone of
directions with vertex $v$. This was done in \cite{luis}. It follows
essentially from the following lemma.

\begin{lemma} \label{l:lifted-set}
Let $f :\R^N \to R$ be as in Proposition \ref{p:lowerbound-app}.
There exists an $r>0$, $\ell>0$ and $m>0$ depending on $M_1$, $E_0$ and $H_0$ such that
\[ |\{v : f(v)>\ell\} \cap B_r| \geq m \]
\end{lemma}

Combining Lemma \ref{l:lifted-set} with the expression
\eqref{e:Kfsimplified}, we deduce the following statement. It is
essentially the same as Lemma 4.8 in \cite{luis}, but with a more
detailed description of the cone of directions where the lower bound
holds.

\begin{lemma} \label{l:cone_of_directions}
Let $f :\R^N \to R$ be non-negative and
\begin{align*}
M_1 \leq \int_{\R^N} f(v) \dd v &\leq M_0, \\
\int_{\R^N} |v|^2 f(v) \dd v &\leq E_0, \\
\int_{\R^N} f(v) \ln f(v) \dd v &\leq H_0.
\end{align*}	

For any $v \in \R^d$, there exists a set of directions $A = A(v) \in\partial B_1$, so that $K_f(v,v') \geq \lambda (1+|v|)^{1+\angulars+\gamma} |v-v'|^{-d-\angulars}$ for all $v'$ so that $(v'-v)/|v'-v| \in A$.

Moreover, this set of directions $A \subset S^{d-1}$ satisfies the following properties.
\begin{itemize}
\item $A$ is symmetric: $A = -A$.
\item Any big circle in $S^{d-1}$ 
intersects $A$ on a set of (one dimensional) measure at least $c (1+|v|)^{-1}$. In particular, the $(d-1)$ dimensional measure of $A$ is at least $\mu(v) := c(1+|v|)^{-1}$.
\item $A$ is contained on a strip of width $\leq C(1+|v|)^{-1}$ around the equator perpendicular to $v$.
\end{itemize}
\end{lemma}

By a big circle, we mean a closed geodesic in $S^{d-1}$. They are the intersection of $S^{d-2}$ with any $2$-dimensional subspace.

The proof of Lemma \ref{l:cone_of_directions} is similar to the one of Lemma 4.8 in \cite{luis}. Here we have a more precise description than in that paper because we add a lower bound of the measure of the intersection of $A$ with any big circle instead of only its total measure. The proof is relatively easy to explain with a picture on the blackboard, but perhaps somewhat cumbersome to write down.

\begin{proof}
Let $F = \{v : f(v)>\ell\} \cap B_r$ be the set described in Lemma \ref{l:lifted-set}, which has measure at least $m$.

From the formula for $K_f$ given in \eqref{e:Kfsimplified}, one immediately sees that $\sigma \in A(v)$ when the hyperplane perpendicular to $\sigma$ intersects $F$ in a set of measure at least  $cm/r$, with $\lambda = c \ell m / r$.

The three properties described in the lemma are simple geometric consequences of this construction using only that the measure of $F$ is bounded below and $F \subset B_r$ for some given constant $r$. Indeed, as $\sigma$ takes all values on a big circle in $\partial B_1$, its perpendicular hyperplanes swipe the space $\R^d$ (see Figure \ref{fig:swipe}). 

\begin{figure}[ht]
\setlength{\unitlength}{1in} 
\begin{picture}(1.647 ,2.000)
\put(0,0){\includegraphics[height=2.000in]{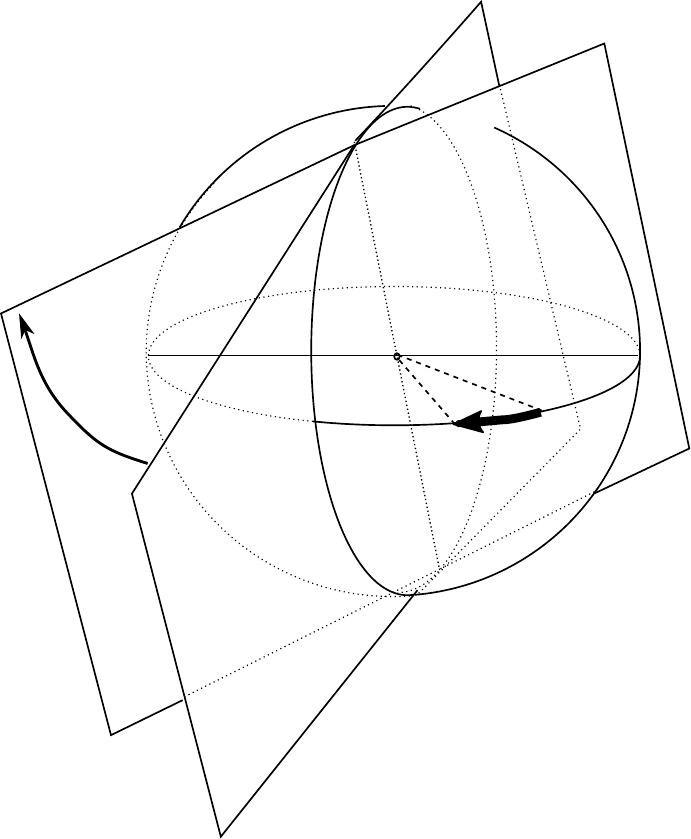}}
\put(1.21,0.9){$\sigma$}
\end{picture}
\caption{As $\sigma$ moves along a big circle in $\partial B_1$, its perpendicular planes swipe the space}
\label{fig:swipe}
\end{figure}
Because of Fubini's theorem, the points $\sigma$ on that big circle for which its perpendicular hyperplane intersects $F$ in a set of measure at least  $cm/(1+|v|)$ has to be at least of measure $cm / (1+|v|)$. 

Note that depending on the direction of the big circle, the lower bound on its intersection with $A(v)$ could be improved. For example, if the big circle is perpendicular to $v$, the measure of its intersection with $A(v)$ is bounded below independently of $v$. This fact will not be relevant to any of the computations below.
\end{proof}

Figure~\ref{fig:cone} is taken from \cite{luis} and shows all the
elements in Lemma \ref{l:cone_of_directions}.
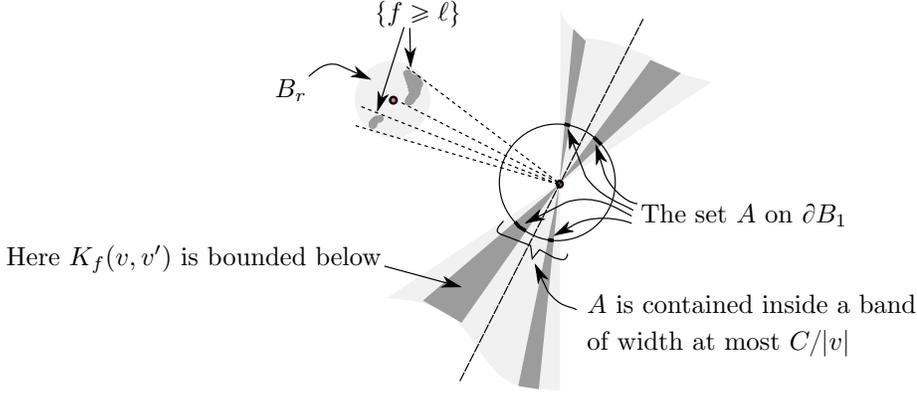
\begin{figure}
\setlength{\unitlength}{1in} 
\begin{picture}(2.11111,2.2)
\put(0,0){\includegraphics[height=2in]{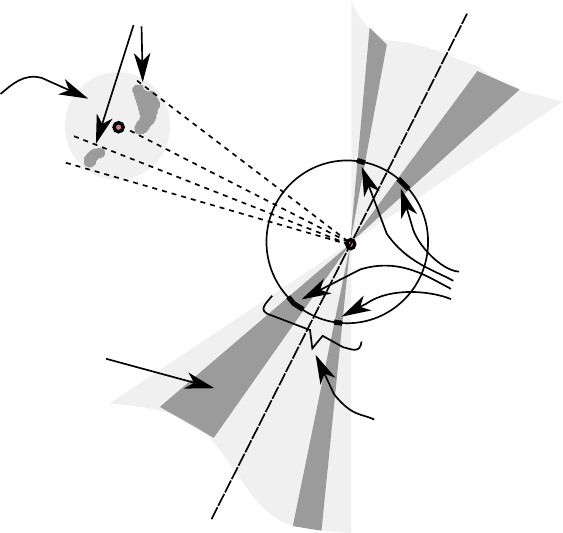}}
\put(0.35,1.94){$\{f \geq \ell\}$}
\put(-0.174,1.54){$B_r$}
\put(1.74,0.882){The set $A$ on $\partial B_1$}
\put(1.46,0.41){$A$ is contained inside a band}
\put(1.46,0.216){of width at most $C/|v|$}
\put(-1.58,0.66){Here $K_f(v,v')$ is bounded below}
\end{picture}
\caption{The geometric setting of
  Lemma~\ref{l:cone_of_directions}. The cone $\cone(v)$ is generated
  by the set $\{ f \ge l \}$.}
\label{fig:cone}
\end{figure}

We may call $\cone(v)$ the symmetric cone of values of $v'$ so that $(v'-v)/|v'-v| \in A$. In particular, the lower bound $K_f(v,v') \geq \lambda (1+|v|)^{1+\angulars+\gamma} |v-v'|^{-d-\angulars}$ holds when $v' \in \cone(v)$. 

The second item in Lemma \ref{l:cone_of_directions} says that there is a universal lower bound on the density of $\Xi(v)$ inside the cone of $v'$ given by \eqref{e:filled_cone}. This is all we will use in order to prove the coercivity estimate below.

The third item in the properties of $A$ says that for any $v' \in \cone(v)$, 
\begin{equation} \label{e:filled_cone}
 |v \cdot (v'-v) | \leq C|v'-v|.
\end{equation}
This third point plays no role in the local version of the coercivity
estimate. It is useful to understand the global coercivity estimate as
explained in Remark \ref{r:rate}.

\subsubsection{Proof of the lower bound}
\label{s:lb}

It turns out that the conditions on the kernel $K_f$ given by Lemma \ref{l:cone_of_directions} plus the cancellation lemma is all we need to obtain the bound from below of Proposition \ref{p:lowerbound-app}. For some arbitrary $R>0$, let us call $\mu := c(1+R)^{-1}$ to the lower bound on the one-dimensional measure of the intersection of $A(v)$ with big circles as in the last item of Lemma \ref{l:cone_of_directions}.

We start with a few preparatory lemmas.

\begin{lemma} \label{l:line_and_cone}
(See Figure \ref{fig:line-cone}) Let $\cone(v)$ be the cones corresponding to the sets of directions $A(v)$ as in Lemma \ref{l:cone_of_directions}. Let $\linenotell$ be a line in $\R^d$ at distance $\rho>0$ from a point $v \in \R^d$. Then, for some constants $c$ and $C$ depending only on $\mu(v)$, 
\[ |\cone(v) \cap \linenotell \cap B_{C \rho}| \geq c \rho.\]
\end{lemma}

\begin{proof} 
The projection of the line $\linenotell-v$ on the sphere $\Sp$ is half of a big circle. According to Lemma \ref{l:cone_of_directions}, the intersection of this projection with the set of directions $A = A(v)$ has (one-dimensional) measure at least $\mu(v)/2$ (recall that $A(v)$ is symmetric). At least half of these directions form an angle with $\linenotell$ of at least $\mu(v)/8$. For each of these points $z \in A(v) \subset \Sp$, there corresponds an actual intersection point in $v+\alpha z \in \cone(v) \cap \linenotell$, with $\alpha \in [\rho, 8 \mu(v)^{-1} \rho]$. Thus, the one dimensional measure of the points $v+\alpha z \in \cone(v) \cap \linenotell \cap B_{C \rho}$ is bounded below by $c \rho$, where $C = 8 \mu(v)^{-1}$ and $c = \mu(v)/4$.

\begin{figure}[ht]
\setlength{\unitlength}{1in} 
\begin{picture}(2.287 ,2.500)
\put(0,0){\includegraphics[height=2.500in]{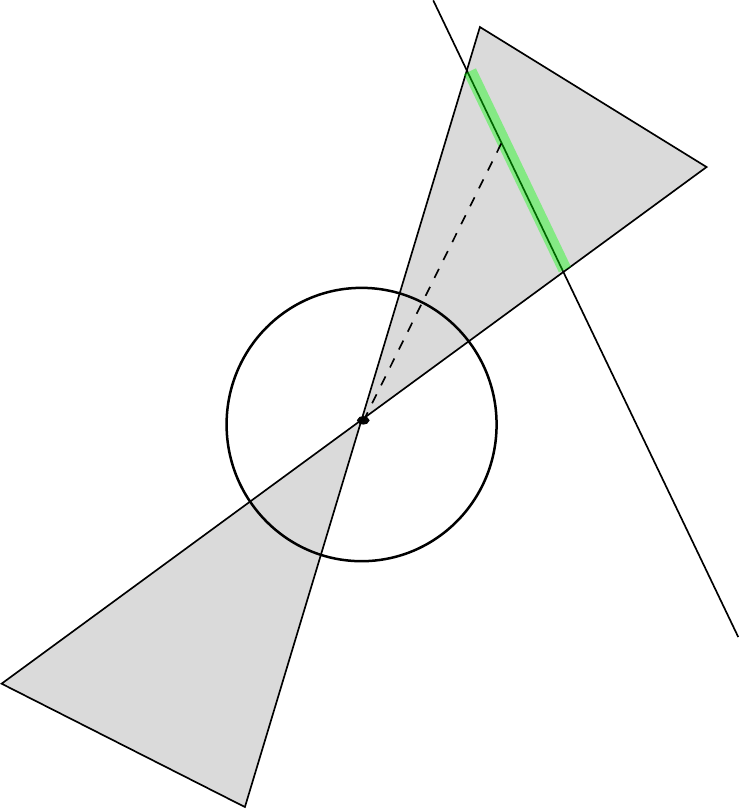}}
\put(1.156,1.138){$v$}
\put(2,1.138){$\linenotell$}
\put(0.403,0.391){$A(v)$}
\put(1.345,1.548){$v+z$}
\put(1.574,2.051){$v+\alpha z$}
\end{picture}
\caption{Intersection of a line $\linenotell$ with a cone $\cone(v)$.}
\label{fig:line-cone}
\end{figure}
\end{proof}

\begin{lemma} \label{l:cone_to_cone}
Let $\cone(v)$ be the cones corresponding to the directions $A(v)$ as in Lemma \ref{l:cone_of_directions}. Let $v_1$ and $v_2$ be two points in $\R^d$. Assume $|v_1| \geq |v_2|$. Let $\mu(v_1) \geq \mu_0$ and $\mu(v_2) \geq \mu_0$ for some $\mu_0 > 0$ .  We have
\[ \left\vert \cone(v_1) \cap \cone(v_2) \cap B_r(v_2) \right\vert \geq c |v_1-v_2|^d,\]
where $r = C |v_1-v_2|$, and $c$ and $C$ depend on $\mu_0$ only.
\end{lemma}

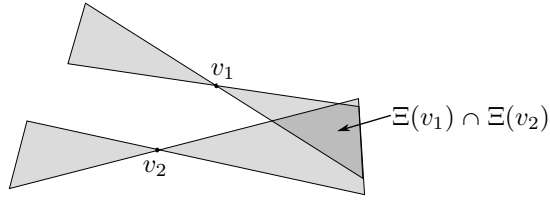
\begin{figure}[ht]
\setlength{\unitlength}{1in} 
\begin{picture}(2.000 ,1.009)
\put(0,0){\includegraphics[width=2.000in]{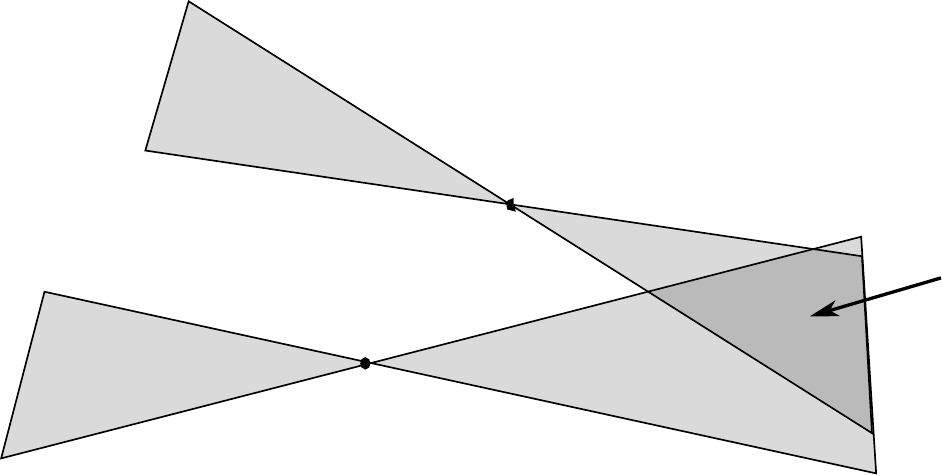}}
\put(2.006,0.368){$\cone(v_1) \cap \cone(v_2)$}
\put(1.056,0.627){$v_1$}
\put(0.701,0.120){$v_2$}
\end{picture}
\caption{The intersection of two cones inside a ball.}
\end{figure}

\begin{proof}
The lines $\linenotell$ contained in $\cone(v_1)$ are indexed by their directions $e \in A(v_1)$. At least half of these lines form an angle $\theta(\mu)>0$ with the vector $v_2-v_1$. In particular, all such lines are at distance at least $c|v_1 - v_2|$ from $v_2$, where $c$ depends on $\mu$ and $d$ only. According to Lemma \ref{l:line_and_cone}, for all such $\linenotell$,
\[ |\cone(v_2) \cap \linenotell \cap B_r (v_2)| \geq c |v_1-v_2|.\]

Integrating over all directions $e \in A(v_1)$ we conclude the lemma.
\end{proof}

The following lemma is the main estimate in the context of integro-differential equations, which implies Proposition \ref{p:lowerbound-app} when combined with the cancellation Lemma.

\begin{lemma} \label{l:genericlowerbound}
Let $K: \R^d \times \R^d \to \R$ be a non-negative kernel like in Lemma \ref{l:cone_of_directions}. Let $\mu = \inf \{ \mu(v) : v \in B_R\}$.

Then, there is a constant $c>0$, depending only on $\mu$ , $\lambda$ , $d$ and $s$, so that for any function $g$ supported in $B_R$,
\[ \iint_{\R^d \times \R^d} (g(v) - g(v'))^2 K(v,v') \dd v \dd v' \geq c \|g\|_{\dot H^s}^2.\]
\end{lemma}

\begin{proof}
Symmetrizing the integral, we can replace $K(v,v')$ by $\frac 12 ( K(v,v') + K(v',v))$. Thus, we assume that $K$ is symmetric.

From Lemma \ref{l:cone_to_cone}, we have that for all $v_1, v_2 \in B_{2R}$, there is a constant $C_0$ (sufficiently large depending on $R$ and the various constants involving $f$) so that
\[ |\cone(v_1) \cap \cone(v_2) \cap B_{C_0|v_1-v_2|}(v_2)| \geq c |v_1-v_2|^d.\]
Note that $B_{C_0|v_1-v_2|}(v_2) \subset B_{(C_0+1)|v_1-v_2|}(v_1)$. Moreover, we can choose $c_0$ small enough so that
\[ |\cone(v_1) \cap \cone(v_2) \cap B_{C_0|v_1-v_2|}(v_2) \setminus B_{c_0|v_1-v_2|}(v_1) \setminus B_{c_0|v_1-v_2|}(v_2)| \geq c |v_1-v_2|^d.\]
For the same choice of constants $c_0$ and $C_0$, let
\[ N_r(v) := \int_{B_{C_0r} (v)\setminus B_{c_0r}(v)} |g(v) - g(w)|^2 K(v,w) \dd w.\]

Therefore, for any $v_1, v_2 \in B_{2R}$, using that $|g(v_1)-g(z)|^2 + |g(v_2)-g(z)|^2 \geq |g(v_1) -g(v_2)|^2/2$, if we let $r = |v_1-v_2|$,
\begin{align*} 
 N_r(v_1) + N_r(v_2) &\geq c \left( \int_{\cone(v_1) \cap B_{C_0r}(v_1) \setminus B_{c_0r}(v_1)} |g(v_1) - g(z)|^2 |v_1-z|^{-d-\angulars} \dd z \right. \\
&\phantom{\geq} + \left. \int_{\cone(v_2) \cap B_{C_0r}(v_2) \setminus B_{c_0r}(v_2)} |g(v_1) - g(z)|^2 |v_2-z|^{-d-\angulars} \dd z \right),\\
&\geq c \left( \int_{\cone(v_1) \cap \cone(v_2)  \cap B_{C_0r}(v_2) \setminus B_{c_0r}(v_1) \setminus B_{c_0r}(v_2) } |g(v_1) - g(v_2)|^2 r^{-d-\angulars} \dd z \right),\\
&\geq c |g(v_1) - g(v_2)|^2 |v_1-v_2|^{-\angulars}.
\end{align*}
Therefore
\begin{align*} 
 \|g\|_{H^s}^2 &\leq C \iint_{B_2 \times B_2} |g(v_1) - g(v_2)|^2 |v_1-v_2|^{-d-2s} \dd v_1 \dd v_2, \\
&\leq C \iint_{B_2 \times B_2} (N_r(v_1) + N_r(v_2)) |v_1-v_2|^{-d} \dd v_1 \dd v_2 \qquad \text{ here } r = |v_1 - v_2|,\\
&= 2 C \iint_{B_2 \times B_2} N_r(v_1) |v_1-v_2|^{-d} \dd v_1 \dd v_2, \\
&\leq C \int_{B_2} \int_{r=0}^\infty \int_{\Sp} N_r(v_1) r^{-1} \dd \sigma \dd r \dd v_1, \qquad \text{ using polar coordinates for } v_2,\\ 
               &= C \int_{B_2} \int_{\R^d} |g(v_1) - g(z)|^2 K(v_1,z) \left( \int_{C_0^{-1}|v_1-z|}^{c_0^{-1} |v_1-z|} r^{-1} \dd r \right) \dd z \dd v_1, \\
&= C \int_{B_2} \int_{\R^d} |g(v_1) - g(z)|^2 K(v_1,z) \dd z \dd v_1.
\end{align*}
This finishes the proof.
\end{proof}

Once we have Lemma \ref{l:genericlowerbound}, we can derive Proposition \ref{p:lowerbound-app} as a corollary.

\begin{proof}[Proof of Proposition \ref{p:lowerbound-app}]
  It follows from a simple computation using Lemma
  \ref{l:genericlowerbound} and Lemma \ref{l:Kcancellation0}.
\begin{align*}
- \int \Lv g(v) g(v) \dd v =& \iint_{\R^d \times \R^d} (g(v) - g(v')) g(v) K_f(v,v') \dd v' \dd v, \\
=& \frac 12 \iint_{\R^d \times \R^d} (g(v) - g(v'))^2 K_f(v,v') \dd v' \dd v \\
& +  \frac 12 \int_{\R^d} g(v)^2 \left( \int_{\R^d} (K(v',v)-K(v,v')) \dd v' \right) \dd v, \\
\geq& \; \lambda \|g\|_{H^s}^2 - \Lambda \|g\|_{L^2}^2.
\end{align*}
The first term was bounded using Lemma \ref{l:genericlowerbound} and the second term using Lemma \ref{l:Kcancellation0}.
\end{proof}

\begin{remark} \label{r:rate}
We sketch the precise asymptotics of the coercity estimate for large velocities. This computation plays no role in this paper, but it is interesting to see how the metric introduced in \cite{gressman2011global} arises naturally from the geometry described above. We only analyze the symmetric part of the bilinear form as in Lemma \ref{l:genericlowerbound}. Analyzing the full bilinear form requires another similar computation for the cancellation estimate.

For large values of $v$, the cone $A(v)$ is
  approximately of width $1/|v|$ and perpendicular to $v$. The lower
  bound in Lemma \ref{l:genericlowerbound} depends only on a lower
  bound for $\mu(v)$ in $B_R$ and the lower bound for $K(v,v')$ for
  $v' \in \cone(v)$. It is easy to see how the estimate behaves for
  large velocities from a scaling argument. Indeed, let $v_0 \in \R^d$. For every
  $v \in B_1(v_0)$, the cone $\cone(v)$ has measure
  $\mu(v) \gtrsim (1+|v_0|)^{-1}$ and it is approximately
  perpendicular to $v_0$ in the sense described above. Let $T$ be the
  linear change of variables
\[ Tv = (1+|v_0|)^{-1} P v + P^\perp v , \qquad \text{where } Pv = \frac {\langle v, v_0 \rangle}{|v_0|^2} v_0, \qquad P^\perp v = v - Pv.\]
Let $\tilde g(v) = g(T v)$. So that
\[ \iint_{\R^d \times \R^d} (g(v) - g(v'))^2 K(v,v') \dd v \dd v' = \iint_{\R^d \times \R^d} (\tilde g(v) - \tilde g(v'))^2 \tilde K(v,v') \dd v \dd v',\]
where
\[ \tilde K(v,v') = |\det(DT)|^{2} K(T v, T v') = (1+|v_0|)^{-2} K(T v, T v').\]
The point of this change of variables is to make the non-degeneracy cone $\tilde K$ bounded below in measure for all $v \in B_1(T^{-1}(v_0))$, uniformly in $v_0$, i.e. $\tilde \mu(v) \gtrsim 1$ for all $v \in B_1(T^{-1}v_0)$. Moreover, for $v'$ in this nondegeneracy cone $\tilde \cone(v)$, we have
\begin{align*} 
 \tilde K(v,v') &=(1+|v_0|)^{-2} K(Tv, Tv'),\\
&\geq \lambda (1+|v_0|)^{-1+\gamma+2s}  |T v - T v'|^{-d-2s}.
\end{align*}
Therefore, from the computation in the proof of Lemma \ref{l:genericlowerbound}, for some universal constant $r>0$, and $D_r = T(B_r(T^{-1}v_0))$, we get
\begin{align*}
 \iint_{D_1 \times D_1} (g(v) - g(v'))^2 K(v,v') \dd v \dd v' &= \iint_{B_1(T^{-1}v_0) \times B_1(T^{-1}v_0)} (\tilde g(v) - \tilde g(v'))^2 \tilde K(v,v') \dd v \dd v', \\
&\geq c (1+|v_0|)^{-2} \iint_{B_r(T^{-1}v_0) \times B_r(T^{-1}v_0)} \frac{(\tilde g(v) - \tilde g(v'))^2}{|v-v'|^{d+2s}} \dd v \dd v', \\
&= c (1+|v_0|)^{-1+\gamma+2s} \iint_{B_r(T^{-1}v_0) \times B_r(T^{-1}v_0)} \frac{|\tilde g(v) - \tilde g(v')|^2}{|v-v'|^{d+2s}} \dd v \dd v', \\
&= c (1+|v_0|)^{1+\gamma+2s} \iint_{D_r \times D_r} \frac{|g(v) - g(v')|^2}{|T^{-1}v-T^{-1}v'|^{d+2s}}  \dd v \dd v'.
\end{align*}
Note that $|T^{-1}v-T^{-1}v'|$ is equivalent to the metric $d(v,v')$ introduced in \cite{gressman2011global}. The set $D_r$ is exactly the ball of radius $r$ centered at $v_0$. Therefore, covering $\R^d$ with these balls $D_r(v_0)$ and adding up, we get
\[ \iint_{d(v,v') < 1} (g(v) - g(v'))^2 K(v,v') \dd v \dd v' \geq c (1+|v_0|)^{1+\gamma+2s} \iint_{d(v,v') < r} \frac{|g(v) - g(v')|^2}{d(v,v')^{d+2s}}  \dd v \dd v'.\]
The right hand side is the same as the norm $\| g \|_{N^{s,\gamma}}^2$ introduced in \cite{gressman2011global} minus a lower order correction corresponding to the tails of the integral.
\end{remark}

\begin{remark} \label{r:coercivity}
  It is interesting to notice that the estimate of Lemma
  \ref{l:genericlowerbound} depends only on the structure of the
  kernel described in Lemma \ref{l:cone_of_directions}. It would be
  interesting to see whether the result of Lemma
  \ref{l:genericlowerbound} holds for general kernels $K$ (not
  necessarily arising from the Boltzmann equation) under less
  restrictive conditions on the cones $A(v)$. There is an interesting conjecture mentioned in \cite{dyda2015regularity} which is related to our condition \eqref{e:Knondegeneracy}.
\end{remark}

\subsection{Technical lemmas}
\label{subsec:technical}

\subsubsection{Change of variables}

We recall here a change of variables from \cite{luis}. 
\begin{lemma}[Change of variables  {\cite[Lemma~A.1]{luis}}]\label{lem:cdv}
For any non-negative function of $(v,v_*,v',v'_*)$, 
\begin{align}
\label{cdv:v'} \int_{\R^d} \int_{\Sp} F \dd \sigma dv_* &= 2^{d-1} \int_{\R^d} \frac1{|v'-v|} 
\int_{w \perp v'-v} F \frac1{r^{d-2}} \dd w \dd v'\\
\label{cdv:v'*}
& = \int_{\R^d} \frac1{|v'_*-v|} 
\int_{w \perp v'_*-v} F \frac1{r^{d-2}} \dd w \dd v'_*.
\end{align} 
\end{lemma}
Other changes of variables were used in proofs. 
\begin{lemma}[Change of variables -- II] \label{l:change-of-variables}
Let $F : \R^d \to \R$ be any integrable function. Then, the following identities hold.
\begin{align}
\label{e:cdv1}
\int_{\partial B_r} \int_{ \{w : w \perp \sigma \} } F(w) \dd w \dd \sigma&
= \omega_{d-2} r^{d-1} \int_{\R^d} \frac{F(z)}{|z|} \dd z, \\
\label{e:cdv2}
\int_{\partial B_r} \int_{ \{w : w \perp \sigma \} } F(\sigma+w) \dd w \dd \sigma &
= \omega_{d-2} r^{d-1} \int_{\R^d \setminus B_r} F(z) \frac{(|z|^2-r^2)^{\frac{d-3}2}}{|z|^{d-2}} \dd z, \\
\label{e:cdv3}
\int_{\partial B_r} \int_{ \{w : w \perp \sigma \} } \sigma F(\sigma+w) \dd w \dd \sigma &
= \omega_{d-2} r^{d+1} \int_{\R^d \setminus B_r} z F(z) \frac{(|z|^2-r^2)^{\frac{d-3}2}}{|z|^{d}} \dd z.
\end{align}
Here the constant $\omega_{d-2}$ stands for the surface area of
$\mathcal{S}^{d-2}$. Note that the integrals on the left hand side are on
spheres and hyperplanes, thus $\dd w$ and $\dd \sigma$ stand for
differential of surface.
\end{lemma}

\subsubsection{Positive part of subsolutions}

\begin{lemma}[Positive part of subsolutions]\label{lem:convex}
  Let $f$ be a subsolution of \eqref{e:main} in a cylinder $Q$.  Then
  $f_+=\max(f,0)$ is still a subsolution of \eqref{e:main} (where $h$
  is replaced with $h \un_{f \ge 0}$) in $Q$.
\end{lemma}
\begin{proof}
Since we assume that $f_t + v\cdot \nabla_x f \in L^2$, then 
\begin{equation} \label{e:c1}
 \partial_t f_+ + v\cdot \nabla_x f_+ = \begin{cases}
f_t + v\cdot \nabla_x f & \text{where $f>0$}, \\
0 &\text{elsewhere}.
\end{cases} 
\end{equation}
The equality holds in the sense of distributions.

In order to conclude that $f_+$ is a subsolution of \eqref{e:main}, we
need to prove the following inequality in the sense of distributions.
\begin{equation} \label{e:c2}
 \Lv f_+ \geq \begin{cases}
\Lv f & \text{where $f>0$},\\
0 & \text{elsewhere}.
\end{cases} 
\end{equation}

  Let $\gamma (r) = r_+$ and let $\{\gamma_\delta\}_\delta$ be a
  smooth approximation of $\gamma$ such that $|\gamma'_\delta| \le 1$
  and $\gamma_\delta$ is convex. Let $\rho_\eps$ be an even mollifier and
  $f_\eps = \rho_\eps \ast f$. Here the mollification is done with respect to the variable $v$ only.

Since $f \in L^2([0,T],\R^d,H^s(\R^d))$, it is not hard to see that 
\[ \lim_{\delta \to 0} \lim_{\eps \to 0} \gamma_\delta (f_\eps) \rightharpoonup f_+ \qquad \text{ weakly in } L^2([0,T],\R^d,H^s(\R^d)).\]

Therefore, because of Theorem \ref{t:upperbound}, for any smooth test function $\varphi$, we have
\begin{align*} 
 \lim_{\delta \to 0} \lim_{\eps \to 0} \iiint \Lv [\gamma_\delta (f_\eps)] \varphi \dv \dx \dt  &= \lim_{\delta \to 0} \lim_{\eps \to 0} \iint \E(\gamma_\delta (f_\eps),\varphi) \dd x \dd t \\ &= \iint \E(f_+,\varphi) \dd x \dd t , \\ 
 &= \iiint \Lv [f_+] \varphi \dv \dx \dt.
\end{align*}

This proves that $\Lv [\gamma_\delta (f_\eps)]$ converges to $\Lv [f_+]$ in the sense of distributions.

Thus, in order to obtain \eqref{e:c2} and finish the proof, we need to prove that for all $\delta>0$ and $\eps>0$,
\begin{equation} \label{e:c3}
 \Lv [\gamma_\delta (f_\eps)] \geq \gamma_\delta'(f_\eps)\Lv[f_\eps]. 
\end{equation}
In fact, we can check by a direct computation that the inequality holds pointwise. Indeed,
\begin{align*}
 \Lv [\gamma_\delta (f_\eps)](v) &= \int (\gamma_\delta(f_\eps(v')) - \gamma_\delta(f_\eps(v))) K(v,v') \dd v', \\
&\geq \int \gamma_\delta'(f_\eps(v))   (f_\eps(v') - f_\eps(v)) K(v,v') \dd v' \qquad \text{ using the convexity of } \gamma_\delta, \\
&= \gamma_\delta'(f_\eps(v)) \Lv f_\eps(v).
\end{align*}

Taking the limit as $\delta \to 0$ and $\eps \to 0$ in \eqref{e:c3} we obtain \eqref{e:c2}. Combining it with \eqref{e:c1} we finish the proof.
\end{proof}
\begin{lemma} [Maximum principle] \label{l:max-pple} If $f$ is a weak
  subsolution of \eqref{e:main}, with $h = 0$, in
  $Q = (a,b] \times \Omega_x \times \Omega_v$, then
\[ \esssup_Q f \leq \esssup \big\{f(t,x,v) : t \in ([a,b] \times \overline{\Omega_x} \times \R^d) 
\setminus (a,b] \times \Omega_x \times \Omega_v\big\}.\]
\end{lemma}
\begin{proof}
  Let
  $m = \esssup \{f(t,x,v) : t \in ([a,b] \times \overline{\Omega_x}
  \times \R^d) \setminus (a,b] \times \Omega_x \times \Omega_v\}$.
  The proof follows using $(f-m)_+$ as a test function in
  \eqref{e:weak-subsolution}.
\end{proof}

 \bibliographystyle{plain}
\bibliography{kinetic-DG}

\end{document}